\documentclass[a4paper,10pt]{article}
\usepackage[a4paper,margin=3cm]{geometry}
\usepackage[T1]{fontenc}
\usepackage{amsfonts,amsthm,amssymb,amsmath}
\usepackage{amssymb}
\usepackage[pdftex]{color,graphicx}
\numberwithin{equation}{section}
\title{\huge The norm of polynomials in large random and deterministic matrices}
\author{Camille Male\footnote{Ecole Normale Sup\'erieure de Lyon, Unit\'e de Math\'ematiques pures et appliqu\'ees, UMR 5669, 46 all\'ee d'Italie, 69364 Lyon Cedex 07, France. camille.male@umpa.ens-lyon.fr}\\ \\{\it with an appendix by }\\ \\Dimitri Shlyakhtenko}

\date{}
\newtheorem{Th}{Theorem}[section]

\newtheorem{Def}[Th]{Definition}
\newtheorem{Prop}[Th]{Proposition}
\newtheorem{Lem}[Th]{Lemma}

\newtheorem{Cor}[Th]{Corollary}

\renewcommand\leq\leqslant
\renewcommand\geq\geqslant

\def\Tr{\textrm{Tr}}

\def\sp{\mathrm{Sp }}
\def\esp{\mathbb E}
\def\var{\mathbb V\textrm{ar}}
\def\etc{,\ldots ,}
\def\traceN{ ( \mathrm{id}_k \otimes \tau_N ) }
\def\trace{ ( \mathrm{id}_k \otimes \tau ) }
\def\ttrN{ ( \tau_k \otimes \tau_N ) }
\def\ttr{ ( \tau_k \otimes \tau ) }
\def\UnN{ \otimes \mathbf 1_N }
\def\Un{ \otimes \mathbf 1 }
\def\MkMn{\mathrm{M}_k(\mathbb C)   \otimes   \mathrm{M}_N(\mathbb C)}
\def\Mk{\mathrm{M}_k(\mathbb C)}
\def\MN{\mathrm{M}_N(\mathbb C) }
\def\MkA{\mathrm{M}_k(\mathbb C)\otimes \mathcal A}
\def\toN{^{(N)}}
\def\toNs{^{(N)*}}
\def\xyy{\mathbf x, \mathbf y, \mathbf y^*}
\def\XYY{\mathbf X_N, \mathbf Y_N, \mathbf Y_N^*}

\def\unun{\mathbf 1_k\otimes \mathbf 1}
\def\limN{ \underset{N\rightarrow \infty}{\longrightarrow}    }

\begin{document}
\maketitle 
 
\begin{center}
\begin{minipage}{14cm}
\begin{center}{\sc abstract:}\end{center}
 {\sl 
Let $\mathbf X_N= (X_1^{(N)} \etc X_p^{(N)})$ be a family of $N \times N$ independent, normalized random matrices from the Gaussian Unitary Ensemble. We state sufficient conditions on matrices $\mathbf Y_N  =(Y_1^{(N)} \etc Y_q^{(N)})$, possibly random but independent of $\mathbf X_N$, for which the operator norm of $P(\mathbf X_N,\mathbf Y_N, \mathbf Y_N^*)$ converges almost surely for all polynomials $P$. Limits are described by operator norms of objects from free probability theory. Taking advantage of the choice of the matrices $\mathbf Y_N$ and of the polynomials $P$, we get for a large class of matrices the ''no eigenvalues outside a neighborhood of the limiting spectrum`` phenomena. We give examples of diagonal matrices $\mathbf Y_N $ for which the convergence holds. Convergence of the operator norm is shown to hold for block matrices, even with rectangular Gaussian blocks, a situation including non-white Wishart matrices and some matrices encountered in MIMO systems.}
\end{minipage}
\end{center}

\section{Introduction and statement of result}
\noindent For a Hermitian $N \times N$ matrix $H_N$, let $\mathcal L_{H_N}$ denote its empirical eigenvalue distribution, namely
		$$\mathcal L_{H_N} =\frac 1 N \sum_{i=1}^N \delta_{\lambda_i},$$
where $\delta_\lambda$ is the Dirac mass in $\lambda$ and $\lambda_1 \etc \lambda_N$ are the eigenvalues of $H_N$. The empirical eigenvalue distribution of large dimensional random matrices has been studied with much interest for a long time. One pioneering result is Wigner's theorem \cite{Wig}, from 1958. Let $W_N$ be an $N \times N$ Wigner matrix. Then the theorem states that, under appropriate assumptions, the $n$-th moment of $\mathcal L_{W_N}$ converges in expectation to the $n$-th moment of the semicircular law as $N$ goes to infinity for any integer $n$. This result has been generalized in many directions, notably by Arnold \cite{ARN} for the almost sure convergence of the moments. The convergence of the empirical eigenvalue distribution for covariance matrices was first shown by Mar\u cenko and Pastur \cite{MP} in 1967, and has been generalized in the late 1970's and the early 1980's by many people, including Grenander and Silverstein \cite{GS}, Wachter \cite{WAC}, Jonsson \cite{JON2}, Yin and Krishnaiah \cite{YK}, Bai, Yin and Krishnaiah \cite {BYK} and Yin \cite{YIN}.
\\
\\In 1991, Voiculescu \cite{VOI5} discovered a connection between large random matrices and free probability theory. He showed the so-called asymptotic freeness theorem, which has been generalized for instance in \cite{HP,THO,VOI4}, which implies the almost sure weak convergence of the empirical eigenvalue distribution for Hermitian matrices $H_N$ of the form 
		$$H_N = P(\mathbf X_N, \mathbf Y_N, \mathbf Y_N^*),$$
where
\begin{itemize}
	\item $P$ is a fixed polynomial in $2p+q$ non commutative indeterminates,
	\item $\mathbf X_N  =(X_1^{(N)} \etc X_p^{(N)})$ is a family of independent $N\times N$ matrices of the normalized Gaussian Unitary Ensemble (GUE),
	\item $\mathbf Y_N  =(Y_1^{(N)} \etc Y_q^{(N)})$ are $N\times N$ matrices with appropriate assumptions (see Theorem \ref{VoiTh}  below).
\end{itemize}
The limiting empirical eigenvalue distribution of $H_N$ can be computed by using the notion of freeness. Recall that an $N \times N$ random matrix $X^{(N)}$ is said to be a normalized GUE matrix if it is Hermitian with entries $(X\toN_{n,m})_{1\leq n,m\leq N}$, such that the set of random variables $(  X\toN_{n,n} )_{1\leq n\leq  N},$ and $ (\sqrt 2 $Re $( X\toN_{n,m} ), \sqrt 2 $$ \mathrm{Im } \ ( X\toN_{n,m} ) \ )_{1\leq n<m \leq N }$ forms a centered Gaussian vector with covariance matrix $\frac 1 N \mathbf 1_{N^2}$. Moreover, the result of Voiculescu holds even for independent Wigner or Wishart matrices instead of GUE matrices, as it has been proved by Dykema \cite{DYK}  and Capitaine and Casalis \cite{CC} respectively.
\\
\\Currently, it is known for some random matrices, as for example Wigner and Wishart matrices, that, almost surely, the eigenvalues of the matrix belong to a small neighborhood of the limiting eigenvalue distribution for $N$ large enough. More formally, if $H_N$ is a Hermitian matrix whose empirical eigenvalue distribution converges weakly to a probability measure $\mu$ it is observed in many situations \cite{BY,YBK,BSY,BS,PS} that : for all $ \varepsilon >0$, almost surely there exists $N_0\geq 1$ such that for all $N\geq N_0$ one has
\begin{equation}\label{IntroEq1}
		\sp \big ( \  H_N \ \big) \subset \mathrm{Supp} \ \big ( \  \mu \ \big )  +(-\varepsilon, \varepsilon),
\end{equation}
where '' $\sp $ `` means the spectrum and '' $ \mathrm{Supp}  $ `` means the support.
\\
\\The convergence of the extremal eigenvalues to the edges of the spectrum of a single Wigner or Wishart matrix has been shown in the early 1980's by Geman \cite{GEM}, Juh\'asz \cite{JUH}, F\"uredi and Koml\'os \cite{FK}, Jonsson \cite{JON} and Silverstein \cite{SIL,SIL2}. In 1988, in the case of a real Wigner matrix, Bai and Yin stated in \cite{BY} necessary and sufficient conditions for the convergence in terms of the first four moments of the entries of these matrices. In the case of a Wishart matrix, the similar result is due to Yin, Bai, and Krishnaiah \cite{YBK} and Bai, Silverstein, and Yin \cite{BSY}. The case of a complex matrix has been investigated later by Bai \cite{BAI}. The phenomenon ''no eigenvalues outside (a small neighborhood of) the support of the limiting distribution`` has been shown in 1998 by Bai and Silverstein \cite{BS} for large sample covariance matrices and in 2008 by Paul and Silverstein \cite{PS} for large separable covariance matrices.
\\
\\In 2005, Haagerup and Thorbj\o rnsen \cite{HT} have shown (\ref{IntroEq1}) using operator algebra techniques for matrices $H_N=P(X_1\toN \etc X_p\toN)$, where $P$ is a polynomial in $p$ non commutative indeterminates and $X_1\toN \etc X_p\toN$ are independent, normalized $N \times N$ GUE matrices. This constitutes a real breakthrough in the context of free probability. Their method has been used by Schultz \cite{SCH} to obtain the same result for Gaussian random matrices with real or symplectic entries, and by Capitaine and Donati-Martin \cite{CD} for Wigner matrices with symmetric distribution of the entries satisfying a Poincar\'e inequality and for Wishart matrices.
\\
\\A consequence of the main result of the present article is that the phenomenon (\ref{IntroEq1}) holds in the setting considered by Voiculescu, i.e. for certain Hermitian matrices $H_N$ of the form $H_N = P(\XYY)$.

\begin{Th}[The spectrum of large Hermitian random matrices]\label{MainThBis} Let $\mathbf X_N=( X_1\toN \etc X_p\toN)$ be a family of independent, normalized GUE matrices and $\mathbf Y_N= (Y_1^{(N)}\etc Y_q^{(N)})$ be a family of $N\times N$ matrices, possibly random but independent of $\mathbf X_N$. Assume that for every Hermitian matrix $H_N$ of the form 
		$$H_N = P(\mathbf Y_N, \mathbf Y_N^*),$$
where $P$ is a polynomial in $2q$ non commutative indeterminates, we have with probability one that:
\begin{enumerate}
\item{\bf Convergence of the empirical eigenvalue distribution:} there exists a compactly supported measure $\mu$ on the real line such that the empirical eigenvalue distribution of $H_N$ converges weakly to $\mu$ as $N$ goes to infinity.
\item{\bf Convergence of the spectrum:} for any $\varepsilon>0$, almost surely there exists $N_0$ such that for all $N\geq N_0$,
\begin{equation}\label{MainThCvSp}
		\sp \big ( \  H_N \ \big) \subset \mathrm{Supp} \ \big ( \  \mu \ \big )  +(-\varepsilon, \varepsilon).
\end{equation}
\end{enumerate}
Then almost surely the convergences of the empirical eigenvalue distribution and of the spectrum also hold for all Hermitian matrices $H_N = P(\mathbf X_N, \mathbf Y_N, \mathbf Y_N^*)$, where $P$ is a polynomial in $p+2q$ non commutative indeterminates.
\end{Th}

\noindent Theorem \ref{MainThBis} is a straightforward consequence of Theorem \ref{MainTh} below, where the language of free probability is used. Moreover, Theorem \ref{MainTh} specifies Theorem \ref{MainThBis} by giving a description of the limit of the empirical eigenvalue distribution. For readers convenience, we recall some definitions (see \cite{NS} and \cite{AGZ} for details).

\begin{Def} \begin{enumerate}
\item A $^*$-probability space $(\mathcal A, .^*,\tau)$ consists of a unital $\mathbb C$-algebra $\mathcal A$ endowed with an antilinear involution $.^*$ such that $(ab)^*=b^*a^*$ for all $a,b$ in $\mathcal A$, and a state $\tau$. A state $\tau$ is a linear functional $\tau: \mathcal A \mapsto \mathbb C$ satisfying
\begin{equation}
		\tau[\mathbf 1] =1, \ \ \tau[a^* a]\geq 0 \ \forall a \in \mathcal A.
\end{equation}
The elements of $\mathcal A$ are called non commutative random variables. We will always assume that $\tau$ is a trace, i.e. that it satisfies $\tau[ab]=\tau[ba]$ for every $a,b \in \mathcal A$. The trace $\tau$ is said to be faithful when it satisfies $\tau[a^*a]=0$ only if $a=0$.

\item The non commutative law of a family $\mathbf { a}=(a_1 \etc a_p)$ of non commutative random variables is defined as the linear functional $ P  \mapsto \tau \big[ P(\mathbf a,\mathbf a^*) \ \big] $, defined on the set of polynomials in $2p$ non commutative indeterminates. The convergence in law is the pointwise convergence relative to this functional.

\item \label{DefFreeness} The families of non commutative random variables $\mathbf a_1 \etc \mathbf a_n$ are said to be free if for all $K$ in $\mathbb N$, for all non commutative polynomials $P_1 \etc P_K$
\begin{equation}\label{FreeProp}
		\tau \Big [ P_1(\mathbf a_{i_1}, \mathbf a_{i_1}^*)  \ldots  P_K(\mathbf a_{i_K},\mathbf a_{i_K}^*) \ \Big] = 0
\end{equation}
as soon as $i_1\neq i_2Ê\neq \ldots \neq i_K$ and $\tau\big[ P_k(\mathbf a_{i_k}, \mathbf a_{i_k}^*) \ \big]=0$ for $k=1, \ldots ,  K$.

\item \label{DefSemiC} A family of non commutative random variables $\mathbf x=(x_1 \etc x_p)$ is called a free semicircular system when the non commutative random variables are free, selfadjoint ($x_i=x_i^*$, $i=1 \etc p$), and for all $k$ in $\mathbb N$ and $i=1 , \ldots , p$, one has
\begin{equation}\label{SemiCTrace}
		\tau[ x_i^k] =  \int t^k d\sigma(t),
\end{equation}
with $d\sigma(t) = \frac 1 {2\pi} \sqrt{4-t^2} \ \mathbf 1_{|t|\leq2} \ dt$ the semicircle distribution.
\end{enumerate}
\end{Def}

\noindent Recall first the statement of Voiculescu's asymptotic freeness theorem.

\begin{Th}[ \cite{HP,THO,VOI3,VOI4} The asymptotic freeness of $X_1\toN \etc X_p\toN, \mathbf Y_N$] \label{VoiTh} Let $\mathbf X_N=( X_1\toN \etc X_p\toN)$ be a family of independent, normalized GUE matrices and $\mathbf Y_N= (Y_1^{(N)}\etc Y_q^{(N)})$ be a family of $N\times N$ matrices, possibly random but independent of $\mathbf X_N$. Let $\mathbf x=(x_1 \etc x_p)$ be a free semicircular system in a $^*$-probability space $(\mathcal A, .^*, \tau)$ and $\mathbf y=(y_1\etc y_q)$ in $\mathcal A^q$ be a family of non commutative random variables free from $\mathbf x$. Assume the following.
\begin{enumerate}
\item{\bf Convergence of $\mathbf Y_N$:} Almost surely, the non commutative law of $\mathbf Y_N$ in $(\MN, .^*,\tau_N)$ converges to the non commutative law of $\mathbf y$, which means that for all polynomial $P$ in $2q$ non commutative indeterminates, one has
\begin{equation}
		\tau_N \big [ P(\mathbf Y_N,  \mathbf Y_N^*) \big ]     \underset{N\rightarrow \infty}{\longrightarrow}    \tau  \big [ P( \mathbf y,\mathbf y^*) \big ],
\end{equation}
where $\tau_N$ denotes the normalized trace of $N \times N$ matrices.
\item{\bf Boundedness of the spectrum:} Almost surely, for $j=1\etc q$ one has
\begin{equation}
		\underset{N \rightarrow \infty} \limsup \ \| Y_j\toN \|<\infty,
\end{equation}
where $\| \cdot \|$ denotes the operator norm.
\end{enumerate}
Then the non commutative law of $(\mathbf X_N,\mathbf Y_N)$ in $(\MN, .^*,\tau_N)$ converges to the non commutative law of $(\mathbf x,\mathbf y)$, i.e. for all polynomial $P$ in $p+2q$ non commutative indeterminates, one has
\begin{equation}\label{EqAsF}
		\tau_N \big [ P(\mathbf X_N, \mathbf Y_N,  \mathbf Y_N^*) \big ]     \underset{N\rightarrow \infty}{\longrightarrow}    \tau  \big [ P(\mathbf x, \mathbf y,\mathbf y^*) \big ].
\end{equation}

\end{Th}
\noindent In \cite{HT} Haagerup and Thorbj\o rnsen strengthened the connection between random matrices and free probability. Limits of random matrices have now to be seen in more elaborated structure, called $\mathcal C^*$-probability space, which is endowed with a norm.

\begin{Def} A $\mathcal C^*$-probability space $(\mathcal A, .^*, \tau, \|\cdot\|)$ consists of a $^*$-probability space $(\mathcal A, .^*, \tau)$ and a norm $\|\cdot \|$ such that $(\mathcal A,.^*, \|\cdot\|)$ is a $\mathcal C^*$-algebra.
\end{Def}
\noindent By the Gelfand-Naimark-Segal construction, one can always realize $\mathcal A$ as a norm-closed $\mathcal C^*$-subalgebra of the algebra of bounded operators on a Hilbert space. Hence we can use functional calculus on $\mathcal A$. Moreover, if $\tau$ is a faithful trace, then the norm $\|\cdot \|$ is uniquely determined by the following formula (see \cite[Proposition 3.17]{NS}):
\begin{equation}\label{DefNorm}
		\| a \|= \underset{k\rightarrow \infty} {\textrm{lim} }\Big (\ \tau \big[ \ (a^*a)^k \ \big] \ \Big)^{\frac 1 {2k}}, \forall a \in \mathcal A.
\end{equation}

\noindent The main result of \cite{HT} is the following.

\begin{Th}[ \cite{HT} The strong asymptotic freeness of independent GUE matrices] \label{HTTh}Let $X_1\toN \etc X_p\toN$ be independent, normalized $N \times N$ GUE matrices and let $x_1 \etc x_p$ be a free semicircular system in a $\mathcal C^*$-probability space $(\mathcal A, .^*, \tau, \| \cdotÊ\|)$ with a faithful trace. Then almost surely, one has: for all polynomials $P$ in $p$ non commutative indeterminates, one has
\begin{equation}
	\big \| P( X_1\toN \etc X_p\toN ) \big  \|  \underset{N\rightarrow \infty}{\longrightarrow}  \|P(x_1 \etc x_p)  \|.
\end{equation}
\end{Th}

\noindent This article is mainly devoted to the following theorem which is a generalization of Theorem \ref{HTTh} in the setting of Theorem \ref{VoiTh}.

\begin{Th}[The strong asymptotic freeness of $X_1\toN \etc X_p\toN,\mathbf Y_N$]\label{MainTh} Let $\mathbf X_N=( X_1\toN \etc X_p\toN)$ be a family of independent, normalized GUE matrices and $\mathbf Y_N= (Y_1^{(N)}\etc Y_q^{(N)})$ be a family of $N\times N$ matrices, possibly random but independent of $\mathbf X_N$. Let $\mathbf x=(x_1 \etc x_p)$ and $\mathbf y=(y_1\etc y_q)$ be a family of non commutative random variables in a $\mathcal C^*$-probability space $(\mathcal A, .^*, \tau, \| \cdot \|)$ with a faithful trace, such that $\mathbf x$ is a free semicircular system free from $\mathbf y$. Assume the following.
\\{\bf Strong convergence of $\mathbf Y_N$:} Almost surely, for all polynomials $P$ in $2q$ non commutative indeterminates, one has
\begin{eqnarray}
		\tau_N\big [ P(\mathbf Y_N, \mathbf Y_N^* )  \big ] & \underset{N\rightarrow \infty}{\longrightarrow} & \tau [ P(\mathbf y, \mathbf y^*)  ],\label{Th1wCV} \\
		\big \| P(\mathbf Y_N, \mathbf Y_N^*)   \big \| & \underset{N\rightarrow \infty}{\longrightarrow} & \|P(\mathbf y, \mathbf y^*)  \|.\label{Th1sCV}
\end{eqnarray}
Then, almost surely, for all polynomials $P$ in $p+2q$ non commutative indeterminates, one has
\begin{eqnarray}
		\tau_N\big [ P( \mathbf X_N,  \mathbf Y_N, \mathbf Y_N^* )  \big ] & \underset{N\rightarrow \infty}{\longrightarrow} & \tau [ P(\mathbf x,\mathbf y, \mathbf y^*)  ],\label{MainThEqVoic} \\
		\big  \| P( \mathbf X_N,  \mathbf Y_N,  \mathbf Y_N^*)  \big \| & \underset{N\rightarrow \infty}{\longrightarrow} & \| P( \mathbf x, \mathbf y,\mathbf y^*) \| \label{MainThEq}.
\end{eqnarray}
\end{Th}
\noindent The convergence of the normalized traces stated in (\ref{MainThEqVoic}) is the content of Voiculescu's asymptotic freeness theorem and is recalled in order to give a coherent and complete statement. Theorem \ref{MainThBis} is easily deduced from Theorem \ref{MainTh} by applying Hamburger's theorem \cite{HAM} for the convergence of the measure and functional calculus for the convergence of the spectrum.
\\
\\{\bf Organization of the paper:} In Section \ref{Applications} we give applications of Theorem \ref{MainTh} which are proved in Section \ref{PartEx}. Sections \ref{PartSheme} to \ref{App2} are dedicated to the proof of Theorem \ref{MainTh}.  
\\
\\{\bf Acknowledgments:} The author would like to thank Alice Guionnet for dedicating much time for many discussions to the subjects of this paper and, along with Manjunath Krishnapur and Ofer Zeitouni, for the communication of Lemma \ref{FromStoS}. He is very much obliged to Dimitri Shlyakhtenko for his contribution to this paper. He would like to thank Benoit Collins for pointing out an error in a previous version of Corollary \ref{CorDetDiag} and giving the idea to fix it. He also likes to thank Mikael de la Salle for useful discussions.

\section{Applications}\label{Applications}
\subsection{Diagonal matrices}
\noindent The first and the simpler matrix model that may be investigated to play the role of matrices $\mathbf Y_N$ in Theorem \ref{MainTh} consists of deterministic diagonal matrices with real entries and prescribed asymptotic spectral measure.

\begin{Cor}[diagonal matrices] \label{CorDetDiag} Let $\mathbf X_N=(X_1\toN \etc X_p\toN)$ be a family of independent, normalized GUE matrices and let $\mathbf D_N=(D_1^{(N)} \etc D_q^{(N)})$ be $N\times N$ deterministic real diagonal matrices, such that for any $j=1 \etc q$,
\begin{enumerate}
\item the empirical spectral distribution of $D^{(N)}_j$ converges weakly to a compactly supported probability measure $\mu_j$,
\item the diagonal entries of $D_j\toN$ are non decreasing:
	$$D_j\toN=\ \mathrm{diag} \ \Big (\lambda_1\toN(j) \etc \lambda_N\toN(j) \ \Big),\textrm{ with } \lambda_1\toN(j) \leq \hdots \leq \lambda_N\toN(j),$$
\item for all $\varepsilon>0$, there exists $N_0$ such that for all $N\geq N_0$, for all $j=1\dots q$,
	$$ \sp \big ( \  D_j\toN \ \big) \subset  \  \mathrm{Supp}  \big ( \  \mu_j  \ \big )+ (-\varepsilon, \varepsilon).$$
\end{enumerate}
Let $v=(v_1\etc v_q)$ in $[0,1]^q$. We set $\mathbf D_N^v=\big( D_1^{(N)}(v_1) \etc D_q^{(N)}(v_q) \big)$, where for any $j=1\etc q$, one has
	$$D_j\toN(v_j) = \ \mathrm{diag} \ \Big (\lambda_{1+\lfloor v_jN \rfloor }\toN(j) \etc \lambda_{N+\lfloor v_jN \rfloor}\toN(j) \ \Big),  \textrm{ with indices modulo } N.$$
Let $\mathbf x=(x_1\etc x_p)$ and $\mathbf d^v=\big( d_1(v) \etc d_q(v) \big)$ be non commutative random variables in a $\mathcal C^*$-probability space $(\mathcal A, .^*,\tau, \|\cdot\|)$ with a faithful trace, such that
\begin{enumerate}
\item $\mathbf x$ is a free semicircular system, free from $\mathbf d^v$,
\item The variables $d_1(v)\etc d_q(v)$ commute, are selfadjoint and for all polynomials $P$ in $q$ indeterminates, one has
\begin{equation}\label{CorLawd}
			\tau [ P(\mathbf d^v) \ ] =  \int_0^1  P\Big( F_1^{-1}(u + v_1) \etc  F_q^{-1}(u + v_q) \ \Big)du.
\end{equation}
For any $j=1\dots q$, the application $F_j^{-1}$ is the generalized inverse of the cumulative distribution function $F_j: t \mapsto \mu_j\big ( \ ]-\infty, t] \big )$ of $\mu_j$ defined by: $F_j^{-1}$ is $1$-periodic and for all $u$ in $]0,1]$, $F_j^{-1}(u) = \inf\big \{ t \in\mathbb R \ \big| \ F_j(t) \geq u \big\}$.
\end{enumerate}
Then, with probability one, for all polynomials $P$ in $p+q$ non commutative indeterminates, one has
\begin{eqnarray}
			\tau_N \big [ P( \mathbf X_N,  \mathbf D_N^v) \big ] \underset{N\rightarrow \infty}\longrightarrow\tau [ P( \mathbf x, \mathbf d^v) ] \label{cor:diag1} \\
			\big  \| P( \mathbf X_N,  \mathbf D_N^v) \big  \| \underset{N\rightarrow \infty}\longrightarrow \| P( \mathbf x, \mathbf d^v) \|, \label{cor:diag2}
\end{eqnarray}
for any $v$ in $[0,1]^q$ except in a countable set.
\end{Cor}
\noindent Remark that the non commutative random variables $d_1\etc d_q$ can be realized as classical random variables, $d_j$ being $\mu_j$-distributed for $j=1\etc q$. The dependence between the random variables is trivial since Formula {\rm (\ref{CorLawd})} exhibits a deterministic coupling.
\\The convergence of the normalized trace (\ref{cor:diag1}) actually holds for any $v$. In general, the convergence (\ref{cor:diag2}) of the norm can fail: the family of matrices $\mathbf D=(D_1\toN,D_2\toN)$ where 
		$$D_1\toN = \ \mathrm {diag} \ (\mathbf 0_{\lfloor N/2\rfloor}, \mathbf 1_{N -  \lfloor N/2\rfloor}), \ D_1\toN = \ \mathrm {diag} \ (\mathbf 0_{\lfloor N/2\rfloor+1}, \mathbf 1_{N -  \lfloor N/2\rfloor-1})$$
gives a counterexample (consider their difference). Furthermore, let mention that it is clear that we always can take one of the $v_i$ to be zero.

\subsection{Non-white Wishart matrices}
Theorem \ref{MainTh} may be used to deduce the same result for some Wishart matrices as for the GUE matrices. Let $r, s_1 \etc  s_p\geq 1$ be integers. Let $\mathbf Z_N =(Z_1^{(N)} \etc Z_p^{(N)})$ be a family of independent positive definite Hermitian random matrices such that for $j=1 \etc p$ the matrix $Z_j\toN$ is of size $s_jN\times s_jN$. Let $\mathbf W_N =\mathbf W_N(\mathbf Z)   =(W_1^{(N)} \etc  W_{p}^{(N)})$ be the family of $rN\times rN$ matrices defined by: for each $j=1, \ldots, p$, $W_j^{(N)} = M_j^{(N)}\ Z_j\toN \ M_j^{(N)*} $, where $M_j^{(N)}$ is a $rN \times s_jN$ matrix whose entries are random variables,
	$$ M_j^{(N)}=(M_{n,m})_{\substack{1 \leq n \leq  rN \\ 1 \leq m \leq s_jN}  }, $$
and the random variables $(\sqrt 2 $Re $(M_{n,m}), \sqrt 2 $$ \mathrm{Im } \ (M_{n,m}) \ )_{1 \leq n \leq rN, 1 \leq m \leq  s_jN}$ form a centered Gaussian vector with covariance matrix $\frac{1}{rN} \mathbf 1_{2rs_jN^2}$. We assume that $M_1\toN \etc M_p\toN, \mathbf Z_N$ are independent. The matrices $W_1^{(N)} \etc  W_{p}^{(N)}$ are called non-white Wishart matrices, the white case occurring when the matrices $Z_j\toN$ are the identity matrices.

\begin{Cor}[Wishart matrices]\label{CorWisMa}
Let $\mathbf Y_N=(Y_1\toN\etc Y_q\toN)$ be a family of $rN\times rN$ random matrices, independent of $\mathbf Z_N$ and $\mathbf W_N$. Assume that the families of matrices $(Z_1\toN)\etc (Z_q\toN),\mathbf Y_N$ satisfy separately the assumptions of Theorem \ref{MainTh}. Then, almost surely, for all polynomials $P$ in $p+2q$ non commutative indeterminates, one has
\begin{equation}\label{MainThEqWis}
	\big  \| P( \mathbf W_N, \mathbf Y_N, \mathbf Y_N^*)\big  \|  \underset{N\rightarrow \infty}\longrightarrow \| P(\mathbf w, \mathbf y,\mathbf y^*) \|, 
\end{equation}
where $\| \cdot \|$ is given by Formula {\rm(\ref{DefNorm})} with $\tau$ a faithful trace for which the non commutative random variables $\mathbf w = (w_1 \etc w_p)$ and $\mathbf y=(y_1 \etc y_q)$ are free.
\end{Cor}

\noindent In \cite{PS}, motivated by applications in statistics and wireless communications, the authors study the global limiting behavior of the spectrum of the following matrix, referred as separable covariance matrix:
		$$C_n=\frac 1 n A_n^{1/2} X_nB_nX_n^*A_n^{1/2},$$
where $X_n$ is a $n\times m$ random matrix, $A_n^{1/2}$ is a nonnegative definite square root of the nonnegative definite $n\times n$ Hermitian matrix $A_n$ and $B_n$ is a $m \times m$ diagonal matrix with nonnegative diagonal entries. It is shown in \cite{PS} that, for $n$ large enough, almost surely the eigenvalues of $C_n$ belong in a small neighborhood of the limiting distribution under the following assumptions:
\begin{enumerate}
\item $m=m(n)$ with $c_n:=n/m \underset{n\rightarrow \infty} \longrightarrow c>0$.
\item The entries of $X_n$ are independent, identically distributed, standardized complex and with a finite fourth moment.
\item The empirical eigenvalue distribution $\mathcal L_{A_n}$ (respectively $\mathcal L_{B_n}$) of $A_n$ (respectively $B_n$) converges weakly to a compactly supported probability measure $\nu_a$ (respectively $\nu_b$) and the operator norms of $A_n$ and $B_n$ are uniformly bounded.
\item By assumptions 1,2 and 3, it is known that almost surely $\mathcal L_{C_n}$ converges weakly to a probability measure $\mu^{(c)}_{\nu_{a},\nu_{b}}$. This define a map $\Phi : (x, \nu_1,\nu_2) \mapsto \mu^{(x)}_{\nu_1,\nu_2}$ (the input $x$ is a positive real number, the inputs $\nu_1$ and $\nu_2$ are probability measures on $\mathbb R^+$). Assume that for every $\varepsilon>0$, there exists $n_0\geq 1$ such that, for all $n\geq n_0$, one has
		$$ \mathrm{Supp} \ \big( \ \mu^{(c_n)}_{\mathcal L_{A_n},\mathcal L_{B_N}} \ \big ) \subset \mathrm{Supp} \ \big( \ \mu^{(c)}_{\nu_{a},\nu_{b}} \ \big )  \ + (-\varepsilon, \varepsilon).$$
%
\end{enumerate}
Now consider the following situation, where Corollary \ref{CorWisMa} may be applied
\begin{enumerate}
\item[1'] $n=n(N)=rN$, $m=m(N)=sN$ for fixed positive integers $r$ and $s$,
\item[2'] the entries of $X_n$ are independent, identically distributed, standardized complex Gaussian,
\item[3'] the empirical eigenvalue distribution of $A_n$ (respectively ${B_n}$) converges weakly to a compactly supported probability measure,
\item[4'] for $N$ large enough, the eigenvalues of $A_n$ (respectively $B_n$) belong in a small neighborhood of its limiting distribution.
\end{enumerate}
Then we obtain by Corollary \ref{CorWisMa} that for $N$ large enough, almost surely the eigenvalues of $C_n$ belong in a small neighborhood of the limiting distribution. The advantage of our version is the replacement of assumption 4 by assumption 4'. Replacing assumptions 1' and 2' by assumptions 1 and 2 could be an interesting question.
\subsection{Block matrices}
\noindent It will be shown as a consequence of Theorem \ref{MainTh} that the convergence of norms (\ref{MainThEq}) also holds for block matrices.
\begin{Cor}[Block matrices]\label{CorBloMa}
Let $\mathbf X_N, \mathbf Y_N, \mathbf x, \mathbf y$ and $ \tau$ be as in Theorem \ref{MainTh}. Almost surely, for all positive integer $\ell$ and for all non commutative polynomials $(P_{u,v})_{1\leq u,v \leq  \ell}$, the operator norm of the $\ell N \times \ell N$ block matrix
\begin{equation}\label{NormBlock}
		\left (    \begin{array}{ccc}  P_{1,1}( \mathbf X_N, \mathbf Y_N, \mathbf Y_N^*) & \ldots & P_{1,\ell}( \mathbf X_N, \mathbf Y_N, \mathbf Y_N^*) \\
			\vdots & & \vdots\\
			P_{\ell,1}(\mathbf X_N,  \mathbf Y_N, \mathbf Y_N^*) & \ldots & P_{\ell,\ell}(\mathbf X_N,  \mathbf Y_N, \mathbf Y_N^*) \end{array} \right)
\end{equation}
converges to the norm $\| \cdot \|_{\tau_\ell \otimes \tau}$ of 
\begin{equation}
\left (    \begin{array}{ccc}  P_{1,1}( \mathbf x, \mathbf y, \mathbf y^*) & \ldots & P_{1,\ell}( \mathbf x, \mathbf y, \mathbf y^*) \\
			\vdots & & \vdots\\
			P_{\ell,1}(\mathbf x,  \mathbf y, \mathbf y^*) & \ldots & P_{\ell,\ell}(\mathbf x,  \mathbf y, \mathbf y^*)\end{array} \right),
\end{equation}
where $\| \cdot \|_{\tau_\ell \otimes \tau}$ is given by the faithful trace $\tau_\ell \otimes \tau$ defined by
\begin{equation*}
(\tau_\ell \otimes \tau)\left [ \left (    \begin{array}{ccc}  P_{1,1}( \mathbf x, \mathbf y, \mathbf y^*) & \ldots & P_{1,\ell}( \mathbf x, \mathbf y, \mathbf y^*) \\
			\vdots & & \vdots\\
			P_{\ell,1}(\mathbf x,  \mathbf y, \mathbf y^*) & \ldots & P_{\ell,\ell}(\mathbf x,  \mathbf y, \mathbf y^*)\end{array} \right)\right] = \tau \Big[  \ \frac 1 \ell  \sum_{i=1}^\ell P_{i,i}(\mathbf x,  \mathbf y, \mathbf y^*) \Big].
\end{equation*}
\end{Cor}
\subsection{Channel matrices}
\noindent We give a potential application of Theorem \ref{MainTh} in the context of communication, where rectangular block random matrices are sometimes investigated  for the study of wireless Multiple-input Multiple-Output (MIMO) systems \cite{LS,TV}. In the case of Intersymbol-Interference, the channel matrix $H$ reflects the channel effect during a transmission and is of the form
\begin{equation}\label{MIMOMatrix}
		H= \left( \begin{array}{ccccccccc}
		A_1 & A_2 & \hdots & A_L & \mathbf 0 & \hdots & &\hdots & \mathbf 0 \\
		\mathbf 0 & A_1 & A_2 & \hdots & A_L & \mathbf 0 & & & \vdots\\
		\vdots & \mathbf 0 & A_1 & A_2 & \hdots & A_L & \mathbf 0 && \\
		 &  & \ddots & \ddots & \ddots & & \ddots &  \ddots &\vdots\\
		\vdots &	  & & \ddots & \ddots & \ddots & &\ddots & 		\mathbf 0										\\
		\mathbf 0 &  \hdots &  & \hdots &\mathbf 0 & A_1 & A_2 &\hdots  & A_L \end{array} \right),
\end{equation}
 $(A_l)_{1\leq \ell \leq  L}$ are $n_R \times n_T$ matrices that are very often modeled by random matrices e.g. $A_1 \etc A_L$ are independent and for $\ell=1\etc L$ the entries of the matrix $A_\ell$ are independent identically distributed with finite variance. The number of matrices $L$ is the length of the impulse response of the channel, $n_T$ is the number of transmitter antennas and $n_R$ is the number of receiver antennas.
\\In order to calculate the capacity of such a channel, one must know the singular value distribution of $H$,  which is predicted by free probability theory. Theorem \ref{MainTh} may be used to obtain the convergence of the singular spectrum for a large class of such matrices. For instance we investigate in Section \ref{ProofCorMIMO} the following case:

\begin{Cor}[Rectangular band matrices]\label{CorMIMO}
Let $r$ and $t$ be integers. Consider a matrix $H$ of the form (\ref{MIMOMatrix}) such that for any $\ell=1\etc L$ one has $A_\ell = C_\ell M_\ell D_\ell$ where
\begin{enumerate}
\item $\mathbf M=(M_1 \etc M_L)$  is a family of independent $rN\times tN$ random matrices such that for $\ell=1\etc L$ the entries of $M_\ell$ are independent, Gaussian and centered with variance $\sigma_\ell^2/ N$,
\item the family of $rN \times rN$ matrices $\mathbf C=(C_1 \etc C_L)$ and the family of $tN \times tN$ matrices $\mathbf D=(D_1 \etc D_L)$ satisfy separately the assumptions of Theorem \ref{MainTh},
\item the families of matrices $\mathbf M$, $\mathbf C$ and $\mathbf D$ are independent. 
\end{enumerate}
Then, almost surely, the empirical eigenvalue distribution of $H H^*$ converges weakly to a measure $\mu$. Moreover, for any $\varepsilon >0$, almost surely there exists $N_0$ such that the singular values of $H$ belong to $\mathrm{Supp}(  \mu)+(-\varepsilon, \varepsilon)$.
\end{Cor}

\section{The strategy of proof}\label{PartSheme}
\noindent Let $\mathbf X_N=(X_1\toN \etc X_p\toN)$ and $\mathbf Y_N=(Y_1\toN \etc Y_q\toN)$ be as in Theorem \ref{MainTh}. We start with some remarks in order to simplify the proof.

\begin{enumerate}

\item We can suppose that the matrices of $\mathbf Y_N$ are Hermitian. Indeed for any $j=1\etc q$, one has $Y_j\toN = $ Re $Y_j\toN +i$ Im $Y_j\toN$, where
\begin{equation*}
		\mathrm{Re} \ Y_j\toN:=  \frac{1} 2 \big(  Y_j\toN + Y_j\toNs), \ \ \ \mathrm{Im} \ Y_j\toN:=  \frac{1} {2i} \big(  Y_j\toN - Y_j\toNs)
\end{equation*}
are Hermitian matrices. A polynomial in $(\mathbf Y_N, \mathbf Y_N^*)$ is obviously a polynomial in the family $(\mathrm{Re } \ Y_1\toN\etc \mathrm{Re } \ Y_q\toN,  \mathrm{Im } \ Y_1\toN\etc \mathrm{Im } \ Y_q\toN)$ and so the latter satisfies the assumptions of Theorem \ref{MainTh} as soon as $\mathbf Y_N$ does.

\item It is sufficient to prove the theorem for deterministic matrices $\mathbf Y_N$. Indeed, the matrices $\mathbf X_N$ and $\mathbf Y_N$ are independent. Then we can choose the underlying probability space to be of the form $\Omega=\Omega_1\times \Omega_2$, with $\mathbf X_N$ (respectively $\mathbf Y_N$) a measurable function on $\Omega_1$ (respectively $\Omega_2$). The event ''for all polynomials $P$ the convergences (\ref{MainThEqVoic}) and (\ref{MainThEq}) hold`` is a measurable set $\tilde \Omega\subset \Omega$. Assume that the theorem holds for deterministic matrices. Then for almost all $\omega_2 \in \Omega_2$, there exists a set $\tilde \Omega_1(\omega_2)$ for which for all $\omega_1\in \tilde \Omega_1$, (\ref{MainThEqVoic}) and (\ref{MainThEq}) hold for $(\mathbf X_N(\omega_1), \mathbf Y_N(\omega_2))$. The set of such couples $(\omega_1,\omega_2)$ is of outer measure one and is contained in $\tilde \Omega$, hence by Fubini's theorem $\tilde \Omega$ is of measure one.

\item It is sufficient to prove that for any polynomial the convergence of the norm in (\ref{MainThEq}) holds almost surely (instead of almost surely the convergence holds for all polynomials). Indeed we can switch the words ''for all polynomials with rational coefficients`` and ''almost surely`` and both the left and the right hand side in (\ref{MainThEq}) are continuous in $P$.

\end{enumerate}

\noindent In the following, when we say that $\mathbf Y_N=(Y_1\toN \etc Y_q\toN)$ is as in Section \ref{PartSheme}, we mean that $\mathbf Y_N$ is a family of deterministic Hermitian matrices satisfying (\ref{Th1wCV}) and (\ref{Th1sCV}).
\\
\\Remark that by (\ref{Th1sCV}), almost surely the supremum over $N$ of $\| Y_j\toN\|$ is finite for all $j=1\etc q$. Hence by Theorem \ref{VoiTh}, with probability one the non commutative law of $(\mathbf X_N,\mathbf Y_N)$ in $(\MN, .^*, \tau_N )$ converges to the law of non commutative random variables  $(\mathbf x,  \mathbf y)$ in a $^*$-probability space $( \mathcal A, .^*, \tau,)$: almost surely, for all polynomials $P$ in $p+q$ non commutative indeterminates, one has
\begin{equation}\label{ConvSheme1}
		\tau_N\big [ P(\mathbf X_N, \mathbf Y_N) \ \big ] \underset{N\rightarrow\infty}{\longrightarrow} \tau[P(\mathbf x, \mathbf y)],
\end{equation}
where the trace $\tau$ is completely defined by:
\begin{itemize}
\item $\mathbf x  = (x_1\etc x_p)$ is a free semicircular system,
\item $\mathbf y  =(y_1 \etc  y_q)$ is the limit in law of $\mathbf Y_N$,
\item $\mathbf x, \mathbf y$ are free.
\end{itemize}
Since $\tau$ is faithful on the $^*$-algebra spanned by $\mathbf x$ and $\mathbf y$, we can always assume that $\tau$ is a faithful trace on $\mathcal A$. Moreover, the matrices $\mathbf Y_N$ are uniformly bounded in operator norm. If we define $\| \cdot \|$ in $\mathcal A$ by Formula (\ref{DefNorm}), then $\| y_j\|$ is finite for every $j=1\etc q$. Hence, we can assume that $\mathcal A$ is a $\mathcal C^*$-probability space endowed with the norm $\| \cdot \|$.
\\
\\Haagerup and Thorbj\o rnsen describe in \cite{HT} a method to show that for all non commutative polynomials $P$, almost surely one has
\begin{equation}\label{HTmet}
		 \big \| P(\mathbf X_N )\big \|  \underset{N\rightarrow\infty}{\longrightarrow}  \| P(\mathbf x ) \|.
\end{equation}
We present in this section this method with some modification to fit our situation. First, it is easy to see the following.
\begin{Prop}For all non commutative polynomials $P$, almost surely one has
\begin{equation}\label{MajNorm}
		\underset{N \rightarrow\infty}{ \liminf } \ \big  \| P(\XYY )\big  \| \geq \| P(\xyy  )\|.
\end{equation}
\end{Prop}

\begin{proof} In a $\mathcal C^*$-algebra $(\mathcal A, .^*, \| \cdot \|)$, one has $\forall a\in \mathcal A$, $\|a\|^2=\|a^*a\|$. Hence, without loss of generality, we can suppose that $H_N:=P(\XYY)$ is non negative Hermitian and $h:=P(\xyy)$ is selfadjoint. Let $\mathcal L_N$ denote the empirical spectral distribution of $H_N$: 
	$$\mathcal L_N=\frac 1N \sum_{i=1}^N \delta_{\lambda_i},$$
where $\lambda_1 \etc \lambda_N$ denote the eigenvalues of $H_N$ and $\delta_{\lambda}$ the Dirac measure in $\lambda\in \mathbb R$. By (\ref{ConvSheme1}) and Hamburger's theorem \cite{HAM}, almost surely $\mathcal  L_N$ converges weakly to the compactly supported probability measure $\mu$ on $\mathbb R$ given by: for all polynomial $P$,
		$$\int P\textrm{d}\mu = \tau[ P(h)].$$
Since $\tau$ is faithful, the extrema of the support of $\mu$ is $\| h\|$ (\cite[proposition 3.15]{NS}). In particular, if $f:\mathbb R \rightarrow \mathbb R$ is a non negative continuous function whose support is the closure of a neighborhood of $\| h\|$ ($f$ not indentically zero), then almost surely there exists a $N_0\geq 0$ such that for all $N\geq N_0$ one has $\mathcal L_N(f)>0$. Hence for $N\geq N_0$ some eigenvalues of $H_N$ belong to the considered neighborhood of $ \| h \|$ and so $\|H_N\| \geq \| h\|$.
 \end{proof}


\noindent It remains to show that the limsup is smaller than the right hand side in (\ref{MajNorm}). The method is carried out in many steps.

\begin{center}
\begin{minipage}{16cm}{\bf Step 1.  A linearization trick:} \it With inequality {\rm{(\ref{MajNorm})}} established, the question of almost sure convergence of the norm of any polynomial in the considered random matrices can be reduced to the question of the convergence of the spectrum of any matrix-valued selfadjoint degree one polynomials in these matrices. More precisely, in order to get {\rm{(\ref{HTmet})}}, it is sufficient to show that for all $\varepsilon >0$, $k$ positive integer, $L$ selfadjoint degree one polynomial with coefficients in $\Mk$, almost surely there exists $N_0$ such that for all $N\geq N_0$,
\begin{equation}\label{SpecRed}
		\sp \big ( \ L(\XYY  ) \ \big ) \ \subset \sp \big ( \ L(\xyy  ) \ \big ) +(-\varepsilon, \varepsilon).
\end{equation}
\end{minipage}
\end{center}
We refer the readers to \cite[Parts 2 and 7]{HT} for the proof of this step, which is based on $\mathcal C^*$-algebra and operator space techniques. We only recall here the main ingredients. By an argument of ultraproduct it is sufficient to show the following: Let $(\tilde {\mathbf x}, \tilde {\mathbf y})$ be elements of a $\mathcal C^*$-algebra. Assume that for all selfadjoint degree one polynomials $L$ with coefficients in $\Mk$, one has
\begin{equation}\label{Step1Eq1}
	\sp \big ( \ L(\tilde {\mathbf x}, \tilde {\mathbf y}, \tilde {\mathbf y}^*  ) \ \big ) \ \subset \sp \big ( \ L(\xyy ) \ \big ).
\end{equation}
Then for all polynomials $P$ one has $\|P(\xyy)\| \geq \| P(\tilde {\mathbf x}, \tilde {\mathbf y}, \tilde {\mathbf y}^*) \| $. The linearization trick used to prove that fact arises from matrix manipulations and Arveson's theorem: with a dilation argument, one deduces from (\ref{Step1Eq1}) that there exists $\phi$ a unital $*$-homomorphism between the $\mathcal C^*$-algebra spanned by $(\mathbf x, \mathbf y)$ and the one spanned by $(\tilde {\mathbf x}, \tilde {\mathbf y})$ such that one has $\phi(x_i)=\tilde x_i$ for $i=1\etc p$, and $\phi(y_i)=\tilde y_i$ for $i=1\etc q$. A $*$-homomorphism being always contractive, one gets the result.
\\
\\We fix a selfadjoint degree one polynomial $L$ with coefficients in $\Mk$. To prove (\ref{SpecRed}) we apply the method of Stieltjes transforms. We use an idea from Bai and Silverstein in \cite{BS}: we do not compare the Stieltjes transform of $L(\mathbf X_N, \mathbf Y_N)$ with the one of $L(\mathbf x, \mathbf y)$, but with an intermediate quantity, where in some sense we have taken partially the limit $N$ goes to infinity, only for the GUE matrices. To make it precise, we realize the non commutative random variables $\big (\mathbf x, \mathbf y, (\mathbf Y_N)_{N\geq 1} \big )$ in a same $\mathcal C^*$-probability space $(\mathcal A, .^*, \tau, \| \cdot \|)$ with faithful trace, where
 \begin{itemize}
\item the families $\mathbf x$, $\mathbf y$, $\mathbf Y_1$, $\mathbf Y_2, \dots, \mathbf Y_N, \dots$ are free,
\item for any polynomials $P$ in $q$ non commutative indeterminates $\tau[P(\mathbf Y_N)]:=\tau_N[P(\mathbf Y_N)]$.
\end{itemize}

\noindent The intermediate object $L(\mathbf x, \mathbf Y_N )$ is therefore well defined as an element of $\mathcal A$. We use a theorem about norm convergence, due to D. Shlyakhtenko and stated in Appendix \ref{App1}, to relate the spectrum of $L(\mathbf x, \mathbf Y_N )$ with the spectrum of $L(\mathbf x, \mathbf y)$.

\begin{center}
\begin{minipage}{16cm}{\bf Step 2.  An intermediate inclusion of spectrum:} \it for all $\varepsilon>0$ there exists $N_0$ such that for all $N\geq N_0$, one has
\begin{equation}\label{SpecRed2}
		\sp \big ( \ L(\mathbf x, \mathbf Y_N ) \ \big ) \ \subset \sp \big ( \ L(\mathbf x, \mathbf y ) \ \big ) +(-\varepsilon, \varepsilon).
\end{equation}
\end{minipage}
\end{center}

\noindent We define the Stieltjes transforms $g_{L_N}$ and $g_{\ell_N}$ of $L_N=L(\mathbf X_N, \mathbf Y_N)$ and respectively $\ell_N=L(\mathbf x, \mathbf Y_N)$ by the formulas
\begin{eqnarray}
		g_{L_N} ( \lambda) & = & \esp \bigg [ \ttrN \Big[   \big (\lambda \unun_N -  L(\mathbf X_{N}, \mathbf Y_{N} ) \ \big)^{-1} \  \Big] \bigg ],\\
		g_{\ell_N}( \lambda) & = & \ttr \Big[  \big ( \lambda \unun-  L(\mathbf x, \mathbf Y_N )\ \big)^{-1} \  \Big]   \label{StieltInfty},
\end{eqnarray}
for all complex numbers $\lambda$ such that $ \mathrm{Im } \ \lambda>0$.

\begin{center}
\begin{minipage}{16cm}{\bf Step 3.  From Stieltjes transform to spectra:} \it  In order to show (\ref{Step1Eq1}) with (\ref{SpecRed2}) granted, 
%
%
it is sufficient to show the following: for every $\varepsilon>0$, there exist $N_{0}, \gamma, c, \alpha>0$ such that for all $N \geq N_{0}$, for all $\lambda$ in $\mathbb C$ such that $\varepsilon\leq (\mathrm{Im } \  \lambda)^{-1} \leq N^{\gamma}$, one has
\begin{equation}\label{EqPurp}
		| g_{L_N}(\lambda) - g_{\ell_N}(\lambda) | \leq \frac {c}{N^{2}}( \mathrm{Im} \ \lambda)^{-\alpha}.
\end{equation}
\end{minipage}
\end{center}

\noindent The proof of Estimate (\ref{EqPurp}) represents the main work of this paper. For this task we consider a generalization of the Stieltjes transform. We define the $\Mk$-valued Stieltjes transforms $G_{L_N}$ and $G_{\ell_N}$ of $L_N=L(\mathbf X_N, \mathbf Y_N)$ and respectively $\ell_N=L(\mathbf x, \mathbf Y_N)$ by the formulas
\begin{eqnarray}
		G_{L_N} ( \Lambda) & = & \esp \bigg [ \traceN \Big[   \big (\Lambda \UnN -  L(\mathbf X_{N}, \mathbf Y_{N} ) \ \big)^{-1} \  \Big] \bigg ],\\
		G_{\ell_N}( \Lambda) & = & \trace \Big[  \big ( \Lambda \Un-  L(\mathbf x, \mathbf Y_N )\ \big)^{-1} \  \Big] ,
\end{eqnarray}
for all $k \times k$ matrices $\Lambda$ such that the Hermitian matrix $ \mathrm{Im} \ \Lambda:= (\Lambda- \Lambda^*)/(2i)$ is positive definite. Since $g_{L_N}(\lambda) = \tau_k[ G_{L_N}(\lambda \mathbf 1_k)]$ and $g_{\ell_N}(\lambda) = \tau_k[ G_{\ell_N}(\lambda \mathbf 1_k)]$, a uniform control of $\|G_{L_N} ( \Lambda) - G_{\ell_N}( \Lambda)\|$ will be sufficient to show (\ref{EqPurp}). Here $\| \cdot \|$ denotes the operator norm.
\\
\\Due to the block structure of the matrices under consideration, these quantities are more relevant than the classical Stieltjes transforms. The polynomial $L$ is selfadjoint and of degree one, so we can write $L_N=a_0\otimes \mathbf 1_N +S_N+T_N$, $\ell_N=a_0\otimes \mathbf 1 +s +T_N$, where 
		$$S_N  =  \sum_{j=1}^p a_j \otimes X_j^{(N)}, \ s = \sum_{j=1}^p a_j \otimes x_j,  \ T_N =  \sum_{j=1}^q b_j \otimes Y_j^{(N)},$$
and $a_0\etc a_p,b_1 \etc b_q$ are Hermitian matrices in $\Mk$. We also need to introduce the $\Mk$-valued Stieltjes transforms $G_{T_N}$ of $T_N$:
\begin{eqnarray}
		G_{T_N}(\Lambda)  =  \traceN \Big[  \big ( \Lambda \Un-  T_N\ \big)^{-1} \  \Big],
\end{eqnarray}
for all $\Lambda$ in $\Mk$ such that Im $\Lambda$ is positive definite.
\\
\\The families $\mathbf x$ and $\mathbf Y_N$ being free in $\mathcal A$ and $\mathbf x$ being a free semicircular system, the theory of matrix-valued non commutative random variables gives us the following equation relating $G_{\ell_N}$ and $G_{T_N}$. It encodes the fundamental property of $\mathcal R$-transforms, namely the linearity under free convolution.

\begin{center}
\begin{minipage}{16cm}{\bf Step 4.  The subordination property for $\Mk$-valued non commutative random variables:} \it For all $\Lambda$ in $\Mk$ such that $\mathrm{Im} \ \Lambda$ is positive definite, one has
\begin{eqnarray}\label{Step4Eq}
		G_{\ell_N}(\Lambda) = G_{T_N}\Big ( \Lambda -a_0  - \mathcal R_s\big (  G_{\ell_N} (\Lambda) \ \big) \ \Big ),
\end{eqnarray}
where
		$$\mathcal R_s : M \mapsto \sum_{j=1}^p a_j M a_j.$$
\end{minipage}
\end{center}

\noindent We show that the fixed point equation implicitly given by (\ref{Step4Eq}) is, in a certain sense, stable under perturbations. On the other hand, by the asymptotic freeness of $\mathbf X_N$ and $\mathbf Y_N$, it is expected that Equation (\ref{Step4Eq}) is asymptotically satisfied when $G_{\ell_N}$ is replace by $G_{L_N}$. Since, in order to apply Step 3, we want an uniform control, we make this connection precise by showing the following:

\begin{center}
\begin{minipage}{16cm}{\bf Step 5.  The asymptotic subordination property for random matrices:} \it For all $\Lambda$ in $\Mk$ such that $\mathrm{Im} \ \Lambda$ is positive definite, one has
\begin{eqnarray}\label{Step5Eq}
		G_{L_N}(\Lambda) = G_{T_N}\Big ( \Lambda -a_0 - \mathcal R_s\big (  G_{L_N} (\Lambda) \ \big) \ \Big ) +\Theta_N(\Lambda),
\end{eqnarray}
where $\Theta_N(\Lambda)$ satisfies
\begin{eqnarray*}
		\left \| \Theta_N(\Lambda) \right \| & \leq & \frac {c } {N^2}  \left \| (\mathrm{Im } \ \Lambda)^{-1} \right\|^5
\end{eqnarray*}
for a constant $c$ and with $\| \cdot \|$ denoting the operator norm.
\end{minipage}
\end{center}

\noindent {\bf Organization of the proof}
\\We tackle the different points of the proof described above in the following order:

\begin{itemize}

\item {\bf Proof of Step 4.} The precise statement of the subordination property for $\Mk$-valued non commutative random variables is contained in Proposition \ref{PropRtr} and Proposition \ref{UniqSol}. We highlight in this section the relevance of matrix-valued Stieltjes transforms in a quite general framework.

\item {\bf Proof of Step 5.} The asymptotic subordination property for random matrices is stated in Theorem \ref{SubMATh} in a more general situation. The matrices $\mathbf Y_N$ can be random, independent of $\mathbf X_N$, satisfying a Poincar\'e inequality, without assumption on their asymptotic properties. This result is based on the Schwinger-Dyson equation and on the Poincar\'e inequality satisfied by the law of $\mathbf X_N$.

\item {\bf Proof of Estimate (\ref{EqPurp}).} The estimate will follow easily from the two previous items.

\item {\bf Proof of Step 2.} This part is based on $\mathcal C^*$-algebra techniques. Step 2 is a consequence of a result due to D. Shlyakhtenko which is stated Theorem \ref{ShlTh} of Appendix \ref{App1}. In a previous version of this article, when we did not know this result, we used the subordination property with $L(\mathbf x, \mathbf Y_N)$ replaced by $L(\mathbf x, \mathbf y)$ and $T_N$ replaced by its limit in law $t =  \sum_{j=1}^q b_j \otimes y_j$. Hence we obtained Theorem \ref{MainTh} with additional assumptions on $\mathbf Y_N$, notably a uniform rate of convergence of $G_{T_N}$ to the $\Mk$-valued Stieltjes transform of $t$.

\item {\bf Proof of Step 3.} The method is quite standard once Steps 2 and 4 are established. We use a version due to \cite{GKZ} which is based on the use of local concentration inequalities.
\end{itemize}

\section[Proof of Step 4]{Proof of Step 4: the subordination property for matrix-valued non commutative random variables}\label{SectionMkncrv}
\noindent In random matrix theory, a classical method lies in the study of empirical eigenvalue distribution by the analysis of its Stieltjes transform. In many situation, it is shown that this functional satisfies a fixed point equation and a lot of properties of the considered random matrices are deduced from this fact. The purpose of this section is to emphasize that this method can be generalized in the case where the matrices have a macroscopic block structure.
\\
\\Let $(\mathcal A, .^*,\tau, \| \cdot \|)$ be a $\mathcal C^*$-probability space with a faithful trace and $k\geq 1$ an integer. The algebra $\MkA$, formed by the $k \times k$ matrices with coefficients in $\mathcal A$, inherits the structure of $\mathcal C^*$-probability space with trace $(\tau_k\otimes \tau)$ and norm $\| \cdot \|_{\tau_k\otimes \tau}$ defined by (\ref{DefNorm}) with $\tau_k\otimes \tau$ instead of $\tau$. We also shall consider the linear functional $\trace$, called the partial trace.
\\
\\For any matrix $\Lambda$ in $\Mk$ we denote Im $\Lambda$ the Hermitian matrix $\frac 1 {2i}(\Lambda-\Lambda^*)$. We write Im $\Lambda>0$ whenever the matrix Im $\Lambda$ is positive definite and we denote
		$$\Mk^+=\big\{ \Lambda\in \Mk \ \big | \ \mathrm{Im} \ \Lambda>0 \big\}.$$
This lemma will be used throughout this paper. See \cite[Lemma 3.1]{HT} for a proof.
\begin{Lem}\label{ImLem} Let $z$ in $\MkA$ be selfadjoint. Then for any $\Lambda\in \Mk^+$, the element $(\Lambda\Un -z)$ is invertible and 
\begin{equation}\label{ImEst}
		\big \| (\Lambda\Un -z)^{-1} \big \|_{\tau_k\otimes \tau}\leq \| (\mathrm{Im} \ \Lambda)^{-1}\|.
\end{equation}
\end{Lem}
\noindent On the right hand side, $\| \cdot\|$ denotes the operator norm in $\Mk$.
%

\noindent For a selfadjoint non commutative random variable $z$ in $\MkA$, its $\Mk$-valued Stieltjes transform is defined by
\begin{eqnarray*}  \begin{array}{cccc} G_z : & \Mk^+ & \rightarrow & \Mk \\
				     				& \Lambda & \mapsto  &\trace \Big [ \big( \Lambda\otimes \mathbf 1-z)^{-1} \Big]. \end{array}
\end{eqnarray*}
The functional $G_z$ is well defined by Lemma \ref{ImLem} and satifies
		$$\forall \Lambda \in \Mk^+, \ \|G_z(\Lambda)\| \leq \| (\mathrm{Im} \ \Lambda)^{-1}\|.$$
It maps $\Mk^+$ to $\Mk^- = \big\{ \Lambda\in \Mk \ \big | \ -\Lambda \in \Mk^+ \big\}$ and is analytic (in $k^2$ complex variables on the open set $\Mk^+\subset \mathbb C^{k^2}$). Moreover, it can be shown (see \cite{VOI3}) that $G_z$ is univalent on a set of the form $U_\delta = \big\{ \Lambda \in \Mk^+ \ \big | \ \|\Lambda^{-1}\|<\delta \ \big \}$ for some $\delta>0$, and its inverse $G_z^{(-1)}$ in $U_\delta$ is analytic on a set of the form $V_\gamma=\big \{ \Lambda \in \Mk^- \ \big | \ \| \Lambda\|<\gamma \big\}$ for some $\gamma>0$.
\\
\\The amalgamated $\mathcal R$-transform over $\Mk$ of $z\in \Mk\otimes \mathcal A$ is the function $\mathcal R_z : G_z(U_\delta)\rightarrow \Mk$ given by
		$$\mathcal R_z(\Lambda) = G_z^{(-1)}(\Lambda) - \Lambda^{-1}, \ \ \forall \Lambda \in G_z(U_\delta).$$
The following proposition states the fundamental property of the amalgamated $\mathcal R$-transform, namely the subordination property, which is the keystone of our proof of Theorem \ref{MainTh}. 
\begin{Prop}\label{PropRtr} Let $ \mathbf x=(x_1\etc x_p)$ and $\mathbf y=(y_1 \etc y_q)$ be selfadjoint elements of $\mathcal A$ and let $\mathbf a=(a_1\etc a_p)$ and $\mathbf b=(b_1\etc b_q)$ be $k\times k$ Hermitian matrices. Define the elements of $\MkA$
\begin{equation*}
		s = \sum_{j=1}^p a_j\otimes x_j, \ \ \ t = \sum_{j=1}^q b_j\otimes y_j.
\end{equation*}
Suppose that the families $\mathbf x$ and $\mathbf y$ are free. Then one has
 \begin{enumerate}
\item {\bf Linearity property:} There is a $\gamma$ such that, in the domain $V_\gamma$, one has
		\begin{equation}\label{LinP}
				\mathcal R_{s+t}  = \mathcal R_{s} + \mathcal R_t .
		\end{equation}
\item {\bf Subordination property:} There is $\delta$ such that, for every $\Lambda$ in $U_\delta$, one has
		\begin{equation}
		G_{s+t}(\Lambda) = G_{t}\Big( \Lambda - \mathcal R_s \big ( \ G_{s+t}(\Lambda) \ \big) \ \Big).
		\end{equation}
\item {\bf Semicircular case:} If $(x_1 \etc x_p)$ is a free semicircular system, then we get
\begin{equation}\label{Rs}
		\mathcal R_s : \Lambda \mapsto \sum_{j=1}^pa_j\Lambda a_j.
\end{equation}
 \end{enumerate}
 \end{Prop}
\begin{proof} The linearity property has been shown by Voiculescu in \cite{VOI3} and the $\mathcal R$-transform of $s$ has been computed by Lehner in \cite{LEH}. We deduce easily the subordination property since by Equation (\ref{LinP}): there exists $ \gamma>0$ such that for all $\Lambda \in V_\gamma$,
		$$ G^{(-1)}_t(\Lambda) = G_{s+t}^{(-1)}(\Lambda) -\mathcal R_s(\Lambda).$$
Then there exists a $\delta>0$ such that, with $G_{s+t}(\Lambda)$ instead of $\Lambda$ in the previous equality,
		$$G^{(-1)}_t\big ( G_{s+t}(\Lambda) \ \big ) = \Lambda -\mathcal R_s\big ( G_{s+t}(\Lambda)  \ \big).$$
We compose by $G^{(-1)}_t$ to obtain the result.
\end{proof}

\noindent The subordination property plays a key role in our problem: it describes $G_{s+t}$ as a fixed point of a simple function involving $s$ and $t$ separately. Such a fixed point is unique and stable under some perturbation, as it is stated in Proposition \ref{UniqSol} below. Remark first that, for $\mathcal R_s$ given by (\ref{Rs}), for any $\Lambda$ in $\Mk^+$ and $M$ in $\Mk^-$,
\begin{equation}\label{FalseL1}
	\mathrm{Im } \ \big ( \Lambda - \mathcal R_s(M) \ \big ) = \mathrm{Im} \ \Lambda - \sum_{j=1}^p a_j \ \mathrm{Im} \ M \ a_j>0
\end{equation}
and 
\begin{equation}\label{FalseL2}
	\Big \| \Big ( \mathrm{Im } \ \big ( \Lambda - \mathcal R_s(M) \ \big ) \ \Big )^{-1} \Big \| \leq \| \ ( \mathrm{Im} \ \Lambda)^{-1} \|.
\end{equation}
In particular, by analytic continuation, the subordination property holds actually for any $\Lambda\in \Mk^+$ when $\mathbf x$ is a free semicircular system.

\begin{Prop}\label{UniqSol} Let $s$ and $t$ be as in Proposition \ref{PropRtr}, with $\mathbf x$ a free semicircular system.
\begin{enumerate}
\item {\bf Uniqueness of the fixed point:} For all $\Lambda \in \Mk^+$ such that
		$$\left \| ( \mathrm{Im } \ \Lambda)^{-1}  \right \| < \sqrt{ \sum_{j=1}^p \|a_j \|^2},$$
the following equation in $G_\Lambda\in \Mk^-$,
\begin{equation}
	G_\Lambda = G_t\Big ( \Lambda -\mathcal R_s ( \ G_\Lambda\ ) \Big ),
\end{equation}
admits a unique solution $G_\Lambda$ in $M_k(\mathbb C)^-$ given by $G_\Lambda = G_{s+t}(\Lambda)$.

\item {\bf Stability under analytic perturbations:} Let $G :Ê \Omega \rightarrow \Mk^-$ be an analytic function on a simply connected open subset $\Omega\subset \Mk^+$ containing matrices $\Lambda$ such that $  \| ( \mathrm{Im } \ \Lambda)^{-1}   \|$ is arbitrary small. Suppose that $G$ satisfies: for all $\Lambda \in \Omega$,
\begin{equation}\label{PropStabAnal}
		G(\Lambda) = G_t\Big (\Lambda-\mathcal R_s \big ( G(\Lambda) \ \big) \ \Big ) +\Theta(\Lambda),
\end{equation}
where the function $\Theta:\Omega \rightarrow \Mk$ is analytic and satisfies: there exists $\varepsilon>0$ such that for all $\Lambda$ in $\Omega$,
		$$\kappa(\Lambda):= \|\Theta(\Lambda)\| \ \| (\mathrm{Im} \ \Lambda)^{-1}\| \ \sum_{j=1}^p \| a_j\|^2<1-\varepsilon.$$
Then one has: $\forall \Lambda \in \Omega$
\begin{equation}
		\| G(\Lambda) - G_{s+t}(\Lambda)\| \leq \big ( 1+ c \ \| (\mathrm{Im }\ \Lambda)^{-1}\|^2 \ \big) \ \| \Theta(\Lambda)\|,
\end{equation}
where $c = \frac 1 \varepsilon  \sum_{j=1}^p \|a_j\|^2$.
\end{enumerate}
\end{Prop}

\begin{proof} {\it 1. Uniqueness of the fixed point:}
\\Fix $\Lambda\in $ M$_k(\mathbb C)^+$ such that
\begin{equation}\label{DomContr}
		\left \| ( \mathrm{Im } \ \Lambda)^{-1}  \right \| < \sqrt{ \sum_{j=1}^p \|a_j \|^2}.
\end{equation}
Denote  for any $M$ in $\Mk^-$ the matrix $\psi(M) =\Lambda  -\mathcal R_s(M)$, which is in $\Mk^+$ by (\ref{FalseL1}). We show that the function 
		$$\Phi_\Lambda: M \rightarrow G_t\big( \psi(M) \ \big )$$
is a contraction on $\Mk^-$. Remark that $\Phi_\Lambda$ maps $\Mk^-$ into $\Mk^-$. Moreover for all $M, \tilde M$ in $\Mk^-$,
\begin{eqnarray*}
		\lefteqn{	\| \Phi_\Lambda(M)-\Phi_\Lambda(\tilde M) \| } \\
		& & = \left \|(id_k \otimes \tau) \bigg [  \Big ( \psi(M)  \otimes \mathbf 1-t \Big )^{-1} - \Big (  \psi(\tilde M)   \otimes \mathbf 1-t \Big )^{-1} \bigg ] \right \| \\
		& & =\bigg \|(id_k \otimes \tau) \bigg [   \Big (  \psi(M)   \otimes \mathbf 1-t \Big )^{-1} \Big ( \sum_{j=1}^p a_j (M-\tilde M)a_j \Big)\UnN \Big ( \psi( \tilde M )  \otimes \mathbf 1-t \Big )^{-1} \bigg ] \bigg \| \\
		& & \leq \bigg \| \Big( {\mathrm{Im } \ \big(\psi(M) \otimes \mathbf 1-t \big )}\Big)^{-1} \bigg \|    \ \bigg \| \Big ( {\mathrm{Im } \ \big( \psi(\tilde M)\otimes \mathbf 1-t \big )} \Big ) ^{-1}\bigg \| \sum_{j=1}^p \| a_j \| ^2 \ \big \| M- \tilde M\Big \| \\
		& & \leq \left \| ({\mathrm{Im } \ \Lambda})^{-1} \right \| ^2 \sum_{j=1}^p \| a_j \| ^2 \ \| M- \tilde M \|.
\end{eqnarray*}
Hence the function $\Phi_{\Lambda}$ is a contraction and by Picard's theorem the fixed point equation $M=\Phi_\Lambda(M)$ admits a unique solution $M_\Lambda$ on the closed set of $k\times k$ matrices whose imaginary part is non positive semi-definite, which is necessarily $G_{s+t}$ by the subordination property.
\\
\\{\it 2. Stability under analytic perturbations:}
\\We set $\tilde G : \Omega \rightarrow \Mk^-$ given by: for all $\Lambda \in  \Omega $,
\begin{equation*}
		\tilde G(\Lambda) = G(\Lambda)-\Theta(\Lambda) = G_t\Big ( \Lambda - \mathcal R_s\big(G(\Lambda)\ \big) \ \Big).
\end{equation*}
We set $\tilde \Lambda :  \Omega \rightarrow \Mk$ given by: for all $\Lambda \in  \Omega $
\begin{equation*}
		\tilde \Lambda(\Lambda) =  \Lambda - \mathcal R_s (\Theta(\Lambda)) = \Lambda - \mathcal R_s\big (G(\Lambda) \ \big ) + \mathcal R_s\big (\tilde G(\Lambda) \ \big ).
\end{equation*}
In the following, we use $\tilde \Lambda$ as a shortcut for $\tilde \Lambda(\Lambda)$. One has $\tilde \Lambda - \mathcal R_s\big (\tilde G(\Lambda) \ \big ) =  \Lambda - \mathcal R_s\big (G(\Lambda) \ \big )$ which is in $\Mk^+$ by (\ref{FalseL1}). Hence we have: for all $\Lambda \in  \Omega $,
\begin{eqnarray}\label{ProofStabAn}
		\tilde G(\Lambda ) & = & G_t\Big ( \tilde \Lambda - \mathcal R_s\big ( \tilde G(\Lambda) \ \big ) \ \Big).
\end{eqnarray}
We want to estimate $\|(\mathrm {Im} \ \tilde \Lambda)^{-1}\|$ in terms of $ \|(\mathrm {Im} \ \Lambda)^{-1}\|$. For all $\Lambda$ in $\Omega$, we use the definition of $\tilde \Lambda$ and we write:
			$$ \mathrm {Im} \ \tilde \Lambda = \mathrm {Im} \ \Lambda \ \Big ( \mathbf 1_k - ( \mathrm {Im} \ \Lambda)^{-1}\mathcal R_s\big (\Theta (\Lambda) \ \big) \ \Big ).$$
Remark that $\|  ( \mathrm {Im} \ \Lambda)^{-1}\mathcal R_s\big (\Theta (\Lambda) \ \big) \|\leq \kappa(\Lambda) =  \|\Theta(\Lambda)\| \ \| (\mathrm{Im} \ \Lambda)^{-1}\| \ \sum_{j=1}^p \| a_j\|^2<1-\varepsilon$ by assumption. Then $ \mathrm {Im} \ \tilde \Lambda$ is invertible and one has
			$$ ( \mathrm {Im} \ \tilde \Lambda)^{-1} = \sum_{\ell\geq 0} \Big ( \ ( \mathrm {Im} \ \Lambda)^{-1}\mathcal R_s\big (\Theta (\Lambda) \ \big) \ \Big )^\ell \ ( \mathrm {Im} \ \Lambda)^{-1}.$$
We then obtain the following estimate
\begin{eqnarray*}
		\| ( \mathrm {Im} \ \tilde \Lambda)^{-1}\| & \leq  &  \Big \|\sum_{\ell\geq 0} \Big ( \ ( \mathrm {Im} \ \Lambda)^{-1}\mathcal R_s\big (\Theta (\Lambda) \ \big) \ \Big )^\ell \ ( \mathrm {Im} \ \Lambda)^{-1} \Big \| \\
				& \leq & \frac 1 {1-\kappa(\Lambda)} \|( \mathrm {Im} \ \Lambda)^{-1}\| < \frac 1 \varepsilon  \|( \mathrm {Im} \ \Lambda)^{-1}\| .
\end{eqnarray*}
By uniqueness of the fixed point and by (\ref{ProofStabAn}), for all $\Lambda\in \Omega$ such that $\| ($Im $\Lambda)^{-1}\| <\varepsilon\sqrt{ \sum_{j=1}^p \|a_j \|^2}$, one has $\tilde G(\Lambda) = G_{s+t}(\tilde \Lambda)$ (such matrices $\Lambda$ exist by assumption on $\Omega$). But the functions are analytic (in $k^2$ complex variables) so that the equality extends to $\Omega$. Then for all $\Lambda\in \Omega$,
\begin{eqnarray*}
	\| G(\Lambda)- G_{s+t} (\Lambda) \| & \leq &  \| G(\Lambda) -  \tilde G(\Lambda)\| + \| G_{s+t}(\tilde \Lambda)- G_{s+t}(\Lambda) \|.
\end{eqnarray*}
For the first term we have by definition of $\tilde G$ that $ \| G(\Lambda) -  \tilde G(\Lambda)\| \leq  \| \Theta(\Lambda) \|$. On the other hand, one has
\begin{eqnarray*}
\lefteqn{\| G_{s+t}(\Lambda) -   G_{s+t}(\tilde \Lambda) \|}\\
		 & = & \Big \| ( \textrm{id}_k \otimes \tau ) \big [ (\Lambda\otimes \mathbf 1  -s-t)^{-1} -  (\tilde \Lambda \otimes \mathbf 1  -s-t)^{-1}\ \big ]  \Big \|   \\
		&= &  \Big \|  ( \textrm{id}_k \otimes \tau ) \big [  (\Lambda\otimes \mathbf 1  -s-t)^{-1} ( \tilde  \Lambda \otimes \mathbf 1 - \Lambda\otimes \mathbf 1)  (\tilde \Lambda \otimes \mathbf 1  -s-t)^{-1}\ \big ]   \Big \| \ \\
		&\leq & \|  (\Lambda\otimes \mathbf 1  -s-t)^{-1} \| \ \|  \tilde \Lambda  - \Lambda\| \ \|   (\tilde \Lambda \otimes \mathbf 1 -s-t)^{-1}\| \\
		&\leq &\frac 1 \varepsilon \  \big  \|\mathcal R_s\big (\tilde G(\Lambda) \ \big ) -   \mathcal R_s\big (G(\Lambda) \ \big ) \ \big \| \ \| ( \mathrm{Im } \  \Lambda)^{-1} \|^2Ê\\
		& \leq &\frac 1 \varepsilon \  \sum_{j=1}^p \|a_j\|^2 \ \| ( \mathrm{Im } \  \Lambda)^{-1} \|^2 \ \| \Theta(\Lambda) \|.
\end{eqnarray*}
We then obtain as expected
\begin{equation*}
		\| G(\Lambda)- G_{s+t} (\Lambda) \|  \leq \Big ( 1 +  \frac 1 \varepsilon \  \sum_{j=1}^p \|a_j\|^2 \ \| ( \mathrm{Im } \  \Lambda)^{-1} \|^2\Big)  \ \| \Theta(\Lambda) \|.
\end{equation*}

\end{proof}

\section[Proof of Step 5]{Proof of Step 5: the asymptotic subordination property for random matrices}\label{PartSubordMA}
The purpose of this section is to prove Theorem \ref{SubMATh} below, where it is stated that, for $N$ fixed, the matrix-valued Stieltjes transforms of certain random matrices satisfy an asymptotic subordination property i.e. an equation as in (\ref{PropStabAnal}). This result is independent with the previous part and does not involve the language of free probability.
\\
\\Let $\mathbf X_N=(X_1\toN \etc X_p\toN)$ be a family of independent, normalized $N\times N$ matrices of the GUE and $\mathbf Y_N=(Y_1\toN \etc Y_q\toN)$ be a family of $N\times N$ random Hermitian matrices, independent of $\mathbf X_N$. We fix an integer $k\geq 1$ and Hermitian matrices $a_0, \ldots , a_p,b_1, \ldots , b_q\in \Mk$. We set $S_N$ and $T_N$ the $kN \times kN$ block matrices 
\begin{eqnarray*}
		 S_N  =  \sum_{j=1}^p a_j \otimes X_j^{(N)}, \ & &T_N =  \sum_{j=1}^q b_j \otimes Y_j^{(N)}.
\end{eqnarray*}
Define the $\Mk$-valued Stieltjes transforms of $S_N+T_N$ and $T_N$: for all $\Lambda \in \Mk^+=\big\{ \Lambda\in \Mk \ \big | \ \mathrm{Im} \ \Lambda>0 \big\}$,
\begin{eqnarray*}
		G_{S_N+T_N} (\Lambda)  & =& \esp \bigg [\traceN \Big[\big(\Lambda \otimes \mathbf 1_N  - S_N - T_N\big)^{-1}\Big] \ \bigg ],\\
		G_{T_N} (\Lambda)  & =& \esp \bigg[ \traceN \Big[\big(\Lambda \otimes \mathbf 1_N  - T_N\big)^{-1}\Big] \ \bigg] .
\end{eqnarray*}

\noindent We denote by $\mathcal R_s$ the functional
\begin{eqnarray*}
		\mathcal R_s :	& \Mk \rightarrow & \Mk\\
					& M \mapsto & \sum_{j=1}^p  a_j \ M \ a_j.
\end{eqnarray*}

\begin{Th}[Asymptotic subordination property]\label{SubMATh} Assume that there exists $\sigma\geq 1$ such that the joint law of the entries of the matrices $\mathbf Y_N$ satisfies a Poincaré inequality with constant $\sigma / N$, i.e. for any $f :\mathbb R^{2qN^2} \rightarrow \mathbb C$ function of the entries of $q$ matrices, of class $\mathcal C^1$ and such that $\esp \Big[  \  | f(\mathbf Y_{N} )| ^2  \ \Big] < \infty$, one has
\begin{equation}\label{HypConc}
		\var \Big( f(\mathbf Y_{N} ) \ \Big) \leq \frac {\sigma} N \ \esp\Big[ \| \nabla f(\mathbf Y_{N} ) \|^2 \Big],
\end{equation}
where $\nabla f$ denotes the gradient of $f$, $\var$ denotes the variance, $\var(\ x \ )=\esp\big [ \ \big| \ x- \esp[\ x \ ] \ \big |^2\big]$.
\\Then for any $\Lambda\in \Mk^+$, the Stieltjes transforms $G_{S_N+T_N}$ and $G_{T_N}$ satisfy
\begin{equation}\label{SYST}
		G_{S_N+T_N} (\Lambda)  = G_{T_N} \Big ( \Lambda   - \mathcal R_s\big (  G_{S_N+T_N} (\Lambda) \ \big) \ \Big ) +  \Theta_N(\Lambda) ,
\end{equation}
where $\Theta$ is analytic $\Mk^+\rightarrow \Mk$ and satisfies
\begin{eqnarray*}
		\left \| \Theta_N(\Lambda) \right \| & \leq & \frac {c } {N^2}  \left \| (\mathrm{Im } \ \Lambda)^{-1} \right\|^5,
\end{eqnarray*}
with $c =  {2k^{9/2}\sigma}\sum_{j=1}^p \|a_j\|^2       \Big ( \sum_{j=1}^p \left \| a_j \right \|+\sum_{j=1}^q \|b_j\|  \Big)^2$, $\| \cdot \|$ denoting the operator norm in M$_k(\mathbb C)$.
\end{Th}
\noindent The proof of Theorem \ref{SubMATh} is carried out in two steps.
\begin{itemize}
\item In Section \ref{MSDpart} we state a mean Schwinger-Dyson equation for random Stieltjes transforms (Proposition \ref{VarEstimate}).
\item In Section \ref{SDMpart} we deduce from Proposition \ref{VarEstimate} a Schwinger-Dyson equation for mean Stieltjes transforms (Proposition \ref{SDforMean}).
\end{itemize}
Theorem \ref{SubMATh} is a direct consequence of Proposition \ref{SDforMean} as it is shown in Section \ref{SecProofSubMA}.
%

%

\subsection{Mean Schwinger-Dyson equation for random Stieltjes transforms}\label{MSDpart}
\noindent For $\Lambda,\Gamma$ in $\Mk^+$, define the elements of $\MkMn$
\begin{eqnarray*}
		h_{S_N+T_N} (\Lambda) 	 & = & (\Lambda \otimes \mathbf 1_N- S_N - T_N)^{-1}, \\
		h_{T_N} (\Gamma)		 & = & (\Gamma\otimes \mathbf 1_N-T_N)^{-1},
\end{eqnarray*}
and  $H_{S_N+T_N} (\Lambda) = \traceN \Big [h_{S_N+T_N} (\Lambda)\Big ]$, $H_{T_N} (\Lambda)=\traceN \Big [ h_{T_N}(\Lambda)\Big ]$.
\begin{Prop}[Mean Schwinger-Dyson equation for random Stieltjes transforms]\label{VarEstimate}
For all $\Lambda,\Gamma \in \Mk^+$ we have
\begin{equation}\label{SD}
		 \esp \bigg[ H_{S_N+T_N}(\Lambda) -  H_{T_N}(\Gamma) - \traceN \Big [ h_{T_N}(\Gamma)  \Big ( \mathcal R_s\big ( H_{S_N+T_N}(\Lambda)\ \big ) -\Lambda +\Gamma \Big)\UnN \ h_{S_N+T_N}(\Lambda) \Big ] \bigg]  = 0.
\end{equation}
\end{Prop}

\noindent The result is a consequence of integration by parts for Gaussian densities and of the formula for the differentiation of the inverse of a matrix. If $(g_1 \etc g_N)$ are independent identically distributed centered real Gaussian variables with variance $\sigma^2$ and $F: \mathbb R^N \rightarrow \mathbb C$ a differentiable map such that $F$ and its partial derivatives are polynomially bounded, one has for $i=1, \ldots ,  N$
		$$ \esp \Big[\ g_i \ F(g_1\etc g_N) \ \Big] = \sigma^2 \esp \bigg[ \frac{ \partial F}{\partial x_i}(g_1\etc g_N) \ \bigg].$$
This induces an analogue formula for independent matrices of the GUE, called the Schwinger-Dyson equation, where the Hermitian symmetry of the matrices plays a key role. For instance, if $P$ is a monomial in $p$ non commutative indeterminates, one has for $i=1, \ldots ,  p$,
\begin{equation*}
		 \esp \bigg[  \tau_N \Big[  \ X_i^{(N)} \ P(\mathbf X_N) \ \Big] \ \bigg]
				= \sum_{P=Lx_iR} \esp \bigg[ \tau_N\Big[L(\mathbf X_N)\ \Big] \ \tau_N\Big[R(\mathbf X_N) \ \Big] \bigg],
\end{equation*}
the sum over all decompositions $P=Lx_iR$ for $L$ and $R$ monomials being viewed as the partial derivative.
\\
\\This formula has an analogue for analytical maps instead of polynomials. The case of the function $\mathbf X_N \mapsto (\Lambda\UnN -S_N)^{-1}$ is investigated in details in \cite[Formula (3.9)]{HT}, our proof is obtained by minor modifications.

\begin{proof} Denote by $(\epsilon_{m,n})_{m,n=1, \ldots ,  N}$ the canonical basis of $\MN$. By \cite[Formula (3.9)]{HT} with minor modification, we get the following: for all $\Lambda, \Gamma$ in $\Mk^+$ and $j=1\etc p$,
\begin{eqnarray*}
		\lefteqn{ \esp \Big [ (\mathbf 1_k \otimes X_j^{(N)}) ( \Lambda \otimes \mathbf 1_N-S_N-T_N)^{-1} \ \Big | \ T_N \Big] }\\
			& &= \esp \Big[  \frac 1 N  \sum_{m,n=1}^N(\mathbf 1_k\otimes \epsilon_{m,n})  ( \Lambda \UnN -S_N-T_N)^{-1} (a_j\otimes \epsilon_{n,m}) ( \Lambda \otimes \mathbf 1_N -S_N-T_N)^{-1} \ \Big | \ T_N \Big].
\end{eqnarray*}
In these equations, $\esp [ \cdot | T_N]$ stands for the conditional expectation with respect to $T_N$. Furthermore, for any $M$ in $\Mk\otimes \MN$, one has
\begin{eqnarray*}
	\frac 1 N  \sum_{m,n=1}^N (\mathbf 1_k\otimes \epsilon_{m,n}) \ M \ (\mathbf 1_k\otimes \epsilon_{n,m})  = \traceN [ \ M \ ]\otimes \mathbf 1_N.
\end{eqnarray*}
Indeed the formula is clear if $M$ is of the form $M=\tilde M \otimes \epsilon_{u,v}$ and extends by linearity. In particular, with $M=( \Lambda \otimes \mathbf 1_N-S_N-T_N)^{-1}(a_j\otimes \mathbf 1_N)$, we obtain that: for all $\Lambda, \Gamma$ in $\Mk^+$ and $j=1\etc p $,
\begin{eqnarray*}
		\lefteqn{ \esp \Big [ (a_j \otimes X_j^{(N)}) ( \Lambda \otimes \mathbf 1_N-S_N-T_N)^{-1} \ \Big | \ T_N \Big] }\\
			& & = \esp \bigg [  (a_j\otimes \mathbf 1_N) \bigg ( \traceN\Big[ ( \Lambda \otimes \mathbf 1_N-S_N-T_N)^{-1} \Big]a_j\otimes \mathbf 1_N \bigg)( \Lambda \otimes \mathbf 1_N-S_N-T_N)^{-1} \ \bigg | \ T_N \bigg]\\
			& & =  \esp \Big [ \big( a_j H_{S_N+T_N} a_j  \otimes \mathbf 1_N \big) \ h_{S_N+T_N} \ \Big | \ T_N \Big].
\end{eqnarray*}
Recall that $S_N = \sum_{j=1}^p a_j\otimes X_j^{(N)}$ and $\mathcal R_s : M \mapsto \sum_{j=1}^p a_jMa_j$, so that for all $\Lambda, \Gamma$ in $\Mk^+$, one has
\begin{eqnarray}
			\lefteqn{ \esp \Big [ (\Gamma \UnN - T_N)^{-1} \ S_N \  ( \Lambda \UnN -S_N-T_N)^{-1}  \Big ] } \nonumber \\
			& & =  \esp \Big [ \ (\Gamma \UnN - T_N)^{-1}  \ \sum_{j=1}^p \esp \Big [ (a_j \otimes X_j^{(N)}) \  ( \Lambda \UnN -S_N-T_N)^{-1} \ \Big | \ T_N \Big ] \  \Big ]\nonumber \\
			& & = \esp \Big [ \ h_{T_N} (\Gamma)  \ \esp \Big [  \big(\sum_{j=1}^p a_j H_{S_N+T_N}(\Lambda) a_j  \otimes \mathbf 1_N \big) \ h_{S_N+T_N}(\Lambda) \ \Big | \ T_N \Big]\nonumber \\
%
%
%
%
			& & = \esp \bigg [ h_{T_N}(\Gamma)\ \Big( \mathcal R_s \big ( H_{S_N+T_N}(\Lambda) \ \big ) \UnN \Big ) \   h_{S_N+T_N}(\Lambda) \bigg ].\label{DemScwD}
\end{eqnarray}
We take the partial trace in Equation (\ref{DemScwD}) to obtain:
\begin{eqnarray}	\lefteqn{  \esp \bigg [\traceN\Big [ h_{T_N}(\Gamma) \ S_N \ h_{S_N+T_N}(\Lambda)  \Big ] \bigg ]} \nonumber \\
	& = &	\esp \bigg [ \traceN \Big [h_{T_N}(\Gamma)\ \Big( \mathcal R_s \big ( H_{S_N+T_N}(\Lambda) \ \big) \UnN \Big )\   h_{S_N+T_N}(\Lambda) \Big ]\bigg ]. \label{proofMSD}
\end{eqnarray}
We now rewrite $S_N$ as follow:
\begin{eqnarray*}
	S_N    =  (\Lambda  - \Gamma) \otimes \mathbf 1_N  + (\Gamma \otimes \mathbf 1_N -T_N) - (\Lambda \otimes \mathbf 1_N -S_N - T_N ).
\end{eqnarray*}
Re-injecting this expression in the left hand side of Equation (\ref{proofMSD}), one gets Equation (\ref{SD}):
\begin{eqnarray*}
	\lefteqn{ \esp \bigg [ \traceN \Big [h_{T_N}(\Gamma)\ \Big(  \mathcal R_s \big ( H_{S_N+T_N}(\Lambda) \ \big )  \UnN \Big )\   h_{S_N+T_N}(\Lambda) \Big ]\bigg ] }\nonumber\\
		  & = & \esp \bigg [\traceN\Big [ h_{T_N} (\Gamma) \ (\Lambda   - \Gamma)\UnN \ h_{S_N+T_N} (\Lambda) + h_{S_N+T_N}(\Lambda)    - h_{T_N}(\Gamma) \Big ] \bigg ] \\
		  & = & \esp \bigg [\traceN\Big [ \ h_{T_N}(\Gamma) \ \Big( (\Lambda   - \Gamma)\UnN\Big) \  h_{S_N+T_N}(\Lambda)  \Big ]+ H_{S_N+T_N}(\Lambda) \ - \ H_{T_N}(\Gamma) \bigg ].
\end{eqnarray*}
\end{proof}

\subsection {Schwinger-Dyson equation for mean Stieltjes transforms}\label{SDMpart}
\noindent We use the concentration properties of the law of $(\mathbf X_N, \mathbf Y_N)$ to get from Equation (\ref{SD}) a relation between $G_{S_N+T_N}$ and $G_{T_N}$. We define the centered version of $H_{S_N+T_N}$ by: for all $\Lambda$ in $\Mk^+$,
\begin{eqnarray}\label{Def1Var}
		K_{S_N+T_N}(\Lambda)  & = & H_{S_N+T_N}(\Lambda) - G_{S_N+T_N} (\Lambda), \textrm{ in }\Mk.
\end{eqnarray}
We introduce the random linear map 
\begin{eqnarray}
		\begin{array}{cccc} l_{N,\Lambda,\Gamma}: & \textrm{M}_k(\mathbb C) \otimes \textrm{M}_N(\mathbb C)&   \rightarrow &  \textrm{M}_k(\mathbb C) \otimes \textrm{M}_N(\mathbb C)\\
		&  M& \mapsto & h_{T_N}(\Gamma) \ M \ h_{S_N+T_N}(\Lambda) \end{array}\label{Def2Var}
\end{eqnarray}
and its mean
\begin{equation}\label{Def3Var}
		 L_{N,\Lambda,\Gamma}: M \mapsto \esp\Big[l_{N,\Lambda,\Gamma}(M)\ \Big].
\end{equation}
Remark that if $M$ is a random matrix, then $L_{N,\Lambda,\Gamma}(M) = \esp\big [  h_{\tilde T_N}(\Gamma) \ M \ h_{\tilde S_N+\tilde T_N}(\Lambda)  \big | M\big ]$, where $(\tilde S_N+\tilde T_N)$ is an independent copy of $(S_N+T_N)$ independent of $M$.

\begin{Prop}[Schwinger-Dyson equation for mean Stieltjes transforms]\label{SDforMean} For all $\Lambda,  \Gamma$ in $\Mk^+$, one has\begin{equation}\label{SDin}
		 G_{S_N+T_N} (\Lambda) -  G_{T_N}(\Gamma) - \traceN \bigg [ L_{N,\Lambda,\Gamma} \bigg( \ \Big(  \mathcal R_sÊ\big (  G_{S_N+T_N} (\Lambda) \ \big ) -\Lambda+\Gamma  \Big) \UnN \ \bigg) \  \bigg]    = \Theta_N(\Lambda, \Gamma),
\end{equation}
where 
\begin{equation}
	\Theta_N(\Lambda, \Gamma) = \esp \bigg[ \traceN \Big [ \ \left( l_ {N,\Lambda,\Gamma} -   L_{N,\Lambda,\Gamma} \right) \Big ( \mathcal R_s \big ( K_{S_N+T_N}(\Lambda) \ \big )  \UnN \Big) \ \Big ] \bigg]
\end{equation}
is controlled in operator norm by the following estimate:
\begin{equation}\label{Theta}
		\left \| \Theta_N(\Lambda, \Gamma) \right \| \leq \frac {c} {N^2}  \left \| (\mathrm{Im } \ \Gamma)^{-1} \right\| \left \| (\mathrm{Im } \ \Lambda)^{-1} \right\|^3 \ \Big (  \| (\mathrm{Im } \ \Gamma)^{-1}\| + \| (\mathrm{Im } \ \Lambda)^{-1}\| \Big ),
\end{equation}
with $c=   {k^{9/2}\sigma} \sum_{j=1}^p \|a_j\| ^2      \Big ( \sum_{j=1}^p \left \| a_j \right \|+\sum_{j=1}^q \|b_j\|  \Big)^2$.
\end{Prop}

\begin{proof}[Proof of Proposition \ref{SDforMean}]
\noindent  We first expand $\Theta_N(\Lambda, \Gamma) $: for all $\Lambda, \Gamma$ in $\Mk^+$, we have
\begin{eqnarray*}
		 \Theta_N(\Lambda, \Gamma) &:= & \esp \Bigg[ \traceN \bigg [ \ \left( l_ {N,\Lambda,\Gamma} -   L_{N,\Lambda,\Gamma} \right) \Big ( \mathcal R_s \big(H_{S_N+T_N}(\Lambda)-G_{S_N+T_N} (\Lambda) \ \big) \UnN \Big) \ \bigg ] \Bigg]\\
		 & = &\esp \Bigg[ \traceN \bigg [  l_ {N,\Lambda,\Gamma}  \Big ( \mathcal R_s\big ( H_{S_N+T_N}(\Lambda) \ \big ) \UnN \Big) \ \bigg ] \Bigg]\\
		 &    & - \traceN \bigg [  L_ {N,\Lambda,\Gamma}  \Big( \mathcal R_s \big (  G_{S_N+T_N} (\Lambda) \ \big )  \UnN \Big) \ \bigg ].
\end{eqnarray*}
By Equation (\ref{SD}), we get the following:
\begin{eqnarray*}
\lefteqn{\esp \Bigg[ \traceN \bigg [ l_{N,\Lambda,\Gamma}  \Big ( \mathcal R_s \big ( H_{S_N+T_N} (\Lambda)\ \big )  \UnN \Big)  \bigg ] \Bigg] }\\
 & = &  \esp \Bigg[ \traceN \bigg [ l_{N,\Lambda,\Gamma} \Big( \ \left ( \Lambda -\Gamma \right)\UnN  \Big) \ \bigg ] - H_{T_N}(\Gamma) + H_{S_N+T_N}(\Lambda) \Bigg] \\
 & = &  \traceN \bigg [ L_{N,\Lambda,\Gamma} \Big( \ \left ( \Lambda -\Gamma \right)\UnN  \Big) \ \bigg ] - G_{T_N}(\Gamma) + G_{S_N+T_N}(\Lambda),
\end{eqnarray*}
which gives Equation (\ref{SDin}).
\\
\\We use the Poincaré inequality to control the operator norm of $\Theta_N$: if $(g_1 \etc g_K)$ are independent identically distributed centered real Gaussian variables with variance $v^2$ and $F$ is a differentiable map $\mathbb R^K \rightarrow \mathbb C$ such that $F$ and its partial derivatives are polynomially bounded, then (see \cite[Theorem 2.1]{CN})
		$$\var \Big( F(g_1 \etc g_K ) \ \Big ) \leq v^2 \esp \Big[ \ \| \nabla F(g_1\etc g_K) \ \|^2 \ \Big].$$
The Poincaré inequality is compatible with tensor product and then such a formula is still valid when $F$ is a function of the matrices $\bf X_N$ and $\bf Y_N$ with $v^2=\frac \sigma N$.
\\
\\We will often deal with matrices of size $k \times k$. Since the integer $k$ is fixed, we can use intensively the equivalence of norms, the constants appearing will not modify the order of convergence. For any integer $K$, we denote the Euclidean norm of a $K\times K$ matrix $A=(a_{m,n})_{1\leq m,n \leq K}$ by 
		$$\| A \|_e = \sqrt{ \sum_{m,n=1}^K |a_{m,n}|^2},$$
and its infinity norm by
		$$\| A \|_\infty = \underset{m,n=1, \ldots ,  K}{\max} \ |a_{m,n}|.$$ 
Recall that if $A,B$ are  $K\times K$ matrices we have the following inequalities
\begin{eqnarray}
		\| A \|   \leq  \|A\|_e \leq \sqrt K \| A \| , \label{CompNorm1}\\
		\| A \|   \leq  \sqrt K \| A \|_\infty  \leq \sqrt K \| A \|_e,  \label{CompNorm12}\\
		\| A B\|   \leq  \|A\|_e \  \|B\|.		\label{CompNorm2}
\end{eqnarray}
When $A$ is in M$_k(\mathbb C) \ \otimes$ M$_N(\mathbb C)$, its Euclidean norm is defined by considering $A$ as a $kN\times kN$ matrix. In the following we will write an element $Z$ of $\MkMn$ 
\begin{eqnarray}\label{NotationOmega}
		Z & = & \sum_{m,n=1}^N \sum_{u,v=1}^k Z^{m,n}_{u,v} \ \epsilon_{u,v}\otimes \epsilon_{m,n} = \sum_{m,n=1}^N Z^{(m,n)}\otimes \epsilon_{m,n} =  \sum_{u,v=1}^k  \epsilon_{u,v}\otimes Z_{(u,v)},
\end{eqnarray}
where for $m,n=1, \ldots ,  N$ and $u,v=1, \ldots ,  k$, $Z^{m,n}_{u,v}$ is a complex number, $Z^{(m,n)}$ is a $k\times k$ matrix, and $Z_{(u,v)}$ is a $N\times N$ matrix; we use the same notation for the canonical bases of $\Mk$ and $\MN$.
\\We fix $\Lambda,\Gamma$ in $\Mk^+$ until the end of this proof and we use for convenience the following notations:
\begin{eqnarray*}
		M_N & = & \mathcal R_s \big ( K_{S_N+T_N}(\Lambda) \ \big )\\
		h_N^{(1)} & = & h_{S_N+T_N}(\Lambda)\\
		h_N^{(2)} & = & h_{T_N}(\Gamma)\\
		l_N & = & l_{N,\Lambda,\Gamma}\\
		L_N & = & L_{N,\Lambda,\Gamma}.
\end{eqnarray*}	
We consider $(\tilde h_N^{(1)},\tilde h_N^{(2)})$ an independent copy of $( h_N^{(1)}, h_N^{(2)})$, independent of $\mathbf X_N$ and $\mathbf Y_N$ (and hence of all the random variables considered). Recall that by definitions (\ref{Def2Var}) and (\ref{Def3Var}): for all $\Lambda, \Gamma$ in $\Mk^+$, we have
\begin{eqnarray*}
		 l_{N}&:  A\in \textrm M_k(\mathbb C)  \  \mapsto \ &h_N^{(2)} \ A \ h_N^{(1)}\in \textrm M_k(\mathbb C) ,\\
		 L_{N}&: A \in \textrm M_k(\mathbb C) \ \mapsto \ & \esp\Big[l_{N}(A)\ \Big]\in \textrm M_k(\mathbb C) .
\end{eqnarray*}
With the notations of (\ref{NotationOmega}) we have
\begin{eqnarray*}
\lefteqn{  \traceN \Big [ ( l_{N}- L_{N} ) \ ( M_N\UnN) \Big ] }\\
	& = & \traceN \big [ h_N^{(2)}  \ (M_N\UnN) \ h_N^{(1)}  \big ] - \esp \Big[ \traceN \big [\tilde h_N^{(2)}  \ (M_N\UnN) \ \tilde h_N^{(1)} \big ] \ \Big | \ M_N \Big ]\\
	& = & \frac 1 N \sum_{m,n=1}^N \bigg [  \Big(h_N^{(2)}   \Big)^{(m,n)} M_N \ \Big(h_N^{(1)}  \Big)^{(n,m)}  -  \esp \Big [  \Big(\tilde h_N^{(2)}  \Big)^{(m,n)} M_N \ \Big(\tilde h_N^{(1)}  \Big)^{(n,m)} \  \Big | \ M_N \Big ] \ \bigg ].
\end{eqnarray*}

\noindent To estimate the operator norm of $\Theta_N$ we use the domination by the infinity norm (\ref{CompNorm12}) in order to split the contributions due to $M_N$ and due to $ l_{N }- L_{N } $: we get
\begin{eqnarray*}
  \lefteqn{  \| \Theta_N(\Lambda,\Gamma) \| =  \bigg \| \esp  \Big [ \traceN \left [ ( l_{N }- L_{N } ) \ ( M_N\UnN) \right ] \Big] \bigg \|}\nonumber\\
  	&  \leq &  \sqrt k  \Bigg \| \esp  \Bigg [ \frac 1 N \sum_{m,n=1}^N   \Big(h_N^{(2)}  \Big)^{(m,n)} M_N \  \Big(h_N^{(1)}   \Big)^{(n,m)}
     -  \esp \Big [  \Big(\tilde h_N^{(2)}  \Big)^{(m,n)} M_N \  \Big(\tilde h_N^{(1)}  \Big)^{(n,m)} \  \Big | \ M_N \Big ] \Bigg ] \Bigg \|_\infty \nonumber\\
      &  \leq &  k^{5/2} \underset{\substack{1\leq u,v \leq   k \\ 1\leq u',v' \leq   k } } {\max} \Bigg | \esp  \Bigg [ (M_N)_{u',v'} \times \frac 1 N\sum_{m,n=1}^N  \Big(h_N^{(2)}  \Big)^{m,n}_{u,u'} \Big(h_N^{(1)}   \Big)^{n,m}_{v',v} - \esp \Big [  \Big( h_N^{(2)}  \Big)^{m,n}_{u,u'} \Big(h_N^{(1)}  \Big)^{n,m}_{v',v} \Big ]  \Bigg ] \Bigg | \nonumber\\
      &  \leq & k^{5/2} \underset{\substack{  u,v , u',v'  } } {\max} \esp  \Bigg [ | (M_N)_{u',v'} | \times  \bigg | \tau_N\Big [       \big (      h_N^{(1,2)}\big ) _{\substack{u,v\\ u',v' }}       \big]- \esp \Big [  \tau_N \Big[        \big( h_N^{(1,2)} \big )_{\substack{u,v\\ u',v' }}          \Big]\Big ]  \bigg | \Bigg ]  \\
      &  \leq & k^{5/2} \underset{\substack{  u,v  ,  u',v'   } } {\max} \esp  \Bigg [ | (M_N)_{u',v'} | \times  \bigg | \tau_N\Big [       \big (      k_N^{(1,2)}\big )_{\substack{u,v\\ u',v' }}       \Big]  \bigg | \Bigg ],\label{EstmTheta1}
\end{eqnarray*}
where we have denoted the $N\times N$ matrices 
\begin{eqnarray*}
		\big (h_N^{(1,2)}\big )_{\substack{u,v\\ u',v' }} & = & \big( h_N^{(2)}  \big)_{(u,u')} \big(h_N^{(1)}  \big)_{(v',v)}, \\
		\big (k_N^{(1,2)} \big)_{\substack{u,v\\ u',v' }} & = & \big (h_N^{(1,2)} \big) _{\substack{u,v\\ u',v' }} -\esp\Big[ \big ( h_N^{(1,2)} \big)_{\substack{u,v\\ u',v' }}  \Big].
\end{eqnarray*}
Remark that by (\ref{CompNorm2}), for $u',v'=1, \ldots ,  k$,
\begin{eqnarray*}
	|(M_N)_{u',v'}| & = & \Big |    \Big(\sum_{j=1}^p a_j K_{S_N+T_N}(\Lambda) a_j\Big)_{u',v'} \Big | \leq \Big \|    \sum_{j=1}^p a_j K_{S_N+T_N}(\Lambda) a_j \Big \|_e  \leq    \sum_{j=1}^p   \|  a_j \|^2 \ \|K_{S_N+T_N}(\Lambda) \|_e.\label{EstmTheta2}
\end{eqnarray*}
Then by Cauchy-Schwarz inequality we get:
\begin{eqnarray}
 \| \Theta_N(\Lambda,\Gamma) \|    & \leq & k^{5/2} \sum_{j=1}^p \| a_j\|^2 \Big (        \esp \big [  \|K_{S_N+T_N}(\Lambda) \|_e^2 \big ]                  \ \underset{\substack{u,v,u',v' } } {\max}         \esp \Big [         \Big | \tau_N \big[       \big(   k_N^{(1,2)}\big)_{\substack{u,v\\ u',v' }}       \big] \ \Big |^2     \   \Big ] \ \Big )^{1/2} \nonumber \\
   & \leq & k^{5/2} \sum_{j=1}^p \| a_j\|^2 \bigg (     \sum_{u,v=1}^k     \var    \big(H_{S_N+T_N}(\Lambda)\big) _{u,v}   \ \underset{\substack{u,v,u',v' } } {\max} \var \Big (  \tau_N\big[ \big( h_N^{(1,2)} \big)_{\substack{u,v\\ u',v' }} \big] \ \Big ) \ \bigg )^{1/2}.\ \ \ \label{EstmTheta3}
\end{eqnarray}

\noindent One is reduced to the study of variances of random variables. To use the Poincaré inequality, we write for $u,v,u',v'=1, \ldots ,  k$,
\begin{eqnarray*}
		\big(H_{S_N+T_N}(\Lambda)\ \big)_{u,v} & = & F^{(1)}_{u,v}\big(\mathbf X_N,\mathbf Y_N \big),\\
		\tau_N \Big[ \big(h_N^{(1,2)} \big)_{\substack{u,v\\ u',v' }} \Big] & = & F^{(2)}_{u,v,u',v'}\big(\mathbf X_N,\mathbf Y_N \big ),
\end{eqnarray*}
where for all selfadjoint matrices $\mathbf A=(A_1, \dots ,A_p)$ in $\MN$, for all $\mathbf B=(B_1 \etc B_q)$ in $\MN$  and with $\tilde S_N= \sum_{j=1}^p a_j\otimes A_j$, $\tilde T_N= \sum_{j=1}^q b_j\otimes B_j $, we have set

\begin{eqnarray*}
		\lefteqn{F^{(1)}_{u,v} (\mathbf A , \mathbf B )  =  \Big ( \traceN \big[(\Lambda\UnN- \tilde S_N-\tilde T_N)^{-1}\big] \ \Big)_{u,v}} \nonumber\\
				& = & \frac 1 N (\Tr_k \otimes \Tr_N) \big [ (\epsilon_{v,u} \otimes \mathbf 1_N) ( \Lambda\UnN - \tilde S_N-\tilde T_N  )^{-1}\big],\label{EstmTheta4}
\end{eqnarray*}
\begin{eqnarray*}
		\lefteqn{F^{(2)}_{u,v,u',v'} (\mathbf A , \mathbf B )  =   \tau_N \bigg[ \Big( (\Lambda\UnN  -  \tilde S_N-\tilde T_N)^{-1} \Big)_{(u,u')} \ \Big( (\Gamma\UnN-\tilde T_N)^{-1} \Big)_{(v',v)} \bigg]} \nonumber \\
				& = & \frac 1 N (\Tr_k \otimes \Tr_N) \Big [ (\epsilon_{v,u} \otimes \mathbf 1_N) (\Gamma\UnN- \tilde T_N)^{-1}  (\epsilon_{u',v'}\UnN ) \  ( \Lambda\UnN  -\tilde S_N-\tilde T_N )^{-1}\Big ].\label{EstmTheta5}
\end{eqnarray*}
The functions and their partial derivatives are bounded (see \cite[Lemma 4.6]{HT} with minor modifications), so that, since the law of $(\mathbf X_N, \mathbf Y_N)$ satisfies a Poincaré inequality with constant $\frac \sigma N$, one has
\begin{eqnarray*}
		 \var \ \Big(H_{S_N+T_N}(\Lambda)\ \Big)_{u,v}& \leq & \frac \sigma N \esp \Big [ \big \| \nabla \ F^{(1)}_{u,v} (\mathbf X_N, \mathbf Y_N )    \big \|^2 \Big ],\\
		 \var \Big( \tau_N\big [       \big(      h_N^{(1,2)}\big)_{\substack{u,v\\ u',v' }}       \big] \ \Big) & \leq & \frac \sigma N \esp \Big [ \big \| \nabla \ F^{(2)}_{u,v,u',v'} (\mathbf X_N, \mathbf Y_N  )   \big \|^2 \Big].
\end{eqnarray*}
We define the set $\mathcal W$ of families $( \mathbf V,\mathbf W )$ of $N\times N$ Hermitian matrices, with $\mathbf V= (V_1\etc V_p)$, $\mathbf W=(W_1\etc W_q )$, of unit Euclidean norm in $\mathbf R^{(p+q)N^2}$. Then we have
\begin{eqnarray*}
		 \var   \Big(H_{S_N+T_N}(\Lambda)\ \Big)_{u,v} & \leq & \frac \sigma N \esp \bigg [ \underset{(\mathbf V,\mathbf W)\in \mathcal W} {\max} \Big | \frac d{dt}_{|t=0} \ F^{(1)}_{u,v} ( \mathbf X_N + t \mathbf V, \mathbf Y_N + t \mathbf W   )    \Big |^2 \bigg ],\\
		 \var \Big(  \tau_N\big [      \big(       h_N^{(1,2)}\big)_{\substack{u,v\\ u',v' }}       \big]  \ \Big) & \leq & \frac \sigma N \esp \bigg [\underset{(\mathbf V,\mathbf W)\in \mathcal W} {\max} \Big | \frac d{dt}_{|t=0} \ F^{(2)}_{u,v,u',v'} (  \mathbf X_N + t \mathbf V, \mathbf Y_N + t \mathbf W   )    \Big |^2 \bigg ].
\end{eqnarray*}
For all $ ( \mathbf V, \mathbf W )$ in $\mathcal W$, for all  selfadjoint $N \times N$ matrices $\mathbf A= (A_1\etc A_1)$, $\mathbf B=  (B_1\etc B_1 )$:
\begin{eqnarray*}
		\lefteqn{
	 \bigg | \frac d{dt}_{|t=0} \ F^{(1)}_{u,v}  ( \mathbf A  + t \mathbf V  , \mathbf B  + t \mathbf W  )    \bigg |^2} \\
	  & = & \Bigg | \frac d{dt}_{|t=0} \  \frac 1 N (\Tr_k \otimes \Tr_N) \bigg[ (\epsilon_{v,u} \otimes \mathbf 1_N) \Big ( \Lambda\UnN  - \sum_{j=1}^p a_j\otimes (A_j +t V_j ) -  \ \sum_{j=1}^q b_j\otimes (B_j  +t W_j )  \Big )^{-1}  \bigg ]  \Bigg|^2 \\
		 & = & \bigg | \frac 1 N (\Tr_k \otimes \Tr_N) \Big [ (\epsilon_{v,u} \otimes \mathbf 1_N)  ( \Lambda\UnN - \tilde S_N - \tilde T_N    )^{-1} \\
	  & & \ \ \ \  \ \ \ \times  \ \Big ( \sum_{j=1}^p a_j\otimes V_j + \sum_{j=1}^q b_j\otimes W_j   \Big )   ( \Lambda\UnN  - \tilde S_N - \tilde T_N    )^{-1} \Big ]  \bigg |^2.
\end{eqnarray*}
The Cauchy-Schwarz inequality for $\Tr_k \otimes \Tr_N$ (i.e. for $\Tr_{kN}$) gives
\begin{eqnarray*}
	 \left | \frac d{dt}_{|t=0} \ F^{(1)}_{u,v}  ( \mathbf A  + t \mathbf V  , \mathbf B  + t \mathbf W   )    \right |^2 & \leq & \frac 1 {N^2} \Big \| (\epsilon_{v,u} \otimes \mathbf 1_N)   ( \Lambda\UnN  - \tilde S_N - \tilde T_N    )^{-1}  \Big\|_e^2\\
	  & &  \ \ \ \times  \   \left \| \bigg( \sum_{j=1}^p a_j\otimes V_j  + \sum_{j=1}^q b_j\otimes W_j      \bigg)   ( \Lambda\UnN - \tilde S_N - \tilde  T_N    )^{-1}   \right \|_e^2.
\end{eqnarray*}
Using (\ref{CompNorm2}) to split Euclidean norms into the product of an operator norm and an Euclidean norm, we get:
\begin{eqnarray*}
	  	  	\lefteqn{
	 \left | \frac d{dt}_{|t=0} \ F^{(1)}_{u,v}  ( \mathbf A  + t \mathbf V , \mathbf B  + t \mathbf W   )    \right |^2} \nonumber \\
	  & \leq & \frac 1 {N^2}  \| \epsilon_{v,u} \otimes \mathbf 1_N \|_e^2 \    \|   ( \Lambda\UnN  - \tilde S_N - \tilde T_N    )^{-1}   \|^2 \bigg \| \sum_{j=1}^p a_j\otimes V_j  + \sum_{j=1}^q b_j\otimes W_j    \bigg \|^2_e  \nonumber\\
	  & \leq & \frac k {N} \   \| (\mathrm{Im } \ \Lambda)^{-1}  \|^4 \ \bigg \| \sum_{j=1}^p a_j\otimes V_j  + \sum_{j=1}^q b_j\otimes W_j   \bigg \|^2_e . 
\end{eqnarray*}
Remark that, since $ ( \mathbf V , \mathbf W ) \in \mathcal W $, the norm of the matrices $V_j$ and $W_j$ is bounded by one. Then we have the following:
\begin{eqnarray*}
	    \bigg \| \sum_{j=1}^p a_j\otimes V_j + \sum_{j=1}^q  b_j\otimes W_j + b_j^*\otimes W_j^*  \bigg \|_e  \leq                        \sum_{j=1}^p \| a_j\|_e   +2 \sum_{j=1}^q  \|b_j \|_e    \leq \sqrt{ k }\Big(\sum_{j=1}^p \| a_j\| +  \sum_{j=1}^q  \|b_j \|\Big).
\end{eqnarray*}
Hence we finally obtain an estimate of $\var  (H_{S_N+T_N}(\Lambda)\ )_{u,v} )$:
\begin{equation}
		 \var \Big(H_{S_N+T_N}(\Lambda)\ \Big)_{u,v}  \leq \frac {k^2\sigma} {N^2}  \Big ( \sum_{j=1}^p \left \| a_j \right \|+\sum_{j=1}^q \|b_j\|  \Big)^2 \ \| (\mathrm{Im } \ \Lambda)^{-1} \|^4. \label{EstmTheta9}
\end{equation}
We obtain a similar estimate for $ \var \Big(  \tau_N\big [        \big(     h_N^{(1,2)}\big)_{\substack{u,v\\ u',v' }}       \big]  \ \Big)$. The partial derivative of $F^{(2)}_{u,v,u',v'}$ gives two terms: $\forall ( \mathbf V, \mathbf W ) \in \mathcal W$, $\forall (\mathbf A,\mathbf B)\in \MN^{p+q}$
\begin{eqnarray*}
		\lefteqn{
	\frac d{dt}_{|t=0} \ F^{(2)}_{u,v,u',v'}( \mathbf A + t \mathbf V ,\mathbf B  + t \mathbf W  )    }\\
	   & = &   \frac 1 N (\Tr_k \otimes \Tr_N) \Big [ (\epsilon_{v,u} \otimes \mathbf 1_N) (\Gamma \UnN -\tilde T_N)^{-1}  \Big (  \sum_{j=1}^q b_j\otimes W_j    \Big )\\
	  	 & &  \ \ \ \ \  \ \ \ \ \ \times \ (\Gamma \UnN -\tilde T_N)^{-1} (\epsilon_{u',v'}\UnN)   ( \Lambda\UnN  - \tilde S_N - \tilde T_N    )^{-1} \\
	  	& &  \ \ \ \ \  + \ (\epsilon_{v,u} \otimes \mathbf 1_N) (\Gamma \UnN -\tilde T_N)^{-1} (\epsilon_{u',v'}\UnN)   ( \Lambda\UnN  - \tilde S_N - \tilde T_N    )^{-1}\\
	  	 &  &   \ \ \ \ \  \ \ \ \ \ \times \  \Big ( \sum_{j=1}^p a_j\otimes V_j^{(N)} + \sum_{j=1}^q b_j\otimes W_j^{(N)}  \Big )   ( \Lambda\UnN  - \tilde S_N - \tilde T_N    )^{-1}   \Big ].
\end{eqnarray*}
We then get the following:
\begin{eqnarray*}
		\lefteqn{
	 \left | \frac d{dt}_{|t=0} \ F^{(2)}_{u,v,u',v'}( \mathbf A  + t \mathbf V ,\mathbf B  + t \mathbf W  )    \right |^2} \\
	   & \leq &  \frac {k^2} {N} \Big ( \sum_{j=1}^p \left \| a_j \right \|+\sum_{j=1}^q \|b_j\|  \Big)^2  \ \| (\mathrm{Im } \ \Gamma)^{-1}\|^2 \ \| (\mathrm{Im } \ \Lambda)^{-1}\|^2 \Big ( \| (\mathrm{Im } \ \Lambda)^{-1}\| + \| (\mathrm{Im } \ \Gamma)^{-1}\| \Big)^2.
\end{eqnarray*}
Hence we have
\begin{eqnarray}
	\lefteqn{ \var \Big(    \tau_N\big [      \big (       h_N^{(1,2)} \big)_{\substack{u,v\\ u',v' }}       \big]   \ \Big)}\nonumber\\
	& \leq& \frac {k^2\sigma} {N^2} \Big ( \sum_{j=1}^p \left \| a_j \right \|+\sum_{j=1}^q \|b_j\|  \Big)^2   \| (\mathrm{Im } \ \Gamma)^{-1}\|^2 \ \| (\mathrm{Im } \ \Lambda)^{-1}\|^2 \ \Big (  \| (\mathrm{Im } \ \Gamma)^{-1}\| + \| (\mathrm{Im } \ \Lambda)^{-1}\| \Big)^2.\label{EstmTheta10}
\end{eqnarray}
We then obtain as desired, by (\ref{EstmTheta3}), (\ref{EstmTheta9}) and  (\ref{EstmTheta10}):
\begin{eqnarray*}
		 \left \| \Theta_N(\Lambda, \Gamma) \right \|&  \leq &  k^{5/2} \sum_{j=1}^p \| a_j\|^2 \bigg (     \sum_{u,v=1}^k     \var    \big(H_{S_N+T_N}(\Lambda)\big) _{u,v}   \ \underset{\substack{u,v,u',v' } } {\max} \var \Big (  \tau_N\big[ \big( h_N^{(1,2)} \big)_{\substack{u,v\\ u',v' }} \big] \ \Big ) \ \bigg )^{1/2}\\
		  & \leq & \frac{ c} { N^{2}}     \left \| (\mathrm{Im } \ \Gamma)^{-1} \right\| \left \| (\mathrm{Im } \ \Lambda)^{-1} \right\|^3 \ \Big (  \| (\mathrm{Im } \ \Gamma)^{-1}\| + \| (\mathrm{Im } \ \Lambda)^{-1}\| \Big ), 
\end{eqnarray*}
where $c=  {k^{9/2}\sigma} \sum_{j=1}^p \|a_j\|^2       \Big ( \sum_{j=1}^p \left \| a_j \right \|+\sum_{j=1}^q \|b_j\|  \Big)^2 . $

\end{proof}

\subsection{Proof of Theorem 5.1}\label{SecProofSubMA}

By (\ref{FalseL1}), for all $\Lambda$ in $\Mk^+$, the matrix $\Lambda -   \mathcal R_sÊ\big (  G_{S_N+T_N} (\Lambda) \ \big )$ is in $\Mk^+$ and then it makes sense to choose $\Gamma = \Lambda -   \mathcal R_sÊ\big (  G_{S_N+T_N} (\Lambda) \ \big )$ in Equation (\ref{SDin}). We obtain for all $\Lambda $ in $\Mk^+$,
		$$G_{S_N+T_N} (\Lambda) =  G_{T_N} \Big ( \Lambda - \mathcal R_s \big ( G_{S_N+T_N}(\Lambda) \ \big ) \ \Big) + \Theta_N(\Lambda),$$
where $\Theta_N(\Lambda) = \Theta_N\Big(\Lambda,  \Lambda -   \mathcal R_sÊ\big (  G_{S_N+T_N} (\Lambda) \ \big ) \ \Big)$ is analytic in $k^2$ complex variables. Recall that by (\ref{FalseL2}), we have $\big \|  \big (  \Lambda -   \mathcal R_sÊ\big (  G_{S_N+T_N} (\Lambda) \ \big )\ \big)^{-1} \big\| \leq \|(\Lambda)^{-1}\|$, which gives (when replacing $c$ in (\ref{Theta}) by $c/2$) the expected estimate of $\Theta_N(\Lambda)$.

\section{Proof of Estimate (3.9)}
Let $(\mathbf X_N, \mathbf Y_N, \mathbf x, \mathbf y)$ be as in Section \ref{PartSheme}. We assume that $\big (\mathbf x, \mathbf y, (\mathbf Y_N)_{N\geq 1} \big )$ are realized in a same $\mathcal C^*$-probability space $(\mathcal A, .^*, \tau, \| \cdot \|)$ with faithful trace, where
\begin{itemize}
\item the families $\mathbf x$, $\mathbf y$, $\mathbf Y_1$, $\mathbf Y_2, \dots, \mathbf Y_N, \dots$ are free,
\item for any polynomials $P$ in $q$ non commutative indeterminates $\tau[P(\mathbf Y_N)]:=\tau_N[P(\mathbf Y_N)]$.
\end{itemize}
Consider $L$ a degree one selfadjoint polynomial with coefficients in $\Mk$. Define the Stieltjes transform of $L_N= L(\mathbf X_{N}, \mathbf Y_{N} )$ and $\ell_N= L(\mathbf x, \mathbf Y_N )$: for all $\lambda \in \mathbb C^+=\big \{ z \in \mathbb C \big | \ \mathrm{Im} \ z>0\big\}$,
\begin{eqnarray}
		g_{L_N} ( \lambda) & = & \esp \bigg [ \ttrN \Big[   \big (\lambda \unun_N -  L_N \ \big)^{-1} \  \Big] \bigg ],\\
		g_{\ell_N}( \lambda) & = &   \ttr \Big[  \big ( \lambda \unun-  \ell_N \ \big)^{-1} \  \Big]  .
\end{eqnarray}
One can always write $L_N=a_0\otimes \mathbf 1_N +S_N+T_N$, $\ell_N=a_0\otimes \mathbf 1 +s +T_N$, where 
		$$S_N  =  \sum_{j=1}^p a_j \otimes X_j^{(N)}, \ s = \sum_{j=1}^p a_j \otimes x_j,  \ T_N =  \sum_{j=1}^q b_j \otimes Y_j^{(N)},$$
and $a_0\etc a_p,b_1 \etc b_q$ are Hermitian matrices in $\Mk$. Define the $\Mk$-valued Stieltjes transforms of $S_N+T_N$ and $s+T_N$: for all $\Lambda\in \Mk^+=\big\{ \Lambda\in \Mk \ \big | \ \mathrm{Im} \ \Lambda>0 \big\}$,
\begin{eqnarray*}
  		G_{S_N+T_N}(\Lambda)   & =& \esp \bigg[ \traceN \Big[\big(\Lambda \otimes \mathbf 1_N - S_N  - T_N\big)^{-1}\Big] \bigg ],\\
  		G_{s+T_N}(\Lambda)   & =& \trace \Big[\big(\Lambda \otimes \mathbf 1 - s  - T_N\big)^{-1}\Big].
\end{eqnarray*}
Then one has: for all $\lambda$ in $\mathbb C^+$
		$$g_{L_N}(\lambda) = \tau_k \Big [ \ G_{S_N+T_N}(\lambda \mathbf 1_k - a_0) \ \Big], \ g_{\ell_N}(\lambda) = \tau_k \Big [ \ G_{s+T_N}(\lambda \mathbf 1_k - a_0) \ \Big].$$
By Proposition \ref{PropRtr}, for any $\Lambda\in \Mk^+$, one has
		$$G_{s+T_N}(\Lambda) = G_{T_N} \Big ( \Lambda   - \mathcal R_s \big ( G_{s+T_N} (\Lambda) \ \big ) \ \Big ).$$
On the other hand, since the matrices of $\mathbf Y_N$ are deterministic, we can apply Theorem \ref{SubMATh} with $\sigma=1$
		$$G_{S_N+T_N} (\Lambda)  = G_{T_N} \Big ( \Lambda   - \mathcal R_s \big ( G_{S_N+T_N} (\Lambda) \ \big ) \ \Big ) + \Theta_N(\Lambda),$$
where $\left \| \Theta_N(\Lambda) \right \|   \leq   \frac {c} {N^2}  \left \| (\mathrm{Im } \ \Lambda)^{-1} \right\|^5$ for a constant $c>0$. Define 
		$$\Omega_\eta\toN=\Big \{ \Lambda \in \Mk^+ \Big | \ \|(\mathrm{Im} \ \Lambda)^{-1}\| < N^\eta\Big\}.$$
Then for $\eta<1/3$, there exists $N_0$ such that for all $N\geq N_0$ and for any $\Lambda$ in $\Omega_\eta\toN$, one has
\begin{eqnarray*}
	\kappa(\Lambda):=\|\Theta_N(\Lambda)\| \ \|(\mathrm{Im} \ \Lambda)^{-1}\| \ \sum_{j=1}^p \|a_j\|^2   \leq  \frac {c } {N^2}  \|(\mathrm{Im} \ \Lambda)^{-1}\|^6 \leq c N^{6\eta-2}\leq \frac 1 2.
\end{eqnarray*}
Then by Proposition \ref{UniqSol} with $(t,G,\Theta, \Omega, \varepsilon)=(T_N, G_{S_N+T_N}, \Theta_N, \Omega^{(N)}_\eta,1/2)$, one has
\begin{eqnarray*}
		\| G_{s+T_N}(\Lambda) - G_{S_N+T_N}(\Lambda)\| & \leq & \Big (1+2\sum_{j=1}^p \|a_j\|^2 \ \| (\mathrm{Im } \Lambda)^{-1}\|^2\Big) \ \| \Theta(\Lambda)\|\\
				& \leq &  c \Big (1+2\sum_{j=1}^p \|a_j\|^2 \ \| (\mathrm{Im } \ \Lambda)^{-1}\|^2\Big) \ \frac { \|  (\mathrm{Im } \ \Lambda)^{-1}\|^5}{N^2}.
\end{eqnarray*}
Hence for every $\varepsilon>0$, there exist $N_{0}$ and   $\gamma$ such that for all $N \geq N_{0}$, for all $\lambda$ in $\mathbb C$ such that $\varepsilon\leq (\mathrm{Im } \  \lambda)^{-1} \leq N^{\gamma}$, one has
\begin{equation}
		| g_{L_N}(\lambda) - g_{\ell_N}(\lambda) | \leq \| G_{s+T_N}(\lambda \mathbf 1_k - a_0) - G_{S_N+T_N}(\lambda \mathbf 1_k - a_0)\| \leq  \frac {c}{N^{2}} ( \mathrm{ Im} \  \lambda)^{-7},
\end{equation}
where $c$ denotes now the constant $c = {k^{9/2}} \sum_{j=1}^p \|a_j\|       \Big ( \sum_{j=1}^p \left \| a_j \right \|+\sum_{j=1}^q \|b_j\|  \Big)^2 \Big (\varepsilon^{-2}+2\sum_{j=1}^p \|a_j\|^2 \Big)$.

\section[Proof of Step 2]{Proof of Step 2: An intermediate inclusion of spectrum}\label{App12}
For a review on the theory of $\mathcal C^*$-algebras, we refer the readers to \cite{CON} and \cite{BO}. Notably, Appendix A of the second reference contains facts about ultrafilters and ultraproducts that are used in this section.
\\
\\Let $\big (\mathbf x, \mathbf y, (\mathbf Y_N)_{N\geq 1} \big )$ be as in Section \ref{PartSheme}. We assume that these non commutative random variables are realized in the same $\mathcal C^*$-probability space $(\mathcal A, .^*, \tau, \| \cdot \|)$ with faithful trace, where
\begin{itemize}
\item the families $\mathbf x$, $\mathbf y$, $\mathbf Y_1$, $\mathbf Y_2, \dots, \mathbf Y_N, \dots$ are free,
\item for any polynomials $P$ in $q$ non commutative indeterminates $\tau[P(\mathbf Y_N)]:=\tau_N[P(\mathbf Y_N)]$.
\end{itemize}

\noindent A consequence of Voiculescu's theorem and of Shlyakhtenko's Theorem \ref{ShlTh} in Appendix \ref{App1} is that for all polynomials $P$ in $p+ q$ non commutative indeterminates,
\begin{eqnarray}
		\tau [ P (\mathbf x , \mathbf Y_N  )  ] & \underset{N\rightarrow \infty }\longrightarrow & \tau[P(\mathbf x , \mathbf y  )  ],\label{CvOmega1}\\
		\| P (\mathbf x , \mathbf Y_N  ) \| &  \underset{N\rightarrow \infty }\longrightarrow &  \| P(\mathbf x , \mathbf y  ) \|.\label{CvOmega2}
\end{eqnarray}

\noindent In order to prove Step 2, it remains to show that (\ref{CvOmega2}) still holds when the polynomials $P$ are $\Mk$-valued. This fact is a folklore result in $\mathcal C^*$-algebra theory, we give a proof for readers convenience. We need first the two following lemmas.

\begin{Lem}\label{LemCs1} Let $\mathcal A$ and $\mathcal B$ be unital $\mathcal C^*$-algebra. Let $\pi : \mathcal A \rightarrow \mathcal B$ be a morphism of unital $^*$-algebra. Then $\pi$ is contractive.
\end{Lem}
\begin{proof} It is easy to see that for any $a$ in $\mathcal A$, the spectrum of $\pi(a)$ is included in the spectrum of $a$ (since $\lambda \mathbf 1_\mathcal A - a$ invertible implies that $\lambda \mathbf 1_\mathcal A - \pi(a)$ is also invertible). Hence we get that for all $a$ in $\mathcal A$
\begin{equation*}	\| \pi(a)\|^2 = \| \pi(a^* a)\| \leq \| a^* a\| = \|a\|^2.
\end{equation*}
\end{proof}
\begin{Lem}\label{LemCs2} Let $\mathcal A$ be a unital $\mathcal C^*$-algebra. Then for any integer $k\geq 1$, there exists a unique $\mathcal C^*$-algebra structure on $\Mk\otimes \mathcal A$ compatible with the structure on $\mathcal A$. In particular, if $\mathcal A$ is a $\mathcal C^*$-probability space equipped with a faithful tracial state $\tau$, then $\Mk\otimes \mathcal A$ is a $\mathcal C^*$-probability space with trace $(\tau_k \otimes \tau)$ and norm $\| \cdot \|_{\tau_k \otimes \tau}$, where $\tau_k$ is the normalized trace on $\Mk$ and $\| \cdot \|_{\tau_k \otimes \tau}$ is given by Formula {\rm(\ref{DefNorm})}.
\end{Lem}
\begin{proof}[Sketch of the proof] For the existence we consider the norm given by the spectral radius. The uniqueness follows from Lemma \ref{LemCs1}.
\end{proof}

\begin{Prop}\label{PropBloMa} Let $k\geq 1$ be an integer. For all $N\geq 1$, let $\mathbf z_N=(z_1\toN \etc z_p\toN)$, respectively $\mathbf z=(z_1  \etc z_p)$, be self-adjoint non commutative random variables in a $\mathcal C^*$- probability space $(\mathcal A_N, .^*, \tau_N, \| \cdot \|_{\tau_N})$, respectively $(\mathcal A, .^*, \tau, \| \cdot \|_{\tau})$. Assume that the traces $\tau_N$ and $\tau$ are faithful (hence the notation for the norms) and that for any polynomial $P$ in $p$ non commutative indeterminates,
\begin{eqnarray}
		\tau_N [ P (\mathbf z_N ) ] & \underset{N\rightarrow \infty }\longrightarrow & \tau[P(\mathbf z)   ],\label{CvOmega3}\\
		\| P (\mathbf z_N  ) \|_{\tau_N} &  \underset{N\rightarrow \infty }\longrightarrow &  \| P(\mathbf z) \|_\tau.\label{CvOmega4}
\end{eqnarray}
Then for any polynomial $P$ in $p$ non commutative indeterminates with coefficients in $\Mk$, 
\begin{eqnarray}
	 \| P (\mathbf z_N  ) \|_{\tau_k \otimes \tau_N} & \underset{N\rightarrow \infty }\longrightarrow & \| P(\mathbf z) \|_{\tau_k \otimes \tau}.\label{CvOmega2bis}
\end{eqnarray}
\end{Prop}
\noindent We abuse notation and write with the same symbol the traces in $\Mk$ and $\mathcal A_N$ when $N=k$. There is no danger of confusion.
\begin{proof}
For any positive integer $k$ and any ultrafilter $\mathcal U$ on $\mathbb N$, we define the ultraproduct 
		$$\mathfrak A^{(k)} = \prod^{\mathcal U} \mathrm{M}_k(\mathbb C) \otimes \mathcal A_N,$$
which is the quotient of 
		$$ \bigg \{ \  ( a_N)_{N\geq 1} \ \bigg | \ \forall N\geq 1, \ a_N \in \mathrm{M}_k(\mathbb C) \otimes \mathcal A_N \textrm{ and } \underset{N\geq 1}\sup \| a_N \| < \infty \ \bigg \},$$
by 
		$$ \bigg \{ \  ( a_N)_{N\geq 1} \ \bigg | \ \forall N\geq 1, \ a_N \in \mathrm{M}_k(\mathbb C) \otimes \mathcal A_N \textrm{ and } \underset{N\rightarrow \mathcal U }\lim \| a_N \| = 0 \ \bigg \}.$$
The algebra $\mathfrak A^{(k)}$ is a $\mathcal C^*$-algebra whose norm $\| \cdot \|_{\mathfrak A^{(k)}}$ is given by: for all $a$ in $\mathfrak A^{(k)}$, equivalence class of $(a_N)_{N\geq 1} $
		$$\| a \|_{\mathfrak A^{(k)}} = \underset{N\rightarrow \mathcal U} \lim \| a_N\|_{\tau_k\otimes \tau_N}.$$
Furthermore $\mathfrak A^{(k)}$ is a $\mathcal C^*$-probability space which can be identified with M$_k(\mathbb C) \otimes \mathfrak A^{(1)}$. The trace $\tilde \tau$ on $\mathfrak A^{(1)}$ is given by: for all $a$ in $\mathfrak A^{(1)}$, equivalence class of $(A_N)_{N\geq 1} $, one has
		$$\tilde \tau[a] = \underset{N\rightarrow \mathcal U }\lim \tau[A_N].$$
If the classical limit as $N$ goes to infinity exists, then the trace of $a$  does not depends on the ultrafilter $\mathcal U$ and is given by the limit. The trace on $\mathfrak A^{(k)}$ is $(\tau_k \otimes \tilde \tau)$. Notice that $(\tau_k\otimes \tilde \tau)$ on $\mathfrak A^{(k)}$ is not faithful in general, which implies that the norm $\| \cdot \|_{\mathfrak A^{(k)}}$ and the norm $\| \cdot \|_{\tau_k\otimes \tilde \tau}$ given by $(\tau_k\otimes \tilde \tau)$ with Formula (\ref{DefNorm}) are not equal on the whole $\mathcal C^*$-algebra.
\\At last, we can equip $\mathfrak A^{(k)}$ with a structure of operator-valued $\mathcal C^*$-probability space. Define the unital sub-algebra $\mathcal B$ of $\mathfrak A^{(k)}$ as the set
		$$ \Big \{ \ b\otimes \mathbf 1_{\mathfrak A^{(1)}}  \ \Big | \ b \in \mathrm{M}_k (\mathbb C) \ \Big\} \subset \mathfrak A^{(k)}.$$ 
The conditional expectation in $\mathfrak A^{(k)}$ is given by $ (\textrm{id}_k \otimes \tilde \tau):\mathfrak A^{(k)}\rightarrow \mathcal B$. 
\\
\\For $j=1 \etc p$, we denote by $\tilde z_j$ in $\mathfrak A^{(1)}$ the equivalence class of the sequence $ (  z_j\toN  )_{N\geq 1}$. We have by definition of $\mathfrak A^{(k)}$: for all polynomial $P$ in $p+2q$ non commutative indeterminates with coefficients in M$_k(\mathbb C)$,
\begin{eqnarray*}
		\| P( \mathbf z_N) \|_{\tau_N} \underset{N \rightarrow \mathcal U}  \longrightarrow \| P(\tilde {\mathbf z})\|_ {\mathfrak A^{(k)}}
\end{eqnarray*}
Let $\mathcal C^*(\tilde {\mathbf z}) $ be the sub-algebra spanned by $\tilde {\mathbf z}= (\tilde z_1 \etc \tilde z_p)$ in $\mathfrak A^{(1)}$ and let $\mathcal C^*( {\mathbf z}) $ be the sub-algebra spanned by $  {\mathbf z} $ in $\mathcal A$. Then by (\ref{CvOmega4}), the $\mathcal C^*$-algebras $\mathcal C^*(\tilde {\mathbf z}) $ and $\mathcal C^*(  {\mathbf z}) $ are isomorphic. Hence we get an isomorphism of the $^*$-algebras M$_k(\mathbb C) \otimes \mathcal C^*(\tilde {\mathbf z}) $ and M$_k(\mathbb C) \otimes \mathcal C^*({\mathbf z}) $, and so an isomorphism of the $\mathcal C^*$-algebras by Lemma \ref{LemCs1}. Hence, for all polynomial $P$ in $p+2q$ non commutative indeterminates with coefficients in M$_k(\mathbb C)$,
\begin{eqnarray*}
		\| P(\tilde {\mathbf z})\|_ {\mathfrak A^{(k)} }= \|P(\mathbf z)\|_{ \tau_k\otimes \tilde \tau}
\end{eqnarray*}
Hence we get
\begin{eqnarray*}
		\| P( \mathbf z_N) \|_{\tau_k \otimes \tau_N} \underset{N \rightarrow \mathcal U}  \longrightarrow \|P(\mathbf z)\|_{ \tau_k\otimes \tilde \tau}
\end{eqnarray*}
for all ultrafilter $\mathcal U$. Then the convergence holds when $N$ goes to infinity.
\end{proof}

\begin{proof}[Proof of Step 2]
Let $L$ be a selfadjoint degree one polynomial in $p+ q$ non commutative indeterminates with coefficients in M$_k(\mathbb C)$. Define $\ell_N=L(\mathbf x, \mathbf Y_N)$ and $\ell=L(\mathbf x, \mathbf y)$. Then by Proposition \ref{PropBloMa}, for all commutative polynomials $P$, one has
		$$\| P(\ell_N)  \|_{\tau_k\otimes \tau} \underset{N\rightarrow \infty} \longrightarrow \| P(\ell)  \|_{\tau_k\otimes \tau}.$$
The convergence extends to continuous function on the real line and then, with an appropriate choice of test functions, Step 2 follows.
\end{proof}

\section{Proof of Step 3: from Stieltjes transforms to spectra}\label{App2}
\noindent Let $\mathbf X_N, \mathbf Y_N, \mathbf x$ and $\mathbf y$ be as in Section \ref{PartSheme}. As before $\mathbf x, \mathbf y,$ and $\mathbf Y_N$ are assumed to be realized in a same $\mathcal C^*$-probability space $(\mathcal A, .^*, \tau, \| \cdot \|)$ with faithful trace. Let $L$ be a selfadjoint degree one polynomial with coefficients in $\Mk$. 
\\
\\For any function $f : \mathbb R \rightarrow \mathbb R$ and any Hermitian matrix $A$ with spectral decomposition $A= U\mathrm{diag} \ (\lambda_1\etc \lambda_K)U^*$, with $U$ unitary, we set the Hermitian matrix $f(A) = U\mathrm{diag} \ (f(\lambda_1) \etc f(\lambda_K))U^*$. For any function $f : \mathbb R \mapsto \mathbb R$, we set
		$$D_N(f) = (\tau_k\otimes \tau_N)\big[f (  L(\mathbf X_N, \mathbf Y_N)  ) \big ].$$
By Step 2, for all $\varepsilon>0$, there exists $N_0\geq 1$ such that for all $N\geq N_0$, one has
\begin{equation*}
		\sp\Big (  \ L(\mathbf x, \mathbf Y_N) \ \Big) \subset \sp\Big (  \ L(\mathbf x, \mathbf y) \ \Big) + (-\varepsilon, \varepsilon).
\end{equation*}
Hence, for any function $f$ vanishing on a neighborhood of the spectrum of $L(\mathbf x, \mathbf y)$, there exists $N_0\geq 1$ such that for all $N\geq N_0$, the function $f$ actually vanishes on a neighborhood of the spectrum of $L(\mathbf x, \mathbf Y_N)$. In particular, with $\mu_N$ (respectively $\nu_N$) denoting the empirical eigenvalue distribution of $L_N=L(\mathbf X_N, \mathbf Y_N)$ (respectively $\ell_N=L(\mathbf x, \mathbf Y_N)$), one has
\begin{equation}\label{RemindSp}
		\esp\Big [ D_N(f) \Big ]  = \esp\Big [ \ \int f \ \mathrm{d}\mu_N \ \Big] = \esp\Big [ \ \int f \ \mathrm{d}\mu_N \ \Big]  - \ \int f \ \mathrm{d}\nu_N.
\end{equation}
\noindent Furthermore, by Estimate (\ref{EqPurp}), with the Stieltjes transforms of $L_N$ and of $\ell_N$ defined by: for all $\lambda$ in $\mathbb C^+$
\begin{eqnarray*}
		g_{L_N}(\lambda) & = & \esp \bigg [ (\tau_k \otimes \tau_N) \Big[ \ \big(\lambda \mathbf 1_k \otimes \mathbf 1_{N} - L_N \ \big)^{-1} \  \Big]  \ \bigg] =  \esp\Big [ \ \int  \frac 1{\lambda - t} \mathrm{d}\mu_N(t) \Big ]  \\
		g_{\ell_N}(\lambda) & = & (\tau_k \otimes \tau) \Big[ \ \big(\lambda \mathbf 1_{k}\otimes \mathbf 1  -  \ell_N \ \big)^{-1} \  \Big] = \int  \frac 1{\lambda - t} \mathrm{d}\nu_N(t) ,
\end{eqnarray*}
we have shown that: for any $\varepsilon>0$ and $A>0$, there exist $N_{0}, c, \eta, \gamma, \alpha>0$ such that for all $N \geq N_{0}$, for all $\lambda$ in $\mathbb C$ such that $\varepsilon\leq (\mathrm{Im } \  \lambda)^{-1} \leq N^{\gamma}$ and $|$Re $\lambda|\leq A$
\begin{equation} \label{RemindEstim}
		| g_{L_N}(\lambda) - g_{\ell_N}(\lambda) | \leq  \frac {c}{N^{2}} ( \mathrm{ Im} \  \lambda)^{-\alpha}.
\end{equation}
\noindent With (\ref{RemindSp}) and (\ref{RemindEstim}) established, it is easy to show with minor modifications of \cite[Lemma 5.5.5]{AGZ} the following result.
\begin{Lem} \label{555}For every smooth function $f: \mathbb R \rightarrow \mathbb R$ non negative, compactly supported and vanishing on a neighborhood of the spectrum of $ L({\mathbf x}, \mathbf y) $, there exists a constant such that for all $N$ large enough
\begin{equation}\label{MeanVan}
		\Big| \esp \big[D_N(f) \ \big] \ \Big| \leq \frac c{N^{2}}.
\end{equation}
\end{Lem} 
\noindent To get an almost sure control of $D_N(f)$, we use the fact that the entries of the matrices $\mathbf X_N$ satisfy a concentration inequality.
\begin{Lem} \label{FromStoS}With $f$ as in Lemma \ref{555}, there exists $\kappa>0$ such that, almost surely
\begin{equation}
		N^{1+\kappa} D_N(f) \underset{N\rightarrow\infty}{\longrightarrow}0.
\end{equation}
\end{Lem}

\begin{proof} The law of the random matrices satisfying a Poincar\'e inequality with constant $\frac 1 N$ and $L$ being a polynomial of degree one, for all Lipschitz function $\Psi: $ M$_{kN}( \mathbb C) \mapsto \mathbb R$, by \cite[Lemma 5.2]{GUI} one has:
\begin{equation}\label{Concentration}
		\mathbb  P \Big ( \ \big | \Psi (L_N  ) - \esp\big[ \Psi (L_N ) \ \big] \ \big | \geq \delta \Big) \leq K_1 e^{-K_2 \frac {\sqrt N \delta} {|\Psi|_{\mathcal L}} },
\end{equation}
where $K_1, K_2$ are positive constants and $|\Psi|_{\mathcal L}= \underset{A \neq  B \in \textrm{M}_{kN}( \mathbb C) }{\textrm{ sup }} \frac{ | \Psi(A)-\Psi( B)|}{\| A - B\|_e}$. Recall that the Euclidean norm $\| \cdot \|_e$ of a matrix $A=(a_{i,j})_{i,j=1}^{kN}$ is given by
		$$ \|A\|_e = \sqrt{ \sum_{i,j=1}^{kN} |a_{i,j}|^2 }.$$
For any Hermitian matrices $A$ in M$_{kN}( \mathbb C)$ and any function $f : \mathbb R \rightarrow \mathbb R$, we set
\begin{equation} 
		\Phi^{(f)}_N( A) = (\tau_k\otimes \tau_N)\big[ f(A) \ \big ].
\end{equation}
For all smooth function $f : \mathbb R \rightarrow \mathbb R$, $N\geq 1$ and $0<\kappa<\frac 1 2$, we define
\begin{equation}
		\mathcal B^{(f)}_{N,\kappa} = \Bigg \{ A \in \textrm{M}_{kN}(\mathbb C) \ \Big | \ A \textrm{ is Hermitian and }\ \Big | \Phi^{(f'^2)}_N(A) \Big | \leq \frac 1 {N^{4\kappa}} \Bigg\},
\end{equation}
and denote $\rho^{(f)}_{N, \kappa} = |(\Phi_N^{(f)})_{|\mathcal B_{N,\kappa}}|_{\mathcal L}.$ Define $\Psi_N^{(f)}: $ M$_{kN}( \mathbb C) \mapsto \mathbb R$ by: $\forall A \in $  M$_N( \mathbb C)$
\begin{equation}
		\Psi_N^{(f)}( A) =  \underset{ B \in \mathcal B^{(f)}_{N,\kappa}}{\sup} \Big\{ \Phi_N^{(f)}( B) -\rho^{(f)}_{N, \kappa} \ \|A-B\|_2 \ \Big\} ,
\end{equation}
and denote $\tilde D_N(f) = \Psi_N^{(f)}(   L_N  ) $. By \cite[Proof of Lemma 5.9]{GUI}, $\Psi_N^{(f)}$ coincides with $\Phi_N^{(f)}$ on $\mathcal B^{(f)}_{N,\kappa}$ and is Lipschitz with constant $|\Psi_N^{(f)}|_\mathcal L \leq\rho^{(f)}_{N, \kappa}$.
\\
\\For all Hermitian matrices $A$ in M$_{kN}(\mathbb C)$, $M$ in M$_{kN}(\mathbb C)$ and $n\geq 1$, one has $\frac d{dt}_{|t=0}(A+tM)^n = \sum_{m=0}^n A^{m}MA^{n-m-1}$ and then $\frac d{dt}_{|t=0}\ttrN[(A+tM)^n] = \ttrN[nA^{n-1}M]$. So for all polynomials $P$, one has D$_A\Phi^{(P)}_N(M) = \ttrN[P'(A)M]$. Hence, by density of polynomials, for any smooth function $f:\mathbb R \rightarrow \mathbb R$ one has D$_A\Phi^{(f)}_N(M) = \ttrN[ {f'}(A)M]$. By the Cauchy-Schwarz inequality, we get
\begin{eqnarray*}
		\big | \textrm{D}_A\Phi^{(f)}_N(M) \big |^2 & =  & |  \ttrN[f'(A)M] |^2\\
				& \leq & \ttrN[f'(A)^2] \times \ttrN[M^*M]\\
				& = & \Phi^{({f'}^2)}_N(M)\times \frac { \| A \|_e }{kN}.
\end{eqnarray*}
Then, for any smooth function $f$, one has
\begin{equation}\label{Step3Rho}
		\rho^{(f)}_{N, \kappa}  \ \leq \frac 1{\sqrt{kN}} \| \ (\Phi^{(f'^2)}_N) _{|\mathcal B^{(f)}_{N,\kappa}} \ \|_{\infty}^{1/2},
\end{equation}
where $\|\cdot \|_\infty$ denotes the supremum of the considered function on the set of $kN\times kN$ Hermitian matrices. Hence we get that $|\Psi_N^{(f)}|_\mathcal L \leq \rho^{(f)}_{N,\kappa} \leq \frac 1 {\sqrt k} {N^{-1/2-2\kappa}}$.
\\
\\We fix $f$ a smooth function, non negative, compactly supported and vanishing on a neighborhood of the spectrum of $ L({\mathbf x}, \mathbf y) $. By the Tchebychev inequality
\begin{eqnarray}
		\mathbb P( L_N \notin \mathcal B^{(f)}_{N,\kappa}) & = & \mathbb P\Big(  D_N(f'^2) \geq \frac 1 {N^{4\kappa}} \Big)\leq  N^{4\kappa} \esp\Big[  D_N(f'^2) \Big ] \leq \frac c {N^{2-4\kappa}},\label{By1}
\end{eqnarray}
where we have used Lemma \ref{555} ($f'^2$ also vanishes in a neighborhood of the spectrum of $ L({\mathbf x}, \mathbf y) $). Moreover, since $\Psi^{(f)}_N$ and $\Phi^{(f)}_N$ are equals in $\mathcal B^{(f)}_{N,\kappa}$ and $\| \Psi^{(f)}_N\|_\infty \leq \|\Phi^{(f)}_N\|_\infty$,
\begin{equation}
		\Big| \esp\big[\tilde D_N(f) - D_N(f)\big] \ \Big| \leq \|\Phi^{(f)}_N\|_\infty \mathbb P( L_N \notin \mathcal B^{(f)}_{N,\kappa})\leq  \|\Phi^{(f)}_N\|_\infty\frac c {N^{2-4\kappa}} \label{By2}
\end{equation}
\\
\\Now, by (\ref{Concentration}) applied to $\Psi^{(f)}_N$: for all $\delta>0$
\begin{eqnarray*}
		\lefteqn{ \mathbb P\bigg ( \Big |  D_N(f) - \esp \big[D_N(f)\ \big] \ \Big| >\frac {\delta}{N^{1+\kappa}} \ \textrm{ and } \ L_N \in \mathcal B^{(f)}_{N,\kappa} \bigg) } \\
		& \leq & P\bigg ( \Big |  \tilde D_N(f) - \esp \big[\tilde D_N(f)\ \big] \ \Big| >\frac {\delta}{N^{1+\kappa} }\ - \ \Big| \esp\big[\tilde D_N(f) - D_N(f)\big] \ \Big| \ \bigg)\\
		& \leq & K_1 \exp\bigg(     -  \sqrt k K_2  N^{\kappa} (\delta- \ \Big| \esp\big[\tilde D_N(f) - D_N(f)\big] \ \Big|)    \bigg)
\end{eqnarray*}
By (\ref{By1}), (\ref{By2}), Lemma \ref{555} and the Borel-Cantelli lemma, $D_N(f)$ is almost surely of order $N^{1+\kappa}$ at most.
\end{proof}

\begin{Prop}
For every $\varepsilon >0$, there exists $N_0$ such that for $N\geq N_0$
\begin{equation}
		\sp \Big ( \ L(\mathbf X_N, \mathbf Y_N) \ \Big ) \ \subset \sp \Big ( \ L(\mathbf x, \mathbf y ) \ \Big ) +(-\varepsilon, \varepsilon)
\end{equation}
\end{Prop}
\begin{proof}
By (\ref{Th1wCV}) and \cite[Exercise 2.1.27]{AGZ}, almost surely there exists $N_0 \in \mathbb N$ and $D\geq 0$ such that the spectral radii of the matrices $(\mathbf X_N, \mathbf Y_N)$ is bounded by $D$ for all $N\geq N_0$. Hence, there exists $M \geq 0$ such that almost surely one has
\begin{equation*}
		\sp \Big ( \ L(\mathbf X_N, \mathbf Y_N ) \ \Big ) \subset [ -M, M ].
\end{equation*}
Let $f: \mathbb R \mapsto \mathbb R$ non negative, compactly supported, vanishing on $\sp  ( \ L(\mathbf x, \mathbf y ) \  ) +(-\varepsilon/2, \varepsilon/2)$ and equal to one on  $[-M,M]\smallsetminus  \big (\sp  ( \ L(\mathbf x, \mathbf y )  ) +(-\varepsilon, \varepsilon) \ \big )$. Then almost surely for $N$ large enough, no eigenvalue of $L(\mathbf X_N, \mathbf Y_N )$ belongs to the complementary of  $\sp  ( \ L(\mathbf x, \mathbf y  ) \ ) +(-\varepsilon, \varepsilon)$, since otherwise
\begin{equation*}
		 (\tau_k\otimes \tau_N)\Big[f\big(  L({\mathbf X}_{N}, \mathbf Y_{N} ) \ \big) \Big ] \geq N^{-1} \geq N^{-1-\kappa} 
\end{equation*}
in contradiction with Lemma \ref{FromStoS}.
\end{proof}

\section{Proof of Corollaries 2.1, 2.2 and 2.4 }\label{PartEx}

\subsection{Proof of Corollary 2.1: diagonal matrices}\label{ProofCorDetDiag}

Let $\mathbf D_N=(D_1\toN\etc D_q\toN)$ be as in Corollary \ref{CorDetDiag}. For any $j=1\etc p$, the number of jump of $F^{-1}_j$ is countable. We show that the convergence of the norm (\ref{cor:diag2}) holds when we chose $v=(v_1\etc v_q)$ in $[0,1]^q$ such that for any $k\neq \ell$ in $\{1 \etc q\}$, the sets of jump points of $u \mapsto F^{-1}_k(u+v_k)$ and $u \mapsto F^{-1}_\ell(u+v_\ell)$ are disjoint. We show that for such a $v$, the family $\mathbf D^v_N$ satisfies the assumptions of Theorem \ref{MainTh}. In all this section, we always denote $\lambda_i$ instead of $\lambda_i\toN$ for any $i=1\etc N$.
\\
\\{\bf The convergence of traces, case $v=(0\etc 0)$:} Since the matrices commute, we only consider commutative polynomials. We start by showing that for all polynomials $P$,
\begin{equation}\label{proof:cor:1}
		\tau_N\Big [ P(\mathbf D_N) \ \Big ] \underset{N\rightarrow \infty}{\longrightarrow}  \int_0^1  P\Big( F_1^{-1}(u) \etc  F_q^{-1}(u) \ \Big)du.
\end{equation}
Denote by $\mu$ the probability distribution of the random variable $\big( F_1^{-1}(U ) \etc  F_q^{-1}(U ) \ \big)\in \mathbb R^q$, where $U$ is distributed according to the uniform distribution on $[0,1]$. In order to get (\ref{proof:cor:1}), we show that the sequence of measure in $\mathbb R^q$
		$$\Big ( \frac 1 N \sum_{i=1}^N \delta_{ \lambda_{i }(1)} \etc \frac 1 N \sum_{i=1}^N \delta_{ \lambda_{i }(q)} \Big ) $$
converges weakly to $\mu $. This sequence is tight, since there exists a $B>0$ such that for all $j=1\dots q$, for all $i=1\dots N$, one has $\lambda_i (j) \in [-B,B]$. Hence it is sufficient to show the following: for all real numbers $a_1 \etc a_q$, for all $\varepsilon>0$, there exists $\eta>0$ such that 
\begin{equation}\label{Cor1Purp}
		\underset{N\rightarrow \infty}{\mathrm{limsup}} \ \Big | \frac 1 N \sum_{i=1}^N  \mathbf 1_{]-\infty, a_1+\eta]}\big ( \lambda_{i } (1) \big) \times \dots \times  \mathbf 1_{]-\infty, a_q+\eta]}\big ( \lambda_{i } (q) \big) - \mu \big( \ ]-\infty, a_1] \times \dots \times ]-\infty, a_q]  \ \big) \ \Big | \leq \varepsilon.
\end{equation}
Fix $(a_1\etc a_q)$ in $\mathbb R^q$ and $\varepsilon >0$. Remark that one has 
		$$ \mu \big( \ ]-\infty, a_1] \times \dots \times ]-\infty, a_q]  \ \big) =   \underset{j=1\dots q}{ \min} F_j(a_j).$$
Let $j_0$ be an integer such that $F_{j_0}(a_{j_0})=\mu \big( \ ]-\infty, a_1] \times \dots \times ]-\infty, a_q]  \ \big)$. For any $j=1\etc q$, the empirical spectral distribution of $D_j\toN$ converges to $\mu_j$. Then for all $a$ in $\mathbb R$ point of continuity for $F_j$, one has
\begin{equation}\label{Cor1LGN}
	\frac 1 N \sum_{i=1}^N \mathbf 1_{]-\infty , a]}\big ( \lambda_i (j)Ê\big ) \underset{N\rightarrow \infty} \longrightarrow \mu_j\big ( \ ]-\infty, a ] \ \big ).
\end{equation}
Let $\eta>0$ such that
\begin{itemize}
	\item $\mu_{j_0}\big ( \ ] a_{j_0} , a_{j_0}+\eta] \ \big )<\varepsilon/2$.
	\item for all $j=1\etc q$, the real numbers $a_j+\eta$ and $a_{j_0}+\eta$ are points of continuity for $F_j$.
\end{itemize}
By (\ref{Cor1LGN}) with $a=a_j+\eta$, there exists $ N_0\geq 1$ such that for all $N\geq N_0$ and $j=1\etc q$, one has
		$$ F_{j}(a_j+\eta) - \varepsilon \leq \frac 1 N \mathrm{Card} \ \Big \{ i=1\dots N \ \Big | \ \lambda_i (j) \leq a_j+\eta \ \Big \}.$$
But $F_{j}(a_j+\eta)\geq F_{j}(a_j) \geq F_{j_0}(a_{j_0})$. Then we have 
		$$   N\Big (  F_{j_0}(a_{j_0}) - \varepsilon \Big ) \leq \mathrm{Card} \  \Big \{ i=1\dots N \ \Big | \ \lambda_i (j) \leq a_j+\eta \ \Big \} .$$
The $\lambda_i (j)$ are non decreasing, so we get
\begin{equation}\label{Cor1ND1}
		\forall j=1\dots q, \ \forall i\leq N\Big (  F_{j_0}(a_{j_0}) - \varepsilon \Big ) , \ \lambda_i (j) \leq a_j+\eta.
\end{equation}
On the other hand, by (\ref{Cor1LGN}) with $j=j_0$ and $a=a_{j_0}+\eta$, there exists $ N_0\geq 1$ such that, for all $N\geq N_0$, one has
		$$ \frac 1 N   \mathrm{Card} \  \Big \{ i=1\dots N \ \Big | \ \lambda_i (j_0) \leq a_{j_0}+\eta \ \Big \} \leq  F_{j_0}(a_{j_0}+\eta)+\varepsilon/2.$$
But $ F_{j_0}(a_{j_0}+\eta) \leq  F_{j_0}(a_{j_0}) +\varepsilon/2$, so that 
		$$ \mathrm{Card} \  \Big \{ i=1\dots N \ \Big | \ \lambda_i (j_0) \leq a_{j_0}+\eta \ \Big \}  \leq N \Big ( F_{j_0}(a_{j_0}) +\varepsilon \Big ).$$
The $\lambda_i (j_0)$ are non decreasing, then we get
\begin{equation}\label{Cor1ND2}
		\forall i\geq N\Big (  F_{j_0}(a_{j_0}) + \varepsilon \Big ) , \ \lambda_i (j_0) \geq a_{j_0}+\eta.
\end{equation}
By (\ref{Cor1ND1}) and (\ref{Cor1ND2}) we obtain: for all $N\geq N_0$
		$$ \Big |  \frac 1 N \sum_{i=1}^N  \mathbf 1_{]-\infty, a_1+\eta]}\big ( \lambda_i (1) \big) \times \dots \times  \mathbf 1_{]-\infty, a_q+\eta]}\big ( \lambda_i (q) \big)  - F_{j_0}(a_{j_0}+ \eta) \Big | \leq \varepsilon,$$
and then (\ref{Cor1Purp}) is satisfied. So the convergence (\ref{proof:cor:1}) holds when $v$ is zero.
\\
\\{\bf The convergence of traces, case $v$ in $[0,1]^q$:} To deduce the general case we shall need the following lemmas.

\begin{Lem}[Quantiles of real diagonal matrices with sorted entries]\label{lem:corDiag:0} Let $D_N= \ \mathrm{diag} \ (\lambda_1  \etc \lambda_N )$ be an $N \times N$ real diagonal matrix with non decreasing entries along its diagonal. Assume that the empirical eigenvalue distribution of $D_N$ converges weakly to a compactly supported probability measure $\mu$. Let $F$ denote the cumulative distribution function of $\mu$ and $F^{-1}$ its generalized inverse. Let $v$ in $(0,1)$ a point of continuity for $F^{-1}$ and $(i_N)_{N\geq 1}$ a sequence of integers, with $i_N$ in $\{1 \etc N\}$, such that $i_N/N$ tends to $v$. Then, one has
		$$\lambda_{i_N} \limN F^{-1}(v).$$
In particular, we have the convergence of the quantile of order $v$:
		$$\lambda_{1+\lfloor vN\rfloor}  \limN F^{-1}(v).$$
\end{Lem}
\begin{proof} Denote $w=F^{-1}(v)$. Let $\eta \geq 0$ be such that $w-\eta$ and $w+\eta$ and points of continuity for $F$. Then, one has
\begin{equation*}
	\frac 1 N \sum_{i=1}^N \mathbf 1_{]-\infty , w-\eta]}\big ( \lambda_i Ê\big ) \underset{N\rightarrow \infty} \longrightarrow \mu\big ( \ ]-\infty, w-\eta ] \ \big ) = F(w-\eta).
\end{equation*}
Then, the $\lambda_i $ being non decreasing, for any $\varepsilon>0$ there exists $N_0$ such that for any $N\geq N_0$, one has
\begin{equation}\label{eq:demLem}
		 \forall i\geq \Big (  F(w-\eta) + \varepsilon \Big )N , \ \ \lambda_i  \geq w-\eta.
\end{equation}
Since $v$ is a point of continuity for $F^{-1}$, we get that $F(w-\eta) < v$. We chose $\varepsilon < v - F(w-\eta)$. Then, we get $F(w-\eta)+\varepsilon <  v$. Hence, there exists $N_0$ such that, for any $N\geq N_0$, one has $i_N \geq \big (  F(w-\eta) + \varepsilon \big )N$ and so, by (\ref{eq:demLem}): for any $\eta>0$, there exists $N_0$ such that for all $N\geq N_0$, one has $w-\eta \leq \lambda_{i_N}$.
Hence, we get for all $\eta>0$,
		$$w- \eta \leq \liminf_{N\rightarrow \infty} \lambda_{i_N} .$$
With the same reasoning, we get that
		$$ \underset{N \rightarrow \infty }\limsup \  \lambda_{i_N}  \geq w+\eta,$$
and hence, letting $\eta$ go to zero, we obtain the expected result.
\end{proof}

\begin{Lem}[Truncation of real diagonal matrices with sorted entries]\label{lem:corDiag:1} Let $D_N= \ \mathrm{diag} \ (\lambda_1  \etc \lambda_N )$ an $N \times N$ real diagonal matrix with non decreasing entries along its diagonal. Assume that the empirical eigenvalue distribution of $D_N$ converges weakly to a compactly supported probability measure $\mu$. For any $v_1<v_2$ in $[0,1]$, we set 
		$$D_N^{(v_1,v_2)} = \mathrm{diag} \ (\lambda_{1+\lfloor v_1 N\rfloor}  \etc \lambda_{\lfloor v_2 N\rfloor} ).$$
Let $F$ denotes the cumulative distribution function of $\mu$ and $F^{-1}$ its generalized inverse. We set $w_1=F^{-1}(v_1)$, $w_2=F^{-1}(v_2)$, $a_1 = F(w_1)-v_1$ and $a_2 = v_2-F(w_2^-)$. Then, the empirical eigenvalue distribution of $D_N^{(v_1,v_2)}$ converges weakly the probability measure proportional to
					$$ a_1 \delta_{w_1} +  \mu\Big( \ \cdot \cap \ ]w_1, w_2[ \ \Big) + a_2 \delta_{w_2}.$$
\end{Lem}
\begin{proof} We only show the lemma for $v_2 =0$, the general case can be deduce by adapting the reasoning. We then use, for conciseness, the symbols $v,w$ and $a$ instead of $v_1,w_1$ and $a_1$ respectively. 
\\
\\If $F$ is not continuous in $w$ (i.e. if $\mu(w)\neq 0$) and $v\neq F(w)$, then for any $\alpha$ in $]0, (F(w)-v)/2[$, the map $F^{-1}$ is continuous in $v+\alpha$ and $F(w)-\alpha$. By Lemma \ref{lem:corDiag:0}, we get that
\begin{equation}\label{Eq1:croDiag:Lem2}
		\underset{N \rightarrow\infty} \lim \lambda_{1+\lfloor (v+\alpha) N \rfloor} = \underset{N \rightarrow\infty} \lim \lambda_{1+\lfloor (F(w)-\alpha) N \rfloor} = w.
\end{equation}
Hence, for any continuous function $f$, we get
\begin{equation}\label{Eq3:croDiag:Lem2}
		\frac 1 N \sum_{i=1+\lfloor (v+\alpha) N \rfloor}^{ 1+\lfloor (F(w)-\alpha) N \rfloor} f(\lambda_i) \limN (a-2\alpha)f(w).
\end{equation}
If $F$ is continuous in $w$, we take $\alpha=0$ in the following.
\\
\\We can always find $\beta>0$, arbitrary small, such that $F(w)+\beta$ is a point of continuity for $F^{-1}$. Remark that we then have
		$$w = F^{-1}\big( F(w) \big) < F^{-1}\big( F(w)+\beta \big).$$
By Lemma \ref{lem:corDiag:0}, we get 
\begin{equation}\label{Eq2:croDiag:Lem2}
		\lambda_{1+\lfloor (F(w)+\beta)N\rfloor} \limN F^{-1}\big( F(w)+\beta \big).
\end{equation}
Moreover, we can always find $\gamma$ in $]0, F^{-1}\big( F(w)+\beta \big)-w[$, arbitrary small, such that $w+\gamma$ is a point of continuity for $F$ and $F(w+\gamma) < F(w)+\beta$. Then, by (\ref{Eq2:croDiag:Lem2}), we get that, for $N$ large enough
		$$ \mathrm{Card}  \Big \{ \ i \geq 1+\lfloor (F(w) -\alpha)N \rfloor \ \Big | \ \lambda_i\leq w+\gamma \ \Big \} \leq \lfloor (F(w) +\beta)N \rfloor - \lfloor (F(w) -\alpha)N \rfloor.$$
Hence, for any continuous function $f$, we get that for $N$ large enough
\begin{eqnarray}
		\left | \frac 1 {N} \sum_{i = 1+\lfloor (F(w)-\alpha) N \rfloor}^N f(\lambda_i) - \int_{]\omega,+ \infty]} f(x) \textrm d \mu(x) \right | & \leq &  \left | \frac 1 {N} \sum_{i = 1}^N f(\lambda_i) \mathbf 1_{]w+\gamma, + \infty ]}(\lambda_i) - \int_{]\omega,+ \infty]} f(x) \textrm d \mu(x) \right | \nonumber \\
			& & \ \ + \ \|f\|_{\infty} \frac {    \lfloor (F(w) +\beta)N \rfloor - \lfloor (F(w) -\alpha)N \rfloor } N.\label{Eq4:croDiag:Lem2}
\end{eqnarray}
By (\ref{Eq3:croDiag:Lem2}) and (\ref{Eq4:croDiag:Lem2}), we obtain
		$$\limsup_{N\rightarrow \infty} \left | \frac 1 {N} \sum_{i = 1+\lfloor vN \rfloor}^N f(\lambda_i) - a f(w) - \int_{]\omega,+ \infty]} f(x) \textrm d \mu(x) \right | \leq \|f\|_\infty \Big ( 4\alpha + \beta+ \mu\big(]w,w+\gamma]\big)  \Big).$$
Letting $\alpha,\beta,\gamma$ go to zero, we get the result.
\end{proof}

\noindent Let $v$ in $[0,1]^q$. We now show that, for any polynomial $P$, one has
\begin{equation}
		\tau_N\Big [ P(\mathbf D^v_N) \ \Big ] \underset{N\rightarrow \infty}{\longrightarrow}  \int_0^1  P\Big( F_1^{-1}(u+v_1) \etc  F_q^{-1}(u+v_q) \ \Big)du.
\end{equation}
At the possible price of relabeling the matrices, we assume $v_1\geq \dots \geq v_q$ and set
\begin{eqnarray*}
		N_1 & =& N-\lfloor v_1N\rfloor,\\
		N_j & = & \lfloor v_{j-1}N\rfloor -  \lfloor v_{j}N\rfloor, \ \forall j=1\etc q.
\end{eqnarray*}
For any $j=1\etc q$, we decompose the matrices $D_j\toN(v_j)$ into
		$$D_j\toN(v_j) = \ \textrm{diag} \ (D_{j,1}\toN \etc D_{j,q}\toN),$$
where for any $i=1\etc q$, the matrix $D_{j,i}\toN$ is $N_i\times N_i$. We set for any $i=1\etc q$, the family $\mathbf D_N(i)=(D_{1,i}\toN \etc D_{q,i}\toN)$. For any $i,j=1\etc q$, we denote by $F_{i,j}$ the cumulative distribution function of the measure obtained in Lemma \ref{lem:corDiag:1} with $(D_N, \mu, v_1,v_2)$ replaced by $(D_j\toN, \mu_j, v_{i-1}, v_i)$. Then, for any polynomial $P$, one as
		$$\tau_N[P(\mathbf D^v_N)] = \sum_{i=1}^q \frac {N_i}N \tau_{N_i}[ P(\mathbf D_N(i))].$$
By Lemma \ref{lem:corDiag:1} and by the case $v=(0\etc 0)$, we deduce that 
		$$\tau_{N_i}[ P(\mathbf D_N(i))] \underset{N\rightarrow \infty}{\longrightarrow}  \frac 1 {v_{q-1} -v_q}\int_{v_q}^{v_{q-1}}  P\Big( F_{i,1}^{-1}(u+v_1) \etc  F_{i,q}^{-1}(u+v_q) \ \Big)du,$$
with the convention $v_0=1$. The merge of the different measures gives as expected
\begin{equation}
		\tau_N\Big [ P(\mathbf D_N^v) \ \Big ] \underset{N\rightarrow \infty}{\longrightarrow}  \int_0^1  P\Big( F_1^{-1}(u+v_1) \etc  F_q^{-1}(u+v_q) \ \Big)du.
\end{equation}

\noindent{\bf The convergence of norms:} Let $v=(v_1\etc v_q)$ in $[0,1]^q$ such that for any $k\neq \ell$ in $\{1 \etc q\}$, the sets of jump points of $u \mapsto F^{-1}_k(u+v_k)$ and $u \mapsto F^{-1}_\ell(u+v_\ell)$ are disjoint. We now show that, for all polynomials $P$, one has
		$$\| P( \mathbf D^v_N) \| \underset {N \rightarrow \infty} \longrightarrow \underset{ \mathrm{Supp} \ \mu^v}{\mathrm{Sup}} \ \big | P \big |,$$
where $\mu^v$ is the probability distribution of the random variable $\big( F_1^{-1}(U +v_1) \etc  F_q^{-1}(U +v_q) \ \big)\in \mathbb R^q$, where $U$ is distributed according to the uniform distribution on $[0,1]$. In view of the above, we have 
		$$\liminf \ \| P( \mathbf D^v_N) \| \geq   \underset{ \mathrm{Supp} \ \mu^v}{\mathrm{Sup}} \big | P \big |.$$
It is sufficient then to show that, for any $\eta>0$, there exists $N_0\geq N$ such that for all $i=1\etc N$, one has
\begin{equation}\label{Cor1in}
		 \Big ( \lambda_{i+\lfloor v_1N \rfloor } (1) \etc \lambda_{i+\lfloor v_qN \rfloor } (q) \ \Big ) \in \mathrm{Supp } \ \mu^v + (-\eta,\eta)^q.
\end{equation}
Indeed, by uniform continuity, for any polynomial $P$ and $\varepsilon >0$, there exists $\eta\geq 0$ such that,  for all $(x_1\etc x_q)$ in $ \mathrm{Supp} \ \mu^v +[-1,1]^q$ and $(y_1 \etc y_q)$ in $\mathbb R^q$, one has
		$$|y_j-x_j| < \eta \  \Rightarrow \ \Big | P(x_1\etc x_q) - P(y_1 \etc y_q) \Big | < \varepsilon$$
and hence: for all $\varepsilon>0$, there exist $\eta\geq 0$ and $N_0\geq 1$ such that for all $N\geq N_0$, for all $i=1\etc N$
		$$\underset{i=1\dots N} {\max} \Big | P\big ( \lambda_{i+\lfloor v_1N \rfloor } (1) \etc \lambda_{i+\lfloor v_qN \rfloor } (q) \ \big ) \ \Big | \leq \underset{ \mathrm{Supp} \ \mu^v + (-\eta,\eta)^q}{\max} \ \big | P \big | \leq  \underset{ \mathrm{Supp} \ \mu^v }{\max} |P| + \varepsilon.$$
Suppose that (\ref{Cor1in}) is not true: there exist $\eta>0$ and $(N_k)_{k\geq 1}$ an increasing sequence of positive integer such that for all $k\geq 1$, there exists $i_k$ such that
		$$ \Big ( \lambda_{i_k+\lfloor v_1N_k \rfloor }^{(N_k)}(1) \etc \lambda_{i_k+\lfloor v_qN_k \rfloor }^{(N_k)}(q) \ \Big ) \notin \mathrm{Supp } \ \mu^v + (-\eta,\eta)^q.$$
By compactness, one can always assume that $i_k/N_k$ converges to $u_0$ in $[0,1]$. For all $j$ in $\{1\etc q\}$ except a possible $j_0$, we have that $u_0+v_j$ is a point of continuity for $F^{-1}_j$ and so, by Lemma \ref{lem:corDiag:0}, $\lambda_{i_k+\lfloor v_jN_k \rfloor }^{(N_k)}(j)$ converges to $F^{-1}_j(u_0+v_j)$. Recall that 
		$$\mathrm{Supp} \ \mu^v = \Big \{ \big ( F_1^{-1}(u+v_1) \etc  F_q^{-1}(u+v_q) \ \big ) \ \Big | \ u\in [0,1] \ \Big \}.$$
Then we have, for $N$ large enough and for all $u$ in $[0,1]$, that $\big | \lambda_{i_k +  \lfloor v_{j_0}N_k \rfloor   }^{(N_k)}(j_0) -  F_{j_0}^{-1}(u+v_{j_0}) \big |>\eta$ i.e. 
		$$\mathrm{ dist} \ \big (  \lambda_{i_k+  \lfloor v_{j_0}N_k \rfloor }^{(N_k)}(j_0), \mathrm{Supp} \ \mu_{j_0} \big) >\eta,$$
which is in contradiction with the fact that for $N$ large enough the eigenvalues of $D_{j_0}\toN$ belong to a small neighborhood of the support of $\mu_{j_0}$.

\subsection{Proof of Corollary 2.2: Wishart matrices}
\noindent Let $r,s_1\etc s_p\geq 1$ and $(\mathbf W_N, \mathbf Y_N)$ be as in Corollary \ref{CorWisMa} and denote $s=s_1 + \hdots + s_p$. We use matrix manipulations in order to see the norm of a polynomial in the $rN\times rN$ matrices $\mathbf W_N, \mathbf Y_N, \mathbf Y_N^*$ as the norm of a polynomial in $(r+s)N\times (r+s)N$ matrices $\tilde{\mathbf X}_N, \tilde{\mathbf Y}_N, \tilde{\mathbf Y}_N^*, \tilde{\mathbf Z}_N$ and some elementary matrices, where $\tilde{\mathbf X}_N$ is a family of independent GUE matrices and $ \tilde{\mathbf Y}_N, \tilde {\mathbf Z}_N$ are modifications of  $\mathbf Y_N, \mathbf Z_N$. We will obtain the result as a consequence of Theorem \ref{MainTh}.
\\
\\Define the $(r+s)N \times (r+s)N$ matrices $\mathbf e_N = (e^{(N)}_0,e^{(N)}_1 \etc e_p\toN)$:
\begin{eqnarray}
		e_0^{(N)} & = & \left ( 
		\begin{array}{cc}
			\mathbf 1_{rN} &  \mathbf 0_{rN,sN} \\
			  \mathbf 0_{sN,rN}& \mathbf 0_{sN} 
		\end{array} \right ),\label{CorDemDef12}
\end{eqnarray}
\begin{eqnarray}
		e_j^{(N)} & = & \left ( \begin{array}{cccc}
			\mathbf 0_{rN} & &   & \\
			 & \mathbf 0_{(s_1+ \dots +s_{j-1})N} & &\\
			 & & \mathbf 1_{s_jN}&\\
			 & & & \mathbf 0_{(s_{j+1}+\dots +s_p)N}\end{array} \right ), \ \ j=1\etc p.
\end{eqnarray}
Recall that by definition of the Wishart matrix model for $j=1, \ldots ,  p$
\begin{equation}
	W_j^{(N)}=M_j^{(N)} Z_j\toN M_j^{(N)*},
\end{equation}
where $M_j^{(N)} $ is an $rN \times s_jN$ complex Gaussian matrix with independent identically distributed entries, centered and of variance $1/{rN}$. Let $\tilde{\mathbf X}_N=(\tilde X_1^{(N)} \etc \tilde X_p^{(N)})$ be a family of $p$ independent, normalized GUE matrices of size $(r+s)N \times (r+s)N$, independent of $\mathbf Y_N$ and $\mathbf Z_N$ and such that for $j=1, \ldots ,  p$, the $rN \times s_jN$ matrix $M_j\toN$ appears as a sub-matrix of $\sqrt{\frac {r+s}r} \tilde X_j^{(N)}$ in the following way: if we denote $\tilde M_j^{(N)}=\sqrt{\frac {r+s}r} e_0^{(N)}  \tilde X_j^{(N)} e_j^{(N)}$ then
\begin{eqnarray}
 \tilde M_j^{(N)} 	& = & \left ( \begin{array}{cccc}
			\mathbf 0_{rN} & &  M_j^{(N)}   & \\
			 & \mathbf 0_{(s_1+ \dots+s_{j-1})N} & &\\
			 & & \mathbf 0_{s_jN}&\\
			 & & & \mathbf 0_{(s_{j+1}+ \dots+s_p)N}\end{array} \right ).
\end{eqnarray}
Let $\tilde {\mathbf Y}_N = (\tilde Y_1^{(N)} \etc \tilde Y_q^{(N)})$ and $\tilde {\mathbf Z}_N= (\tilde Z_1^{(N)} \etc \tilde Z_p^{(N)})$ be the families of $(r+s)N \times (r+s)N$ matrices defined by:
\begin{equation}\label{DefDefWis}
	\tilde Y_j^{(N)} = \left ( \begin{array}{cc}
			Y_j^{(N)} & \mathbf 0_{rN,sN} \\
			\mathbf 0_{sN,rN} & \mathbf 0_{sN} \end{array} \right ), \ \ j=1, \ldots ,  q,
\end{equation}
\begin{equation}\label{DefDefDefWis}
	\tilde Z_j^{(N)}  = \left ( \begin{array}{cccc}
			\mathbf 0_{rN} & & & \\
			 & \mathbf 0_{(s_1+ \dots+s_{j-1})N} & &\\
			 & & Z_j\toN &\\
			 & & & \mathbf 0_{(s_{j+1}+ \dots+s_p)N}\end{array} \right ), \ \ j=1, \ldots ,  p.
\end{equation}
By assumption, with probability one the non commutative law of $\mathbf Y_N$ converges to the law of non commutative random variables $\mathbf y=(y_1\etc y_q)$ in a $\mathcal C^*$-probability space $(\mathcal A_0, .^*,\tau, \| \cdot \|)$ and for $j=1\dots p$ the non commutative law of $Z_j$ converges to the law of a non commutative random variable $z_j$ in a $\mathcal C^*$-probability space $(\mathcal A_j, .^*,\tau, \| \cdot \|)$ (we use the same notations for the functionals in the different spaces). All the traces under consideration are faithful. Let $\mathcal B$ denotes the product algebra $\mathcal B_0 \times \mathcal B_1 \times \dots \times \mathcal B_p$. We equip $\mathcal B$ with the involution $.^*$ and the trace $\tilde \tau$ defined by: for all $(b_0 \etc b_p)$ in $ \mathcal B $
		$$(b_0 \etc b_p)^* = (b^*_0 \etc b^*_p),$$
		$$\tilde \tau\big [ \ ( b_0 \etc b_p) \ \big ] =  \frac{r}{r+s}\tau(b_0) +  \frac{s_1}{r+s}\tau(b_1) + \dots +   \frac{s_p}{r+s}\tau(b_p).$$
The trace $\tilde \tau$ is a faithful tracial state on $\mathcal B$. Equipped with $.^*$, $\tilde \tau$ and with the norm $\| \cdot \|$ defined by (\ref{DefNorm}), the algebra $\mathcal B$ is a $\mathcal C^*$-probability space.
Define $\tilde {\mathbf y}=(\tilde y_1 \etc \tilde y_q)$, $\tilde {\mathbf z}=(\tilde z_1 \etc \tilde z_q)$ and ${\mathbf e}=(e_0 \etc e_p)$ by
		$$ \tilde y_j = (y_j, \mathbf 0_{\mathcal B_{1}} \etc \mathbf 0_{\mathcal B_{p}}), \ j=1\etc q,$$
		$$ \tilde z_j = (\mathbf 0_{\mathcal B_0} \etc  \mathbf 0_{\mathcal B_{j-1}} , z_j,\mathbf 0_{\mathcal B_{j+1}} \etc  \mathbf 0_{\mathcal B_{p}}  ), \ j=1\etc p,$$
		$$ e_j = (\mathbf 0_{\mathcal B_0} \etc  \mathbf 0_{\mathcal B_{j-1}} , \mathbf 1_{\mathcal B_j},\mathbf 0_{\mathcal B_{j+1}} \etc  \mathbf 0_{\mathcal B_{p}}  ), \ j=0\etc q.$$

\begin{Lem}\label{LemHypWis1} With probability one, the non commutative law of $(\tilde {\mathbf Y}_N, \tilde {\mathbf Z}_N, \mathbf e_N)$ in $(\mathrm{M}_{(r+s)N}(\mathbb C), .^*, \break \tau_{(r+s)N})$  converges to the law of $(\tilde {\mathbf y} , \tilde {\mathbf z} , \mathbf e )$ in $(\mathcal B, .^*, \tilde \tau)$.
\end{Lem}
\begin{proof} Let $P$ be a polynomial in $2p+2q+1$ non commutative indeterminates:
\begin{eqnarray}
	\lefteqn{\tau_{(r+s)N} \big [ P( \tilde { \mathbf Y}_N, \tilde {\mathbf Y}_N^*,Ê\mathbf Z_N, \mathbf e_N)\ \big ]}\nonumber \\
		& = & \frac r {r+s} \tau_{rN}\big [ P( \tilde { \mathbf Y}_N, \tilde {\mathbf Y}_N^*,\underbrace{Ê\mathbf 0_{rN} \etc \mathbf 0_{rN}}_{p}, \mathbf 1_{rN},\underbrace{ÊÊ\mathbf 0_{rN} \etc \mathbf 0_{rN}}_p) \ \big  ]\nonumber \\
					&  & \ \ + \ \ \sum_{j=1}^p \frac {s_j}{s+r} \tau_{s_j}\big [   P(\underbrace{Ê\mathbf 0_{s_jN} \etc \mathbf 0_{s_jN}}_{2q+j-1}, Z_j\toN ,\underbrace{ÊÊ\mathbf 0_{s_jN} \etc \mathbf 0_{s_jN}}_p, \mathbf 1_{s_jN},Ê\underbrace{Ê\mathbf 0_{s_jN} \etc \mathbf 0_{s_jN}}_{p-j})\ \big ] \nonumber \\
					& \underset{ N \rightarrow \infty} {\longrightarrow} & \frac r {r+s} \tau\big   [   P(   { \mathbf y},   {\mathbf y}^*,\underbrace{ÊÊ\mathbf 0  \etc \mathbf 0}_p , \mathbf 1 ,\underbrace{ÊÊ\mathbf 0  \etc \mathbf 0}_p )\ \big  ]\nonumber \\
					&  & \ \ + \ \ \sum_{j=1}^p \frac {s_j}{s+r} \tau\big  [   P(\underbrace{Ê\mathbf 0  \etc \mathbf 0}_{2q+j-1} , z_j  ,Ê\underbrace{Ê\mathbf 0  \etc \mathbf 0 }_p, \mathbf 1 ,Ê\underbrace{Ê\mathbf 0 \etc \mathbf 0}_{p-1} )\ \big ] \\
					& = & \tilde \tau [ P( \tilde {\mathbf y}, \tilde {\mathbf y}^* , \tilde {\mathbf z} , \mathbf e ) \ ], \label{LawAsymLemWis}
\end{eqnarray}
where the convergence holds almost surely since each term of the sum converges almost surely.
\end{proof}
\begin{Lem}\label{LemHypWis2} For all polynomials $P$ in $2p+2q+1$ non commutative indeterminates, almost surely
		$$\big \| P( \tilde { \mathbf Y}_N, \tilde {\mathbf Y}_N^*,Ê\mathbf Z_N, \mathbf e_N)\ \big \|  \underset{ N \rightarrow \infty} {\longrightarrow}  \| P( \tilde {\mathbf y}, \tilde {\mathbf y}^* , \tilde {\mathbf z} , \mathbf e ) \ \|.$$
\end{Lem}

\begin{proof} Lemma \ref{LemHypWis2} follows easily since for any polynomial $P$ in $2p+2q+1$ non commutative indeterminates, $\big \| P( \tilde { \mathbf Y}_N, \tilde {\mathbf Y}_N^*,Ê\mathbf Z_N, \mathbf e_N)\ \big \|$ is the maximum of the $p+1$ real numbers
\begin{itemize}
\item $\|P( \tilde { \mathbf Y}_N, \tilde {\mathbf Y}_N^*,\underbrace{Ê\mathbf 0_{rN} \etc \mathbf 0_{rN}}_{p}, \mathbf 1_{rN},\underbrace{ÊÊ\mathbf 0_{rN} \etc \mathbf 0_{rN}}_p)\|$,
\item $\| P(\underbrace{Ê\mathbf 0_{s_jN} \etc \mathbf 0_{s_jN}}_{2q+j-1}, Z_j\toN ,\underbrace{ÊÊ\mathbf 0_{s_jN} \etc \mathbf 0_{s_jN}}_p, \mathbf 1_{s_jN},Ê\underbrace{Ê\mathbf 0_{s_jN} \etc \mathbf 0_{s_jN}}_{p-j})\|$, $j=1\etc p$,
\end{itemize}
and $\|P( \tilde {\mathbf y}, \tilde {\mathbf y}^* , \tilde {\mathbf z} , \mathbf e )\|_{\tilde \tau}$ is the maximum of the $p+1$ real numbers
\begin{itemize}
\item $\|P(   { \mathbf y},   {\mathbf y}^*,\underbrace{ÊÊ\mathbf 0  \etc \mathbf 0}_p , \mathbf 1 ,\underbrace{ÊÊ\mathbf 0  \etc \mathbf 0}_p )\|$,
\item $\|  P(\underbrace{Ê\mathbf 0  \etc \mathbf 0}_{2q+j-1} , z_j  ,Ê\underbrace{Ê\mathbf 0  \etc \mathbf 0 }_p, \mathbf 1 ,Ê\underbrace{Ê\mathbf 0 \etc \mathbf 0}_{p-1} )\|$, $j=1\etc p$.
\end{itemize}
\end{proof}


\noindent Let $\tilde {\mathbf x}=(\tilde x_1\etc \tilde x_p)$ be a free semicircular system in $\mathcal C^*$-probability space. Let $\tilde{\mathcal A}$ be the reduced free product $\mathcal C^*$-algebra of $\mathcal B$ and the $\mathcal C^*$-algebra spanned by $\tilde {\mathbf x}$. We still denotes by $\tilde \tau$ the trace on $\tilde { \mathcal A}$ and the norm considered $\| \cdot \|$ is given by (\ref{DefNorm}) since the trace is faithful. By Voiculescu's theorem and by the independence of $\tilde {\mathbf X}_N$ and $( \tilde {\mathbf Y}_N,\tilde {\mathbf Z}_N)$, with probability one the non commutative law of $(\tilde {\mathbf X}_N, \tilde {\mathbf Y}_N,\tilde {\mathbf Z}_N, \mathbf e_N)$ in $($M$_{(r+s)N}(\mathbb C), .^*, \tau_{(r+s)N})$ converges to the non commutative law of $(\tilde {\mathbf x}, \tilde {\mathbf y},\tilde {\mathbf z}, \mathbf e)$ in $(\tilde {\mathcal A}, .^*, \tilde \tau)$. Define the non commutative random variables $\tilde {\mathbf m}=(\tilde m_1 \etc \tilde m_q)$ and $\tilde {\mathbf w}=(\tilde w_1 \etc \tilde w_q)$ in $\tilde {\mathcal A}$ by: for $j=1, \ldots ,  q$,
\begin{eqnarray}\label{DefLemCorWis1}
	\tilde m_j =\sqrt{\frac {r+s}r} e_0  \tilde x_j  e_j, \ \ \ \tilde w_j =e_0 ( \tilde m_j \ \tilde z_j+ \tilde m_j^*)^2.
\end{eqnarray}

\begin{Lem}\label{CorWishLem}
For any polynomial $P$ in $p+2q$ non commutative indeterminates, there exists a polynomial $\tilde P$ in $3p+2q+1$ non commutative indeterminates, such that one has
\begin{equation}\label{HatTildP}
		\left ( \begin{array}{cc}
			P(\mathbf W_N, \mathbf Y_N, \mathbf Y_N^*) & \mathbf 0_{rN,sN} \\
			\mathbf 0_{sN,rN} & \mathbf 0_{sN} \end{array} \right ) =  \tilde P (\tilde {\mathbf X}_N, \tilde {\mathbf Y}_N,  \tilde {\mathbf Y}_N^*, \tilde {\mathbf Z}_N, \mathbf e_N),
\end{equation}
		$$e_0 P(\tilde {\mathbf w}, \tilde {\mathbf y}, \tilde {\mathbf y}^*) =\tilde P( \tilde {\mathbf x}, \tilde {\mathbf y}, \tilde {\mathbf y}^*, \tilde {\mathbf z}, \mathbf e).$$
		
\end{Lem}

\begin{proof}We set $\tilde {\mathbf W}_N=(W_1\toN \etc W_p\toN)$ given by: for $j=1\etc p$,
\begin{eqnarray}\label{DefDefWisBis}
	\tilde W_j^{(N)}:= e_0^{(N)} ( \tilde M_j^{(N)}\tilde Z_j\toN+ \tilde M_j^{(N)*} )^2 =  \left ( \begin{array}{cc}
 			W_j^{(N)} &\mathbf 0_{rN,sN} \\
 			\mathbf 0_{sN,rN} & \mathbf 0_{sN} \end{array} \right ).
\end{eqnarray}
Let $P$ be a polynomial in $p+2q$ non commutative indeterminates. By the block decomposition of $\tilde { \mathbf W}_N$ and $  \tilde { \mathbf Y}_N$, one has
\begin{eqnarray*}
\left ( \begin{array}{cc}
			P(\mathbf W_N, \mathbf Y_N, \mathbf Y_N^*) & \mathbf 0_{rN,sN} \\
			\mathbf 0_{sN,rN} & \mathbf 0_{sN} \end{array} \right ) & = & e_0^{(N)} \  P( \tilde { \mathbf W}_N,  \tilde { \mathbf Y}_N, \tilde {  \mathbf Y}_N^*).
\end{eqnarray*}
Furthermore, By definitions of $  \tilde  {\mathbf  X}$ and $ \tilde  { \mathbf W}$: for $j=1\etc p$ 
\begin{eqnarray*}
	\tilde W_j\toN & =  & e_0^{(N)} ( \tilde  { M}_j\toN  \tilde { Z}_j\toN+ \tilde  {  M}_j\toNs )^2\\
			& = & e_0^{(N)} \frac {r+s}r (e_0^{(N)} \tilde  { X}_j\toN  e_j^{(N)} \tilde { Z}_j\toN+ e_j^{(N)}  \tilde  { X}_j\toN e_0^{(N)}  )^2.
\end{eqnarray*}
Define for $j=1\etc p$ the non commutative polynomial $P_j$ deduced by the formula
\begin{equation}\label{Cor2indet1}
		P_j( \tilde x_j, \tilde z_j, \mathbf e) = e_0 \frac {r+s}r (e_0 \tilde  { x}_j  e_j  \tilde { z}_j + e_j   \tilde  { x}_j  e_0   )^2,
\end{equation}
and define $\tilde P$ deduced by
\begin{equation}\label{Cor2indet2}
	\tilde P(  \tilde {\mathbf x}, \tilde {\mathbf y}, \tilde {\mathbf y}^*, \tilde {\mathbf z}, \mathbf e )  =  e_0 \  P\Big (P_1( \tilde x_1, \tilde z_1, \mathbf e) \etc P_p(\tilde x_p, \tilde z_p, \mathbf e) ,  \tilde { \mathbf y}, \tilde {  \mathbf y}^*\Big).
\end{equation}
The polynomials are defined without ambiguity if $\tilde {\mathbf x}, \tilde {\mathbf y}, \tilde {\mathbf y}^*, \tilde {\mathbf z}, \mathbf e $ are seen as families of non commutative indeterminates (without any algebraic relation) instead of non commutative random variables. Remark that, by definition, for all $j=1\etc p$ the non commutative random variable $w_j$ equals $P_j( \tilde x_j, \tilde z_j, \mathbf e)$. Hence it follows as expected that
\begin{eqnarray*}
		\left ( \begin{array}{cc}
			P(\mathbf W_N, \mathbf Y_N, \mathbf Y_N^*) & \mathbf 0_{rN,sN} \\
			\mathbf 0_{sN,rN} & \mathbf 0_{sN} \end{array} \right ) & = &  \tilde P (\tilde {\mathbf X}_N, \tilde {\mathbf Y}_N,  \tilde {\mathbf Y}_N^*, \tilde {\mathbf Z}_N, \mathbf e_N),\\
			e_0P(  \tilde {\mathbf w}, \tilde {\mathbf y}, \tilde {\mathbf y}^*) & = & \tilde P(   \tilde {\mathbf x}, \tilde {\mathbf y}, \tilde {\mathbf y}^*, \tilde {\mathbf z}, \mathbf e).
\end{eqnarray*}
\end{proof}

\noindent It is well known as a generalization of Voiculescu's theorem that, under Assumption 1 separately for $Z_1\toN, \etc Z_p\toN,\mathbf Y_N$ and by independence of the families, with probability one the non commutative law of $(\mathbf W_N, \mathbf Y_N)$ in $($M$_N(\mathbb C), .^*, \tau_N)$ converges to the non commutative law of $(\mathbf w, \mathbf y)$ in a $\mathcal C^*$-probability space $(\mathcal A, .^*, \tau, \| \cdot \|)$ with faithful trace, where
\begin{enumerate}
\item $\mathbf w = (w_1 \etc w_p)$ are free selfadjoint non commutative random variables,
\item $\mathbf y=(y_1 \etc y_q)$ is the limit in law of $\mathbf Y_N$,
\item $\mathbf w$ and $\mathbf y$ are free.
\end{enumerate}

\noindent For any polynomial $P$ in $p+2q$ non commutative indeterminates
\begin{eqnarray*}
	\tau[P(  {\mathbf w},   {\mathbf y},  {\mathbf y}^*)] & = & \underset{N\rightarrow \infty}\lim	\tau_{rN} \big [ P( \mathbf W_N, \mathbf Y_N, \mathbf Y_N^*) \ \big ]\\
	& = & \underset{N\rightarrow \infty}\lim	\frac{r+s}r	\tau_{(r+s)N} \left [\left ( \begin{array}{cc}
			P(\mathbf W_N, \mathbf Y_N, \mathbf Y_N^*) & \mathbf 0_{rN,sN} \\
			\mathbf 0_{sN,rN} & \mathbf 0_{sN} \end{array} \right )  \ \right ]\\
	& = &  \underset{N\rightarrow \infty}\lim	\frac{r+s}r	\tau_{(r+s)N} \big [  \tilde P (\tilde {\mathbf X}_N, \tilde {\mathbf Y}_N,  \tilde {\mathbf Y}_N^*, \tilde {\mathbf Z}_N, \mathbf e_N) \ \big ] \\
	& = & \frac{r+s}r \tilde \tau \big [ 	 \tilde P(   \tilde {\mathbf x}, \tilde {\mathbf y}, \tilde {\mathbf y}^*, \tilde {\mathbf z}, \mathbf e) \ \big ]\\
	& = &  \frac{r+s}r \tilde \tau \big [ e_0P(  \tilde {\mathbf w}, \tilde {\mathbf y}, \tilde {\mathbf y}^*) \ \big ],
\end{eqnarray*}
where the limits are almost sure. In particular we obtain that, for all polynomials $P$ in $p+2q$ non commutative indeterminates, one has
\begin{equation}
		\| e_0P(\tilde {\mathbf w}, \tilde {\mathbf y}, \tilde {\mathbf y}^*) \| =\| P(  {\mathbf w},   {\mathbf y},  {\mathbf y}^*)\| .
\end{equation}

\noindent By Lemmas \ref{LemHypWis1} and \ref{LemHypWis2}, the family of $(r+s)N\times (r+s)N$ matrices $( \tilde {\mathbf Y}_N, \tilde {\mathbf Z}_N, \mathbf e_N)$ satisfies the assumptions of Theorem \ref{MainTh}, hence for all polynomials $P$ in $3p+2q+1$ non commutative indeterminates, with $\tilde P$ as in Lemma \ref{CorWishLem}, almost surely one has
\begin{equation}\label{CorDemEnd}
		\| \tilde P (\tilde {\mathbf X}_N, \tilde {\mathbf Y}_N,  \tilde {\mathbf Y}_N^*, \tilde {\mathbf Z}_N, \mathbf e_N)\| \underset{N\rightarrow \infty}{\longrightarrow}  \| \tilde P (\tilde {\mathbf x}, \tilde {\mathbf y},  \tilde {\mathbf y}^*, \tilde {\mathbf z},  \mathbf e)\| .
\end{equation}
Remark that 
	$$\| P(\mathbf W_N, \mathbf Y_N, \mathbf Y_N^*)  \| = \left \| \ \left ( \begin{array}{cc}
			P(\mathbf W_N, \mathbf Y_N, \mathbf Y_N^*) & \mathbf 0_{rN,sN} \\
			\mathbf 0_{sN,rN} & \mathbf 0_{sN} \end{array} \right ) \ \right \| = \| \tilde P (\tilde {\mathbf X}_N, \tilde {\mathbf Y}_N,  \tilde {\mathbf Y}_N^*, \tilde {\mathbf Z}_N, \mathbf e_N)\|, $$
	$$\| \tilde P (\tilde {\mathbf x} , \tilde {\mathbf y} ,  \tilde {\mathbf y} ^*, \tilde {\mathbf z} , \mathbf e )\| = \| e_0P(\tilde {\mathbf w}, \tilde {\mathbf y}, \tilde {\mathbf y}^*) \| =\| P(  {\mathbf w},   {\mathbf y},  {\mathbf y}^*)\| .$$
Together with (\ref{CorDemEnd}), this gives the expected result.

\subsection{Proof of Corollary 2.4: Rectangular band matrices}\label{ProofCorMIMO}
\noindent We only give a sketch of the proof. Details are obtained by minor modification of the proofs of Corollaries \ref{CorWisMa} and \ref{CorBloMa}. Let $H$ be as in Corollary \ref{CorMIMO}:
\begin{equation}
		H= \left( \begin{array}{ccccccccc}
		A_1 & A_2 & \hdots & A_L & \mathbf 0 & \hdots & &\hdots & \mathbf 0 \\
		\mathbf 0 & A_1 & A_1 & \hdots & A_L & \mathbf 0 & & & \vdots\\
		\vdots & \mathbf 0 & A_1 & A_2 & \hdots & A_L & \mathbf 0 && \\
		 &  & \ddots & \ddots & \ddots & & \ddots &  \vdots &\vdots\\
		\vdots &	  & & \ddots & \ddots & \ddots & &\ddots & 		\mathbf 0										\\
		\mathbf 0 &  \hdots &  & \hdots &\mathbf 0 & A_1 & A_2 &\hdots  & A_L \end{array} \right).
\end{equation}
We start with the following observation: the operator norm of $H$ is the square root of the operator norm of $H^*H$, which is a square block matrix. Its blocks consist of sums of $t N \times t N$ matrices of the form $A_l^* A_m$, $l,m=1 \dots L$. By minor modifications of the proof of Corollary \ref{CorWisMa}, we get the almost sure convergence of the normalized trace and of the norm for any polynomial in the matrices $\mathbf A_N=(A_l^*A_m)_{l,m=1..L}$ as $N$ goes to the infinity. By Proposition \ref{PropBloMa}, we get that the convergences hold for square block matrices and in particular for any polynomial in $H^*H$. Hence the result follows by functional calculus.

\appendix
\section[A theorem about norm convergence, by D. Shlyakhtenko]{A theorem about norm convergence, by D. Shlyakhtenko\footnote{Research supported by NSF grant DMS-0900776}}\label{App1}

\textbf{Lemma } Let $(A,\tau)$ be a $C^{*}$-algebra with a faithful
trace $\tau$, and consider $B$ to be the universal $C^{*}$-algebra
generated by $A$ and elements $L^{(1)},\dots,L^{(n)}$ satisfying
$L^{(i)*}xL^{(j)}=\delta_{i=j}\tau(x)$ for all $x\in A$. Moreover,
consider the linear functional $\psi$ determined on $*-\textrm{Alg}(A,\{L^{(j)}\}_{j})$
by:

$\psi|_{A}=\tau$, 

$\psi(x_{0}L^{(i_{1})}x_{1}\cdots x_{k-1}L^{(i_{k})}x_{k}y_{0}L^{(j_{1})*}y_{1}\cdots y_{l-1}L^{(j_{l})*}y_{l})=0$
whenever $x_{1},\dots,x_{k},y_{0},\dots,y_{l}\in A$ and at least
one of $k$ and $l$ is nonzero.

Then $\psi$ extends to a state on $B$ having a faithful GNS representation.
Moreover, $(B,\psi)\cong(A,\tau)*(\mathcal{E},\phi)$ where $(\mathcal{E},\phi)$
is the $C^{*}$-algebra generated by $n$ free creation operators
$\ell_{1},\dots,\ell_{n}$ on the full Fock space $\mathcal{F}(\mathbb{C}^{n})$
and $\phi$ is the vacuum expectation.

\begin{proof}[Sketch of proof.] Consider the $A,A$-Hilbert
bimodule $\mathcal{H}=L^{2}(A,\tau)\otimes A$ with the inner product\[
\langle\xi\otimes a,\xi'\otimes a'\rangle_{A}=\langle\xi,\xi'\rangle_{L^{2}(\tau)}a^{*}a'\]
and the left and right $A$ actions given by\[
x\cdot(\xi\otimes a)\cdot y=x\xi\otimes ay.\]
Let $B$ be the extended Cuntz-Pimsner algebra associated to $\mathcal{H}^{\oplus n}$
(see \cite{PIM}), i.e. the universal $C^{*}$-algebra generated
by $A$ and operators $L_{h}:h\in\mathcal{H}$ satisfying the relations\begin{eqnarray*}
L_{h}^{*}L_{g} & = & \langle h,g\rangle_{A},\qquad h,g\in\mathcal{H}^{\oplus n}\\
aL_{h}b & = & L_{ahb},\qquad h\in\mathcal{H}^{\oplus n},\ a,b\in A.\end{eqnarray*}

It follows from the results of \cite{SHL} that if we denote by $(\hat{B},\hat{\psi})$
the free product $(A,\tau)*(\mathcal{E},\phi)$, then:\begin{eqnarray*}
\ell_{i} ^*x\ell_{j} & = & \delta_{i=j}\tau(x),\qquad\forall x\in A,\\
\hat{\psi}(x_{0}\ell_{i_{1}}x_{1}\cdots x_{k-1}\ell_{i_{k}}x_{k}y_{0}\ell_{j_{1}}^{*}y_{1}\cdots y_{l-1}\ell_{j_{l}}^{*}y_{l}) & = & 0,\qquad\forall x_{1},\dots,x_{k},y_{1},\dots,y_{l}\in A,\ k+l>0\end{eqnarray*}
If $h=(\sum_{i}\xi_{i}^{(k)}\otimes a_{i}^{(k)})_{k=1}^{n}\in(A\otimes A)^{\oplus n}\subset\mathcal{H}^{\oplus n}$
is a finite tensor, write\[
\ell_{h}=\sum_{k,i}\xi_{i}^{(k)}\ell_{k}a_{i}^{(k)}.\]
It then follows that\begin{eqnarray*}
\ell_{h}^{*}\ell_{g} & = & \langle h,g\rangle_{A},\qquad h,g\in\mathcal{H}^{\oplus n}\\
a\ell_{h}b & = & \ell_{ahb},\qquad a,b\in A,\ h\in\mathcal{H}^{\oplus n}\end{eqnarray*}
which in particular means that $\Vert\ell_{h}\Vert_{2}^{2}=\Vert\ell_{h}^{*}\ell_{h}\Vert=\Vert h\Vert^{2}$
so that the mapping $h\mapsto\ell_{h}$ is an isometry. We then extend
$\ell$ to a map from $\mathcal{H}^{\oplus n}$ into $\hat{B}$. Note
that the extension of $\ell$ still satisfies $a\ell_{h}b=\ell_{ahb}$
whenever $a,b\in A$ and $h\in\mathcal{H}^{\oplus n}$.

From this we see that (by the universal property of $B$) there exists
a $*$-homomorphism $\pi:B\to\hat{B}$, so that $\psi=\hat{\psi}\circ\pi$.
Thus all we need to prove is that $\pi$ is injective. But by \cite[Prop. 3.3]{PIM}, it follows that $B$ is isomorphic
to the Toeplitz algebra $\mathcal{T}$ (since in this case obviously
$\langle\mathcal{H}^{\oplus n},\mathcal{H}^{\oplus n}\rangle_{A}=A$)
acting on the Fock space $\mathcal{F}=\bigoplus_{k\geq0}(\mathcal{H}^{\oplus n})^{\otimes_{A}k}$.
If we denote by $E$ the canonical conditional expectation from \emph{$\mathcal{T}$
}onto $A$ and consider the state $\theta=\tau\circ E$, then the
resulting Hilbert space is the closure of $\mathcal{F}$ in the (faithful)
norm $\Vert\xi\Vert=\tau(\langle\xi,\xi\rangle_{A})^{1/2}$; from
this we see that the GNS representation of $B$ associated to the
state $\theta$ on $B$ is faithful. Since $\hat{B}$ is exactly this
GNS representation, it follows that $\pi$ is injective. 
\end{proof}

If $A_{N}$ is a sequence of $C^{*}$-algebras and $\omega \in\beta\mathbb{N}\setminus\mathbb{N}$
is a free ultrafilter, we shall denote by \[
\mathfrak{A}=\prod^{\omega  }A_{N}\]
the quotient\[
\prod^{\omega  }A_{N}=\left(\prod_{N=1}^{\infty}A_{N}\right)/\left\{ (a_{j})_{{N}=1}^{\infty}:\lim_{{N}\to\omega  }\Vert a_{{N}}\Vert=0\right\} .\]
Then $\mathfrak{A}$ is a $C^{*}$-algebra. 

Let now $X_{N}^{(j)}$, $j=1,\dots,n$, $N=1,2,\dots$ be self-adjoint
random variables and assume that $X^{(j)}$, $j=1,\dots,n$ are such
that for any non-commutative polynomial $P$,\begin{eqnarray*}
\tau{_N}(P(X_{N}^{(1)},\dots,X_{N}^{(n)})) & \to & \tau(P(X^{(1)},\dots,X^{(n)}))\\
\Vert P(X_{N}^{(1)},\dots,X_{N}^{(n)})\Vert & \to & \Vert P(X^{(1)},\dots,X^{(n)})\Vert.\end{eqnarray*}
Let $L^{(j)}$, $j=1,\dots,n$ be a family of free creation operators,
free from each other and from $\{X_{N}^{(j)}\}_{N,j}\cup\{X^{(j)}\}_{j}$.
In other words, they satisfy:\[
L^{(j)*}xL^{(j)}=\tau(x),\qquad\forall x\in C^{*}(\{X_{N}^{(j)}\}_{N,j}\cup\{X^{(j)}\}_{j})\]

We use the notations\begin{eqnarray*}
A_{N} & = & C^{*}(X_{N}^{(1)},\dots,X_{N}^{(n)}),\qquad B_{N}=C^{*}(X_{N}^{(1)},\dots,X_{N}^{(n)},L^{(1)},\dots,L^{(n)})\\
A & = & C^{*}(X^{(1)},\dots,X^{(n)}),\qquad B=C^{*}(X^{(1)},\dots,X^{(n)},L^{(1)},\dots,L^{(n)})\end{eqnarray*}
and we denote by $\tau_{N}$ and $\psi_{N}$ the respective states
on $A_{N}$ and $B_{N}$ ($\cong(A_{N},\tau_{N})*(\mathcal{E},\phi)$).
We denote by $\tau$ and $\psi$ the respective states on $A$ and
$B$ ($\cong(A,\tau)*(\mathcal{E},\phi)$). 

Consider now the ultrapowers\[
\mathfrak{A}=\prod^{\omega  }A_{N}\subset\mathfrak{B}=\prod^{\omega  }B_{N}.\]
The formula\[
\psi:(x_{N})_{N=1}^{\infty}\mapsto\lim_{N\to\omega  }\psi_{N}(x_{N})\]
defines a state on $\mathfrak{B}$. 

We shall denote by $\hat{X}^{(j)}\in\mathfrak{A}$ the sequence $(X_{N}^{(j)})_{j=1}^{N}$.
Then by assumption, we have that the map $\alpha$ taking $X^{(j)}$
to $\hat{X}^{(j)}$ extends to a state-preserving isomorphism from
$(A,\tau)$ into $\mathcal{B}$ with range $\hat{A}=C^{*}(\hat{X}^{(1)},\dots,\hat{X}^{(n)})$. 

We shall also denote by $\hat{L}^{(j)}$ the constant sequence $(L^{(j)})_{N=1}^{\infty}\in\mathfrak{B}$.
Then for any element of $\hat{A}$ represented by the sequence $x=(x_{N})_{N=1}^{\infty}$
we have:\[
\hat{L}^{(j)*}x\hat{L}^{(i)}=\delta_{i=j}(\tau_{N}(x_{N}))_{N=1}^{\infty}\]
which (since the $L^{2}$ and operator norms coincide on multiples
of identity) is equal to $\tau(x)1\delta_{i=j}\in\mathfrak{A}$. It follows from
the universality property that\[
\hat{B}\stackrel{\textrm{def}}{=}C^{*}(\hat{X}^{(1)},\dots,\hat{X}^{(n)},\hat{L}^{(1)},\dots,\hat{L}^{(n)})\]
is a quotient of $(A,\tau)*(\mathcal{E},\phi)$, the quotient map
$\beta$ determined by the fact that it is $\alpha$ {on} $A$ and takes
$\ell_{j}$ to $\hat{L}^{(j)}$. On the other hand, if we consider
the GNS-representation $\pi$ of $ $$\hat{B}$ with respect to the
restriction of $\psi$, we easily get (by freeness from $\hat{A}$
and $\{\hat{L}^{(j)}\}_{j})$ that the image is isomorphic to $(A,\tau)*(\mathcal{E},\phi)$.
Thus $\pi\circ\beta=\textrm{id}$ so that actually\[
\beta:(A,\tau)*(\mathcal{E},\phi)\to\hat{B}=C^{*}(\hat{X}^{(1)},\dots,\hat{X}^{(n)},\hat{L}^{(1)},\dots,\hat{L}^{(n)})\]
is an isomorphism.

Consider now a non-commutative $*$-polynomial $P$. Then\begin{eqnarray*}
\Vert P(X^{(1)},\dots,X^{(n)},\ell^{(1)},\dots\ell^{(n)})\Vert_{(A,\tau)*(\mathcal{E},\phi)} & = & \Vert P(\hat{X}^{(1)},\dots,\hat{X}^{(n)},\hat{L}^{(1)},\dots,\hat{L}^{(n)})\Vert_{\mathfrak{B}}\\
 & = & \lim_{N\to\omega  }\Vert P(X_{N}^{(1)},\dots,X_{N}^{(n)},L^{(1)},\dots,L^{(n)})\Vert_{B_{N}}.\end{eqnarray*}
Since the left hand side does not depend on $\omega  $, we have proved:

\begin{Th}\label{ShlTh}Let $X_{N}^{(j)}\in(A_{N},\tau_{N})$, $j=1,\dots,n$,
$N=1,2,\dots$ be self-adjoint random variables and assume that $X^{(j)}\in(A,\tau)$,
$j=1,\dots,n$ are such that for any non-commutative polynomial $P$,\begin{eqnarray*}
\tau(P(X_{N}^{(1)},\dots,X_{N}^{(n)})) & \to & \tau(P(X^{(1)},\dots,X^{(n)}))\\
\Vert P(X_{N}^{(1)},\dots,X_{N}^{(n)})\Vert_{A_{N}} & \to & \Vert P(X^{(1)},\dots,X^{(n)})\Vert_{A}.\end{eqnarray*}
Let $(\ell_{1},\dots,\ell_{n})\in\mathcal{E}$ be free creation operators,
and let $B_{N}=(\mathcal{E},\phi)*(A_{N},\tau_{N})$, $B=(\mathcal{E},\phi)*(A,\tau)$.
Assume that the traces $\tau_{j}$ are faithful. Then for any non-commutative
$*$-polynomial $Q$,\[
\Vert Q(X_{N}^{(1)},\dots,X_{N}^{(n)},\ell_{1},\dots,\ell_{n})\Vert_{B_{N}}\to\Vert Q(X^{(1)},\dots,X^{(n)},\ell_{1},\dots,\ell_{n})\Vert_{B}.\]
\end{Th}

It should be noted that if $S_{1},\dots,S_{n}$ are free semicircular
variables, free from $\{X_{N}^{(j)}\}_{N,j}\cup\{X^{(j)}\}_{j}$,
then $C_{N}=C^{*}(X_{N}^{(1)},\dots,X_{N}^{(n)},S_{1},\dots,S_{n})$
is isometrically contained in $B_{N}$, while $C=C^{*}(X^{(1)},\dots,X^{(n)},S_{1},\dots,S_{n})$
is isometrically contained in $B$. Thus the analog of Theorem A with
$\ell_{j}$'s replaced by a free semicircular family also holds.

\bibliographystyle{plain} 
\bibliography{biblio}

\end{document}